\documentclass{article}

\usepackage[utf8]{inputenc} 
\usepackage[T1]{fontenc}    
\usepackage[colorlinks=true,linkcolor=blue, urlcolor=blue, citecolor=blue]{hyperref}   
\usepackage{url}            
\usepackage{booktabs}       
\usepackage{amsfonts}       
\usepackage{amsmath, amsthm, amssymb}
\usepackage{nicefrac}       
\usepackage{microtype}      
\usepackage{xcolor}         
\usepackage[margin=1in]{geometry}
\newcommand{\D}{\textrm{d}}
\newcommand{\R}{\mathbb{R}}

\newcommand{\A}{\mathbf{A}}
\newcommand{\bG}{\mathbf{G}}
\newcommand{\bbP}{\mathbb{P}}
\newcommand{\bP}{\mathbf{P}}

\newcommand{\cA}{\mathcal{A}}

\newcommand{\cE}{\mathcal{E}}

\newcommand{\cR}{\mathcal{R}}

\newcommand{\cV}{\mathcal{V}}

\newcommand{\EE}{\mathbb{E}}

\newcommand{\var}{\textrm{var}}

\newcommand{\one}{\mathbf{1}}

\newcommand{\eps}{\varepsilon} 
\newcommand{\Ct}{C_{*}}

\newcommand{\lin}{\text{deg}}

\usepackage{enumerate}
\newcommand{\chisqlin}{\text{trunc-deg}}
\newcommand{\chisqmax}{\text{max-trunc-deg}}

\usepackage{enumerate}
\newcommand\numberthis{\addtocounter{equation}{1}\tag{\theequation}}

\newtheorem{assumption}{Assumption}
\newtheorem{theorem}{Theorem}
\newtheorem{lemma}{Lemma}

\newtheorem{corollary}{Corollary}

\newtheorem{proposition}{Proposition}

\usepackage{lmodern}

\usepackage[style=numeric, url=false]{biblatex}
\title{Optimal community detection in dense bipartite graphs}

\author{%
  Julien Chhor and Parker Knight\footnote{Direct correspondence to \texttt{pknight@g.harvard.edu}. The authors contributed equally to this work.} \\
  Toulouse School of Economics and Harvard University
}

\begin{document}

\maketitle

\begin{abstract}
  We consider the problem of detecting a community of densely connected vertices in a high-dimensional bipartite graph of size $n_1 \times n_2$.
  Under the null hypothesis, the observed graph is drawn from a bipartite Erd\H{o}s-Renyi distribution with connection probability $p_0$. 
  Under the alternative hypothesis, there exists an unknown bipartite subgraph of size $k_1 \times k_2$ in which edges appear with probability $p_1 = p_0 + \delta$ for some $\delta > 0$, while all other edges outside the subgraph appear with probability $p_0$. 
  Specifically, we provide non-asymptotic upper and lower bounds on the smallest signal strength $\delta^*$ that is both necessary and sufficient to ensure the existence of a test with small enough type one and type two errors. 
  We also derive novel minimax-optimal tests achieving these fundamental limits when the underlying graph is sufficiently dense. 
  Our proposed tests involve a combination of hard-thresholded nonlinear statistics of the adjacency matrix, the analysis of which may be of independent interest.
  In contrast with previous work, our non-asymptotic upper and lower bounds match for any configuration of $n_1,n_2, k_1,k_2$. 
  
\end{abstract}

\section{Introduction}\label{sec:intro}

The analysis of community structure is a fundamental task in statistical network science \cite{fortunato202220, dey2022community}. 
In the 2000's, a series of striking works including (but not limited to) \cite{albert2002statistical,barabasi2004network,girvan2002community, palla2005uncovering, rosvall2008maps} observe the ubiquity of community-based organization in real world network data. 
Inspired by these observations, a vast literature has emerged over the past two decades attempting to understand the statistical limits of detecting and recovering planted communities in random graphs \cite{coja2010graph, massoulie2014community, mossel2014belief, verzelen2015community, arias2014community, banks2016information, mossel2018proof},  detecting geometry in graphs defined via latent metric spaces \cite{bubeck2016testing, brennan2020phase, liu2023phase, duchemin2023random}, and establishing fundamental computational limits for the community detection and recovery problems \cite{hajek2015computational, chen2016statistical, wu2018statistical, ma2015computational, brennan2018reducibility}. 
Here, we draw a careful distinction between \textit{community detection} and \textit{community recovery}. 
In the community detection problem, also sometimes referred to as the Planted Dense Subgraph (PDS) problem, the statistician observes a random graph and aims to determine whether the graph was drawn from an Erd\H{o}s-Renyi family, or if it contains a densely connected community of vertices (i.e., a dense subgraph). In contrast, the community recovery problem involves assigning each vertex in the observed graph to a latent community, under the assumption that such community structure exists. We focus on community detection in the present work, and we refer readers to \cite{abbe2018community} for a comprehensive review of the results pertaining to the recovery problem in the stochastic block model.

Existing work studying community detection can be divided into two categories. The first camp, including \cite{arias2014community, verzelen2015community, rotenberg2024planted, yuan2022sharp}, focuses on detecting the presence of a small community of $k$ nodes in a high-dimensional graph of $n$ nodes, typically in an asymptotic setting with $k = o(n)$. In these papers, the authors focus on deriving sufficient and necessary conditions on the difference in connection probabilities between the planted subgraph and the rest of the vertices for which detection of the subgraph is possible. While these works are able to give asymptotically precise characterizations of the detection boundary for their problems of interest, their results are limited in generality, as they constrain themselves to fully symmetric settings (the adjacency matrices of the full graph and the subgraph, with the exception of \cite{rotenberg2024planted}, are square) and make asymptotic assumptions that prohibits the appearance of complex phenomena; see for instance equation (6) in \cite{arias2014community}.

A parallel line of work considers the community detection problem in stochastic block models (SBM's) \cite{janson2010asymptotic, massoulie2014community, mossel2015reconstruction, abbe2015community, banerjee2018contiguity, banerjee2017optimal, banks2016information, gao2017testing}. Here the focus is on the statistician's ability detect the presence of two or more communities in the observed graph, and deriving conditions on the difference in expected degree between the communities under which detection is possible. Early works \cite{abbe2015community, mossel2015reconstruction} give fundamental results related to the famed Kesten-Stigum threshold, although they are confined to the restricted case of two equally sized communities. While recent papers such as \cite{banks2016information, mossel2024weak} allow for a growing number of potentially asymmetric communities, these works still required the average expected degree across communities to be held constant, which prevents their results from applying to graphs with highly imbalanced communities. The restriction is lifted in \cite{gao2018community}, in which the authors study community detection in the degree-corrected block model. However, similar to the works of \cite{arias2014community, verzelen2015community, rotenberg2024planted}, the results of \cite{gao2018community} hold only in a particular asymptotic regime.

Given the limitations of existing work, there remains a great interest in studying community detection in random graph models beyond those covered by the literature. In particular, previous results do not apply to the family of \textit{bipartite graphs}, which are defined as sets of edges between two disjoint sets of vertices. In practice, the two sets of vertices may be of very different size, meaning that the results tailored for square adjacency matrices \cite{arias2014community, verzelen2015community, abbe2015community, mossel2015reconstruction} do not capture the subtleties of community detection in bipartite graphs. Furthermore, any planted subgraph in a bipartite graph is, by definition, itself bipartite. While the recent paper \cite{rotenberg2024planted} considers the detection of bipartite communities, they assume that these communities are planted in a larger non-bipartite graph and hence, similarly to \cite{arias2014community, verzelen2015community, abbe2015community, mossel2015reconstruction}, fail to fully characterize the detection boundary in the bipartite case. We also comment that bipartite graphs can be thought of as a special case of a multi-community SBM; however, as outlined above, the prior work studying this problem introduces its own limitations \cite{banks2016information, mossel2024weak, gao2018community}. As bipartite graphs are increasingly used as modeling tools in the diverse fields of biology and medicine \cite{pavlopoulos2018bipartite}, information science \cite{maier2022bipartite}, and game theory \cite{pena2012bipartite}, this lack of understanding in the bipartite setting constitutes a significant gap in the statistics and machine learning literature.

In this paper, we aim to address this gap. Following \cite{arias2014community, verzelen2015community}, we formalize the problem of detecting a community in a random bipartite graph as a hypothesis testing problem. We observe a random bipartite graph $\bG$ defined on disjoint sets of vertices (nodes) $\cV_1$ and $\cV_2$ of size $n_1$ and $n_2$ respectively. Under the null hypothesis, an edge appears between any two nodes in $\bG$ with equal probability $p_0 > 0$. Under the alternative hypothesis, there exist subsets of nodes $\cV_{K_1} \subset \cV_1$ and $\cV_{K_2} \subset \cV_2$ of size $k_1$ and $k_2$ such that any pair of nodes in $\cV_{K_1} \times \cV_{K_2}$ are connected with an elevated probability no less than $p_0 + \delta$, where $\delta > 0$. We derive the \textit{minimax rate of separation} $\delta^*$ \cite{ingster2012nonparametric}, which is defined as the smallest value of $\delta$ such that consistent community detection is possible (see Equation (\ref{eq_separation}) for a formal definition). Our results are non-asymptotic in nature, and hold for any values of $n_1, n_2, k_1,$ and $k_2$. This degree of precision allows us to describe subtle phase transitions in the separation rate $\delta^*$ which have not been observed in the existing community detection literature.   



The paper is structured as follows. In Sections \ref{sec_litreview}, \ref{sec_contributions}, and \ref{sec_notation} we review relevant literature, outline our contributions, and collect notation used through the paper. In Section \ref{sec:problem} we formally state our problem of interest. Our main results and a discussion thereof are given in Section \ref{sec:results} and Section \ref{sec_discussion}. We conclude with a summary of the limitations of our results and promising directions for future work in Section \ref{sec:conclusion}. All of the proofs are provided in the supplement.

\subsection{Prior work}\label{sec_litreview}

 The framework that we use to study optimal community detection in random graphs was pioneered by Arias-Castro and Verzelen \cite{arias2014community, verzelen2015community}. 
 In \cite{arias2014community}, the authors provide matching asymptotic upper and lower bounds on the minimax risk of detecting a community of vanishing size in a dense graph. 
 They also handle the cases when the size of the community and the baseline connection probability $p_0$ are unknown, and provide polynomial-time tests using a convex relaxation of the max degree test, following \cite{berthet2013optimal}. 
 In \cite{verzelen2015community}, the authors extend their analysis to sparse graphs. 
 While the papers \cite{arias2014community, verzelen2015community} broke significant ground, their results are still limited, as they only consider the restricted fully symmetric setting of $n_1 = n_2$ and $k_1 = k_2$. 
 Subsequent works derive analogous theoretical results for hypergraphs \cite{yuan2022testing, yuan2022sharp}, in which edges can be drawn between more than two vertices. 
 The work \cite{huleihel2024random} studies community detection under information constraints, namely that the statistician can only access small parts of the graph via non-adaptive edge queries. 
 More recently, \cite{rotenberg2024planted} considers the detection of a bipartite (i.e., imbalanced) community in a non-bipartite graph. We provide a detailed comparison of our results to the results of \cite{rotenberg2024planted} in Section \ref{sec:results}. 
 
 In the stochastic block model (SBM) literature, the community detection problem is posed as the task of distinguishing an SBM with a given set of parameters from an Erd\H{o}s-Renyi graph. This problem was first considered by \cite{janson2010asymptotic}, whose results were later refined by \cite{mossel2015reconstruction, massoulie2014community, abbe2015community}, which established that detection is possible if and only if the signal strength is above the Kesten-Stigum threshold, as conjectured by \cite{decelle2011asymptotic}. In \cite{banerjee2018contiguity, banerjee2017optimal}, the authors prove analogous results for asymptotically growing expected degrees and provide optimal hypothesis tests based on signed cycle statistics. The work of \cite{banks2016information} provides an impossibility result for distinguishability for SBM's with a possibly growing number of communities, and \cite{gao2017testing} analyzes a powerful test in this setting constructed using small subgraph statistics. 

 Under the framework of \cite{arias2014community, verzelen2015community}, community detection is equivalent to detecting a submatrix in the observed adjacency matrix. The seminal work of \cite{butucea2013detection} provides matching upper and lower bounds for detecting a sparse submatrix of elevated mean in a matrix of Gaussian random variables under a particular asymptotic regime. Following this, \cite{ma2015computational} initiates a rigorous study of statistical-computational gaps in the submatrix detection problem which was continued by \cite{chen2016statistical, hajek2017information, brennan2018reducibility}. In \cite{luo2022tensor}, the authors study the problem of detecting a planted sparse sub-tensor in a high dimensional Gaussian tensor. Finally, \cite{dadon2024detection} derives upper and lower bounds for detecting the presence of multiple sparse submatrices in Gaussian noise that are tight up to log factors.

 \subsection{Our contributions}\label{sec_contributions}

 In this paper, we make the following contributions.

 \begin{enumerate}
     \item We fully settle the non-asymptotic expression of the minimax rate $\delta^*$ defined in (\ref{eq_separation}), which holds provided the graph is sufficiently dense, as specified in Assumption~\ref{asmp:dense}. 
     In contrast to the previous literature, our bounds always match up to multiplicative constants for any possible values of $n_1,n_2,k_1$ and $k_2$, especially in the under-explored unbalanced regimes where $k_1 \ll k_2$ or $n_1 \ll n_2$. 
     The precision afforded by our results reveals subtle phase transitions that, to the best of our knowledge, had not been documented in previous work. 
     
     \item Our lower bound on $\delta^*$ stated in Theorem \ref{thm_lb} holds without requiring Assumption~\ref{asmp:dense} and uses a very precise application of the second moment method \cite{lecam1973convergence}. 
     It requires several non-trivial lemmas for controlling the moment-generating function of the product of binomial random variables, and substantially departs from the lower bound strategies proposed in~\cite{butucea2013detection} and~\cite{arias2014community}.
     \item Our upper bound stated in Theorem~\ref{thm_ub} is achieved by carefully 
     combining three tests: A standard \textit{total degree} test and two entirely novel tests referred to as the \textit{truncated degree} test and the \textit{max truncated degree} test. 
     The two novel tests are defined using truncated non-linear functions of the adjacency matrix, building on the truncated $\chi^2$ test recently studied in the Gaussian sequence model \cite{collier2017minimax, kotekal2023minimax, liu2021minimax,chhor2024sparse}. 
     However, substantial modifications of the truncated $\chi^2$ test are needed to address two challenges:
     \begin{enumerate}
         \item Moving from the vector case to the matrix case, which requires subtle Bonferroni corrections of the classical truncated $\chi^2$ test.
         \item Handling matrices whose entries are Bernoulli rather than Gaussian random variables, which require significant adjustments of the truncated $\chi^2$ test to carefully manage the complex sub-poissonian tails of the binomial distribution.
     \end{enumerate} 
     Our upper bound holds under Assumption~\ref{asmp:dense}, which we discuss thoroughly in later sections. 
     We also prove a series of lemmas in the supplement that we use to control the Type I and Type II errors of these tests, and we anticipate that these results will be of interest to researchers studying related problems.
 \end{enumerate}
\subsection{Notation}\label{sec_notation}

The following notation will be used throughout the paper. For $p \in \mathbb N$, let $[p] := {1,...,p}$. We use $\mathcal P_k(n)$ to denote the set of subsets of $[n]$ of size $k$. For $a,b \in \R$,
denote $a \lor b:= \max\{a,b\}$ and $a \land b =: \min\{a,b\}$.  We will use $a \lesssim b$ if there exists a constant $C > 0$ depending on $\eta$ such that $a \leq C b$. We say $a \asymp b$ if $a \lesssim b$ and $b \lesssim a$. For two sets $A_1$ and $A_2$, we denote $A_1 \times A_2 = \{(i,j) : i \in A_1, j \in A_2\}$ as the Cartesian product of $A_1$ and $A_2$. For a finite set $A$, we use $|A|$ to denote the cardinality of $A$. We use $\one_{\{\cdot\}}$ as the indicator function, meaning that $\one_{A} = 1$ if the event $A$ occurs and $\one_{A} = 0$ otherwise. For a matrix $\mathbf X$, we use $X_{ij}$ to denote its $(i,j)_{\text{th}}$ entry. Given two probability distributions $\bbP$ and $\mathbb Q$, we use $\text{TV}(\bbP, \mathbb Q) = \sup_A|\bbP(A) - \mathbb Q(A)|$ to denote the total variation distance between $\bbP$ and $\mathbb Q$. We use $c, C, \bar{C}, C', c_1, c_2, c_3$, and $c_4$ to denote constants whose value may change between instances of their usage in the paper. 
We will also use the convention that $[a,b] = \emptyset$ for any two real numbers $a,b$ such that $a>b$. 
For any two integers $n,k \in \mathbb N\cup \{0\}$, we denote by ${n \choose k}$ the binomial coefficient, equal to $\frac{n!}{k! (n-k)!}$.

\section{Problem statement}\label{sec:problem}

Let $n_1, n_2, k_1,k_2 \in \mathbb N$ be integers considered as fixed throughout the paper, and assume $k_1 \leq n_1$ and $k_2 \leq n_2$. 
Let $\bG = (\cE, \cV_1, \cV_2)$ be a random undirected bipartite graph defined on disjoint sets of vertices $\cV_1 = \{ V_{1,1}, \dots, V_{1, n_1}\}$ and $\cV_2 = \{V_{2,1}, \dots, V_{2,n_2}\}$. 
We observe the adjacency matrix $\A \in \R^{n_1 \times n_2}$ with entries $A_{ij} = \one\{\text{$V_{1, i}$ is connected to $V_{2,j}$ by an edge}\}$ and assume that the edges are mutually independent Bernoulli random variables. 
For some parameter $p_0 \in [0,1]$ assumed to be fixed throughout the paper,
our objective is to determine whether each pair of vertices in $\cV_1 \times \cV_2$ is connected with probability $p_0$, or if there exists a  community (i.e., a bipartite subset of vertices) that is connected with greater probability $p_1 \geq p_0 + \delta$ for some $\delta > 0$. 

Formally, we formulate this testing problem 
in terms of the mean structure of the adjacency matrix $\A$. 
Define the matrix $\bP = \EE[\A] \in \R^{n_1 \times n_2}$, which represents the matrix of connection probabilities of the random graph $\bG$.
For any two subsets $K_1 \in \mathcal{P}_{k_1}(n_1)$, $K_2 \in \mathcal{P}_{k_2}(n_2)$ and any $\delta \geq 0$, we define
\begin{align*}
    \Theta(K_1, K_2, \delta) = \left\{\bP \in \R^{n_1\times n_2} \text{ s.t. } \begin{cases}
        \forall (i,j) \in K_1\times K_2: P_{ij} \geq p_0 + \delta\\
        \forall (i,j) \notin K_1\times K_2: P_{ij} = p_0
    \end{cases}
    \right\}.
\end{align*}
A mean matrix $\bP$ belongs to $\Theta(K_1, K_2, \delta)$ if the bipartite community indexed by $K_1 \times K_2$ is more densely connected than the rest of the graph. 
Our focus in this paper is to optimally detect an unknown bipartite community of size $k_1\times k_2$. 
To this end, for any $\delta>0$, we let 
\begin{align*}
    \Theta(k_1,k_2, n_1, n_2, \delta) = \bigcup_{\substack{K_1 \in \mathcal{P}_{k_1}\!(n_1) } }  \bigcup_{K_2 \in \mathcal{P}_{k_2}\!(n_2)}\Theta(K_1, K_2, \delta).
\end{align*}
We consider the following testing problem
\begin{equation}\label{eq:hypothesis-theta}
    \text{H}_0: \forall (i,j) \in [n_1] \times [n_2], P_{ij} = p_0  \quad \text{against} \quad \text{H}_1: \bP \in \Theta(n_1, n_2, k_1, k_2, \delta).
\end{equation}
The hypothesis $\text{H}_0$ is equivalent to observing a bipartite Erd\H{o}s-Renyi random graph with parameter $p_0$. 
The hypothesis $H_1$ is equivalent to the existence of an unknown bipartite community of size $k_1\times k_2$ with connection probabilities at least $p_0 + \delta$.

A \textit{test} is a measurable function of the observed data $\A$ taking its values in $\{0,1\}$. 
We measure the quality of a test $\Delta$ by its \textit{risk}, defined as the sum of its Type I and worst-case Type II errors. 
Specifically, denoting by $\mathbb P_{\bP}$ the probability distribution of $\A$ for $\bP \in \Theta(n_1, n_2, k_1, k_2, \delta)$, and letting $\mathbb P_{0}$ denote the distribution of $\A$ under the null hypothesis, the risk of a test $\Delta$ is defined as
\begin{equation}\label{eq_def_risk}
    \cR(\Delta, \delta) = \mathbb P_{0}(\Delta = 1) + \sup_{\bP \in \Theta(n_1, n_2, k_1, k_2, \delta)} \mathbb P_{\bP}(\Delta = 0).
\end{equation}
The \textit{minimax risk} associated with the problem~\eqref{eq:hypothesis-theta} is defined as
\begin{equation}
    \cR^*(k_1,k_2,n_1,n_2, \delta) = \inf_{\Delta} \cR(\Delta, \delta),
\end{equation}
where the infimum is taken over all tests $\Delta$.
For a desired level of risk $\eta \in (0,1)$, considered as a fixed constant throughout, we are interested in the \textit{minimax separation rate} $\delta^*$ defined as 
\begin{equation}\label{eq_separation}
    \delta^* = \inf \left\{\delta>0: \cR^*(k_1,k_2,n_1,n_2,\delta) \leq \eta\right\}.
\end{equation}
The minimax separation rate $\delta^*$ encodes the difficulty of the testing problem. 
It is the infimal signal strength ensuring the existence of a test with Type I plus Type II errors controlled by the desired level of risk $\eta$. 
The goal of the paper is to derive the value of $\delta^*$ in terms of $n_1, n_2, k_1,k_2$ and $p_0$ up to absolute multiplicative constants, and to construct the minimax-optimal tests achieving a risk at most $\eta$ for the testing problem~\eqref{eq:hypothesis-theta} when the separation satisfies $\delta \geq C \delta^*$ for some constant $C>0$ depending only on $\eta$.

\section{Main results}\label{sec:results}

For any $k_1,k_2, n_1, n_2 \in \mathbb N$ and for a constant $C > 0$ to be chosen later, we define 
\begin{align}
    \psi(k_1,k_2,n_1,n_2) &= \frac{1}{k_1} \log\left(1+ \frac{n_2}{k_2^2} \log\left(e{n_1 \choose k_1}\right)\right) \label{eq_def_psi}\\[5pt]
    \beta(k_1, k_2, n_1, n_2) &= \frac{1}{k_1}\log\left(\frac{n_2}{k_2}\right)\one_{\{\frac{n_1k_2}{n_1^2}\log\left(\frac{n_2}{k_2}\right) > 1\}} \label{eq_def_nu} \\[5pt]
    \phi(k_1,k_2, n_1,n_2) &= \begin{cases}
        \frac{n_1}{k_1^2} \log\left(1+\frac{n_2}{k_2^2}\right)&  \text{ if } \frac{n_1}{k_1^2} \leq C\\
         \infty & \text{ otherwise.}
    \end{cases}\label{eq_def_phi}
\end{align}
To alleviate the notation, we will write
\begin{align*}
    &\phi_{12} = \phi(k_1,k_2, n_1,n_2)
    &\phi_{21} = \phi(k_2,k_1, n_2,n_1)\,\,\\
    &\beta_{12} = \beta(k_1,k_2, n_1,n_2)
    &\beta_{21} = \beta(k_2,k_1, n_2,n_1)\,\, \\
    &\psi_{12} = \psi(k_1,k_2, n_1,n_2)
    &\psi_{21} = \psi(k_2,k_1, n_2,n_1).
\end{align*}
Finally, we define the quantities
\begin{align}
R &:= R(k_1, k_2, n_1, n_2) =\big(\psi_{12} + \psi_{21}\big) \land \phi_{12}  \land \phi_{21}, \label{eq_rate} \\
\tilde{R} &:= \tilde{R}(k_1, k_2, n_1, n_2) = \big(\psi_{12} + \beta_{21}\big) \land \big(\psi_{21} + \beta_{12}\big) \land \phi_{12} \land \phi_{21}.
\end{align}
We will prove that the minimax separation rate $\delta^*$ satisfies 
\begin{align*}
    (\delta^*)^2 \asymp p_0(1-p_0) R,
\end{align*}
for any values of $k_1k_2,n_1$ and $n_2$, under Assumption~\ref{asmp:dense} on $p_0$ given in Section~\ref{sec:ub}.

\subsection{Lower bound}\label{sec:lb}
 
The following theorem gives a lower bound on $\delta^*$ that holds for any values of $n_1, n_2, k_1,k_2$ and $p_0$.
\begin{theorem}\label{thm_lb}
    Let $\eta \in [0,1]$ be given. There exist constants $c_\delta, \bar{C} > 0, $ such that if $\delta^2 \leq c_\delta p_0(1-p_0)R$ with (\ref{eq_def_phi}) defined with $C = \bar{C}$, then it holds
    \[\cR^*(k_1, k_2, n_1,n_2, \delta) > \eta.\]
\end{theorem}
Theorem~\ref{thm_lb} yields the lower bound $(\delta^*)^2 \geq c_\delta p_0(1-p_0) R$ by definition of $\delta^*$. 
The proof of Theorem \ref{thm_lb} uses a careful application of the second moment argument \cite{lecam1973convergence,ingster1982minimax, ingster1987minimax, ingster2010detection, butucea2013detection, arias2014community}. We refer the reader to the supplement for details.
%

\subsection{Upper bound}\label{sec:ub}

Now we present a matching upper bound on the minimax rate of separation $\delta^*$. We do so by carefully combining three testing procedures, which are each constructed from the observed adjacency matrix $\A$. 
Our three hypothesis tests are defined as follows.
\begin{enumerate}
    \item \textit{Total degree test:} Let $\sigma_{\lin} = \sqrt{n_1n_2p_0(1-p_0)}$ and define the test statistic
\begin{equation}\label{eq_def_linear}
        t_{\lin} = \frac{1}{\sigma_{\lin}}\sum_{\substack{i = 1}}^{n_1}\sum_{j = 1}^{n_2}(A_{ij} - p_0)
    \end{equation}
    The total degree test is defined as $\Delta_\lin^h = \one(t_\lin > h)$ for a choice of threshold $h > 0$.
    \item \textit{Truncated degree test:}
    To define this test, we introduce the Bennett function, defined as
\begin{align}
    &h_B(x) = (1+x) \log(1+x) - x,\qquad \forall x>-1.\label{eq_def_h_bennett}\\
    &h_B(-1) = 1.
\end{align}
For any $j \in [n_2]$, and $a \geq 1$, let $\bar \sigma = \sqrt{n_1p_0(1-p_0)}$ and define  
    \begin{align}
        \bar{A}_j &= \frac{1}{\bar \sigma}\sum_{i = 1}^{n_1}(A_{ij}-p_0)\nonumber\\
        W_j &= 2n_1(1-p_0) \,h_B\!\left(-\frac{\bar \sigma \bar A_j}{n_1(1-p_0)}\right) + 2n_1 p_0 \,h_B\!\left(\frac{\bar \sigma \bar A_j}{n_1p_0}\right).\label{eq_def_trunc_deg}
    \end{align}
    Letting $\nu_a^{n_1} =  \mathbb 
    E\left[W_1| \bar A_1 \geq a\right]$ and $\tau > 0$ denote a parameter to be chosen later, define\begin{equation}\label{eq_def_linchisq}
        t_{\chisqlin, 1} = \sum_{j = 1}^{n_2}\big(W_j - \nu_\tau^{n_1}\big)\one(\bar{A}_j \geq \tau).
    \end{equation}
    The truncated degree test is as $\Delta_{\chisqlin, 1}^h = \one(t_{\chisqlin, 1} > h)$ for some threshold $h > 0$. 
    \item \textit{Max truncated degree test:}  
    For any $J_1 \in \mathcal{P}_{k_1}(n_1)$ and any $j \in [n_2]$, we let $\sigma_{J_1,j} = \sqrt{k_1 p_0 (1 - p_0)}$ and define
    \begin{align}
        &t_{J_1, j} = \frac{1}{\sigma_{J_1,j}}\sum_{i \in J_1}\big(A_{ij} - p_0\big)\nonumber\\
        &W_{J_1,j} = 2k_1(1-p_0)h_B\left(-\frac{\sum_{i\in J_1} (A_{ij}-p_0)}{k_1(1-p_0)}\right) + 2k_1 p_0 h_B\left(\frac{\sum_{i\in J_1} (A_{ij}-p_0)}{k_1 p_0}\right)\label{eq_def_max_trunc_deg}
    \end{align}
    Letting $\nu_a^{k_1} =  \mathbb 
    E\left[W_{J_1}| t_{J_1, j} \geq a\right]$ and $\tau > 0$ be a parameter to be chosen later, define
\begin{equation}\label{eq_def_maxchisq}
        t_{\chisqmax, 1} = \max \left\{\sum_{j = 1}^{n_2}\big(W_{J_1,j} - \nu^{k_1}_{\tau}\big)\one(t_{J_1, j} > \tau)\,\Big|\, J_1 \in \mathcal{P}_{k_1}(n_1)\right\}.
    \end{equation}

    Finally, for a choice of threshold $h > 0$, we define the maximum truncated degree test as $\Delta_{\chisqmax, 1}^h = \one(t_{\chisqmax, 1} > h)$. 
\end{enumerate}
Analogously, we define the tests $\Delta_{\chisqmax, 2}^h$ and $\Delta_{\chisqlin, 2}^h$ by swapping the roles of $k_1,k_2$, and $n_1,n_2$ in the definitions of $\Delta_{\chisqmax, 1}^h$ and $\Delta_{\chisqlin, 1}^h$ respectively.

For some constant $c_1>0$ depending only on the desired level of risk $\eta$, we further combine the degree and truncated degree tests as follows. Let 
\begin{align*}
    &\Delta_a^{h_1,h_2} = \begin{cases}
        \Delta^{h_1}_{\chisqlin, 1} & \text{if  $\frac{n_2}{k_2^2} \geq c_1$,} \\ \Delta^{h_2}_{\lin} & \text{otherwise}
    \end{cases} \qquad \text{ and } 
    \qquad \Delta_b^{h_1',h_2'} = \begin{cases}
        \Delta^{h_1'}_{\chisqlin, 2} & \text{if  $\frac{n_1}{k_1^2} \geq c_1$,} \\ \Delta^{h_2'}_{\lin} & \text{otherwise.}
    \end{cases} 
\end{align*}
Using these procedures as our building blocks, we construct our final test $\Delta^*$ by applying the relevant test depending on what part dominates in the expression $\tilde R$. Specifically, our optimal test is defined as
\begin{align*}
    \Delta^* = \begin{cases}
        \Delta^{h_3}_{\chisqmax, 1}& \text{ if } \tilde R = \psi_{12} + \beta_{21}\\
        \Delta^{h_4}_{\chisqmax, 2}& \text{ if } \tilde R = \psi_{21} + \beta_{12} \\
        \Delta_a^{h_1, h_2} & \text{ if } \tilde R = \phi_{12}\\
        \Delta_b^{h'_1, h'_2} & \text{ if } \tilde R = \phi_{21}.
    \end{cases}
\end{align*}

Our upper bound on the risk of $\Delta^*$ holds under the following assumption, which guarantees that the observed graph $\bG$  is sufficiently dense.

\begin{assumption}[Graph density.]\label{asmp:dense}
    There exists a constant $C_\eta > 0$ depending only on $\eta$ such that the baseline connection probability $p_0$ satisfies $p_0 \leq 1/4$ and
    \begin{align*}
        p_0 \geq \begin{cases}
            \frac{C_\eta}{k_1k_2} \log\left(e{n_1 \choose k_1}{n_2 \choose k_2}\right) & \text{ if } \tilde R = \big(\psi_{12} + \beta_{21}\big) \land \big(\psi_{21} + \beta_{12}\big)\\[5pt]
            \frac{C_\eta}{n_1} \log\left(1+ \frac{n_2}{k_2^2}\right) & \text{ if } \tilde R = \phi_{12} \text{ and } n_2>k_2^2\\[5pt]
            \frac{C_\eta}{n_2} \log\left(1+ \frac{n_1}{k_1^2}\right) & \text{ if } \tilde R = \phi_{21} \text{ and } n_1>k_1^2\\[5pt]
            \frac{C_\eta}{n_1n_2} & \text{ otherwise.}
        \end{cases} 
    \end{align*}
\end{assumption}
The upper bound $p_0 \leq \frac14$ in Assumption \ref{asmp:dense} is used to apply Slud's inequality for binomial anti-concentration \cite{slud1977distribution}. We invoke the different cases of Assumption \ref{asmp:dense} depending on the exact construction of $\Delta^*$. 
Assumption \ref{asmp:dense} places us in a quasi-normal moderate deviation regime, which enables us to control the tails of the test statistics (\ref{eq_def_linear}), (\ref{eq_def_linchisq}), and (\ref{eq_def_maxchisq}) with sufficient precision. 
These tests are based on binomial statistics, which exhibit sub-Gaussian concentration properties in the moderate deviation regime and sub-poissonian concentration in the large deviation regimes (see, e.g.,~\cite{arias2015sparse,chhor2021goodness,kotekal2024locally,chhor2022sharp}).
This dichotomy has also been noted in the two companion papers~\cite{arias2014community} and \cite{verzelen2015community}, which split the analysis between the case of  \textit{dense} graphs, where a quasi-normal regime emerges (see the first row of Table 1 in~\cite{arias2014community}), and \textit{sparse} graphs, where fundamentally different behaviors occur due to lower connectivity. 
In this paper, we follow the same approach by restricting to dense graphs as specified in Assumption~\ref{asmp:dense}, and leave the case of sparse graphs for future work. 
 Furthermore, Assumption \ref{asmp:dense} has a natural probabilistic interpretation, as codified in the following propositions.

\begin{proposition}\label{prop_nonempty_subgraph}
    Let $c > 0$. There exists a constant $C > 0$ such that if
    \[\frac{C}{k_1k_2} \log\left(e{n_1 \choose k_1}{n_2 \choose k_2}\right) \leq p_0 \leq \frac{1}{4},\]
    then it holds
    \(\bbP_0\left(\text{$\bG$ has an empty $k_1 \times k_2$ bipartite subgraph}\right) \leq c.\)
    
    Here, we say that a subgraph $\bG_1$ of $\bG$ is empty if all the vertices in $\bG_1$ has degree equal to 0.
\end{proposition}
\begin{proposition}\label{prop_nonempty_half} 
    Let $c > 0$. There exists a constant $C > 0$ such that if
    \[\frac{C}{n_2k_1} \log\left(e{n_1 \choose k_1}\right) \leq p_0 \leq \frac{1}{4},\]
    then it holds
    \(\bbP_0\Big(\exists I \in \mathcal{P}_{k_1}(n_1): \sum_{i \in I} \sum_{j=1}^{n_2} A_{ij} = 0\Big) \leq c.\)
\end{proposition}

Propositions \ref{prop_nonempty_subgraph} and \ref{prop_nonempty_half} show that Assumption \ref{asmp:dense} prevents the occurrence of extreme events under $\bbP_0$ that would otherwise inflate the Type I error of our testing procedures, especially the max truncated degree test in (\ref{eq_def_maxchisq}). We remark that the lower bound on $p_0$ assumed in Proposition \ref{prop_nonempty_half} is stronger than what we need in Assumption \ref{asmp:dense}; we include this result to provide intuition.

The following theorem guarantees that the minimax risk of $\Delta^*$ is small whenever $\delta$ is large enough.
\begin{theorem}\label{thm_ub}
    Let $\eta \in [0,1]$ be given. Suppose that Assumption \ref{asmp:dense} holds. Then there exist constants $C_\delta, \bar{C} > 0$ and thresholds $h_1, h'_1, h_2, h'_2, h_3, h_4 > 0$ such that if $\delta^2 \geq C_\delta p_0(1- p_0)R$ with (\ref{eq_def_phi}) defined with $C = \bar{C}$, then the test $\Delta^*$ satisfies
    \[\cR(\Delta^*, \delta) < \eta.\]
\end{theorem}
Theorem \ref{thm_ub} demonstrates that our proposed test $\Delta^*$ is able to match the lower bound on $\delta^*$ given in Theorem \ref{thm_lb} under Assumption \ref{asmp:dense}. Combining Theorems \ref{thm_lb} and \ref{thm_ub}, we have therefore identified the minimax rate of separation $\delta^*$ up to constants. 

\section{Discussion}\label{sec_discussion}

\subsection{Discussion of the proposed tests}
Our test $\Delta^*$ is a careful combination of three testing procedures. 
The total degree test has been commonly applied in previous works~\cite{arias2014community,rotenberg2024planted}, along with the \textit{scan test}, which consists of scanning over all possible subgraphs of size $k_1 \times k_2$ and rejecting if one of them contains an unusually large number of edges. 
Notably, our results do not make use of the scan test, and highlight that it can always be successfully replaced by a max truncated degree test, with lower time complexity\footnote{ Computing $\Delta^{h_3}_{\chisqmax, 1}$ requires $O(k_1n_2{n_1 \choose k_1})$ operations, rather than $O(k_1k_2{n_1 \choose k_1}{n_2 \choose k_2})$ for the scan test. }. 

Our two novel tests $\Delta_{\chisqlin, 1}^h$ and $\Delta^{h}_{\chisqmax, 1}$ build on the truncated $\chi^2$ test recently developed in the Gaussian sequence model. 
Given a vector $X \in \R^d$, the truncated $\chi^2$ test statistic is given by
\begin{align*}
    T = \sum_{j=1}^d (X_j^2 - \nu_a) \one{\big(|X_j| \geq a\big)}
\end{align*}
for some $a>0$ and where $\nu_a = \mathbb E\left(Z^2 \big|\,|Z| \geq a\right)$. 
On top of being fast to compute, this test has been shown to be optimal in the Gaussian sequence model for detecting $s$-sparse alternatives separated in $\ell_2$ norm when $s \leq \sqrt{d}$~\cite{collier2017minimax}, whereas the max test, which rejects for large enough values of $\max \big\{\sum_{j \in J} X_j^2 | J \in \mathcal{P}_{s}(d)\big\}$, is known to be statistically suboptimal~\cite{baraud2002non} while exhibiting exponential time complexity. 

We aim to adapt this test to the matrix case where $X \in \R^{d_1 \times d_2}$. 
For the sake of discussion, we will focus on the truncated degree test~\eqref{eq_def_linchisq}, the ideas underlying the max truncated degree test~\eqref{eq_def_maxchisq} being analogous.
An idea for adapting the truncated $\chi^2$ test would be to use 
\begin{align*}
    \sum_{i=1}^{n_1} \Big(\Big(\sum_{j=1}^{n_2} X_{ij}\Big)^2 - \nu_a\Big)\one{\big(\sum_{j=1}^{n_2} X_{ij} \geq a\big)}
\end{align*}
for a suitable re-centering parameter $\nu_a$.
Whereas squares of Gaussian random variables concentrate sub-exponentially (see Lemma 5 in~\cite{liu2021minimax}), squares of binomial distributions do not exhibit such favorable concentration properties, due to their sub-poissonian tails. 
In the definition of the test statistics~\eqref{eq_def_trunc_deg} and~\eqref{eq_def_max_trunc_deg}, we therefore replace the square function with a carefully chosen function with slower growth. 
Focusing on the truncated degree test for illustration, and recalling~\eqref{eq_def_trunc_deg}, we can check that
\begin{align*}
    W_j &= 2n_1(1-p_0) \,h_B\!\left(-\frac{\bar \sigma \bar A_j}{n_1(1-p_0)}\right) + 2n_1 p_0 \,h_B\!\left(\frac{\bar \sigma \bar A_j}{n_1p_0}\right) \asymp 
        \bar \sigma \bar A_j \log(1+\frac{\bar A_j}{\bar \sigma}).
\end{align*}
When $\bar A_j \leq \bar \sigma$, which exactly corresponds to the moderate deviation regime for binomial distributions, we obtain $W_j \asymp \bar A_j^2$ and our test statistic $t_{\chisqlin, 1} = \sum_{j = 1}^{n_2}\big(W_j - \nu_\tau^{n_1}\big)\one(\bar{A}_j \geq \tau)$ becomes analogous to the truncated $\chi^2$ test described above. 
In the large deviation regime where $\bar A_j \gg \bar \sigma$, we have $W_j \asymp \bar \sigma \bar A_j \log(\frac{\bar A_j}{\bar \sigma})$, which grows less fast than $\bar A_j^2$ and mitigates the extreme values of the binomial distribution. 

\subsection{Comparison with existing results}

Through Theorems \ref{thm_lb} and \ref{thm_ub}, we have shown
\begin{equation}\label{eq_delta*}
    (\delta^*)^2 \asymp p_0 (1- p_0) R,
\end{equation}
with $R$ defined in (\ref{eq_rate}). We can understand the nuances of (\ref{eq_delta*}) by first considering our results when restricted to the fully balanced (effectively non-bipartite) setting. Suppose that $k_1 = k_2 = k$ and $n_1 = n_2 = n$, and that $n \lesssim k^2$. From the definition of $R$, we can rewrite (\ref{eq_delta*}) as
\begin{equation}\label{eq_delta*_square}
\frac{(\delta^*)
^2}{p_0(1-p_0)} \asymp \frac{1}{k}\log\Big(\frac{n}{k}\Big) \land \frac{n}{k^2}\log\Big(1 + \frac{n}{k^2}\Big).
\end{equation}
If $\frac{n}{k^2} \gg 1$ or $\frac{n}{k^2} \ll 1$, the form of $\delta^*$ in (\ref{eq_delta*_square}) matches that given by Theorem 2 of \cite{arias2014community} in the moderate deviation regime imposed by Assumption \ref{asmp:dense}. Our result is, in fact, more refined than that of \cite{arias2014community}, as we are able to precisely capture the phase transition around $\frac{n}{k^2} \asymp 1$ thanks to our truncated degree test.

The detection boundary for imbalanced bipartite graphs is more complex. The following proposition describes a new phase transition that emerges in the imbalanced bipartite setting, but not in the non-bipartite setting.

\begin{proposition}\label{prop_example}
Suppose that $k_1^2 \geq \bar{c}n_1 k_2$ for a constant $\bar{c} > 0$ and $k_j \leq c_j n_j$ for $j \in \{1, 2\}$ where $c_1,c_2 > 0$ are sufficiently small constants. Additionally, suppose that there exists a constant $\alpha > 0$ such that $n_2 \geq k_2^{2 + \alpha}$ and that $\frac{n_1}{k_1} \geq e \log(\frac{n_2}{k_2})$. Then it holds
\begin{equation}\label{eq_pt}
    \frac{(\delta^*)
^2}{p_0(1-p_0)} \asymp \frac{1}{k_2}\log\left(1 + \frac{n_1k_2}{k_1^2}\log(n_2)\right).
\end{equation}
In particular, this reveals a phase transition at $\frac{n_1k_2}{k_1^2}\log(n_2)\asymp 1$.
\end{proposition}
To our knowledge, (\ref{eq_pt}) is described nowhere in the existing community detection literature, as previous works either impose restrictions on the shape of the observed graph and the community of interest (i.e. enforcing $n_1 = n_2$ and/or $k_1 = k_2$) or provide results that are not sharp for all configurations of $n_1, n_2, k_1$, and $k_2$. The recent paper of \cite{rotenberg2024planted} presents the results that are most similar to ours in the literature. They consider the case $n_1 = n_2 = n$ with $k_1 \neq k_2$ in the asymptotic regime $k_1 + k_2 = o(n)$. We again emphasize that our results hold for $n_1 \neq n_2$ and any $k_1$ and $k_2$, and as such are immediately more general than those of \cite{rotenberg2024planted}. However, we argue that our results are more precise even in the restricted case of $n_1 = n_2 = n$ and $k_1 + k_2 = o(n)$. Under the assumption that $p_0, \delta^* = \Theta(n^{-\alpha})$ for $\alpha \in (0,2]$, in Theorem 1 of \cite{rotenberg2024planted} the authors show that 
\[\frac{(\delta^*)^2}{p_0(1-p_0)} = \Omega\left(\frac{1}{k_1 \land k_2} \land \frac{n^2}{k_1^2 k_2^2} \land \frac{n}{k_1^2 \lor k_2^2}\right),\]
and in Theorem 2, they show
\[\frac{(\delta^*)^2}{p_0(1-p_0)} = O \left(\frac{\log (n)}{k_1 \land k_2} \land \frac{n^2}{k_1^2 k_2^2} \land \frac{n\log(n)}{k_1^2 \lor k_2^2}\right).\]
We refer readers to equations (14) and (15) in \cite{rotenberg2024planted} and the surrounding discussion for details. 
Their lower bound on $\delta^*$ is loose by a logarithmic factor in $n$, and the asymptotic assumption $\delta^* = \Theta(n^{-\alpha})$ excludes the regime in which the complex rate that we derive in (\ref{eq_delta*}) emerges. 
As such, the results of \cite{rotenberg2024planted} do not describe, for instance, the phase transition (\ref{eq_pt}). We also comment on the assumptions on $p_0$ made in \cite{rotenberg2024planted} and in our work. Recall that the results of \cite{rotenberg2024planted} require $p_0 \gtrsim n^{-\alpha}$ for $\alpha \in (0,2]$, where under our Assumption \ref{asmp:dense}, we require at least $p_0 \gtrsim \frac{1}{n^2}$ and, in certain regimes, we need $p_0 \gtrsim \frac{1}{k_1k_2}\log( {n \choose k_1} {n \choose k_2})$. This assumption is stronger than that of \cite{rotenberg2024planted}, and hence excludes our results from holding for all values of $p_0$ considered by \cite{rotenberg2024planted}. However, as their results are not sharp, they are not able to describe the entirety of the detection boundary in the case when our Assumption \ref{asmp:dense} does not hold. We anticipate that the minimax rate of separation $\delta^*$ may actually differ non-trivially from the rate derived in \cite{rotenberg2024planted} in this setting, and we view the extension of our results beyond Assumption \ref{asmp:dense} as a crucial piece of future work. Finally, we remark that \cite{elimelech2025detecting} extends the results of \cite{rotenberg2024planted} beyond planted bipartite subgraphs, but they do so in the same asymptotic setting and thus suffer from the same limitations.

\section{Conclusion, limitations, and open problems}\label{sec:conclusion}

We have presented a rigorous study of the community detection problem in bipartite graphs. We present a lower bound on the minimax rate of separation, and describe a novel optimal testing strategy that achieves the minimax rate when the observed graph is dense. Our non-asymptotic results hold for all possible dimensions of the observed graph and the planted subgraph, which reveals new behavior in the detection boundary. Here, we outline some limitations and opportunities for future work.

\bigskip 

\noindent \textbf{Adaptivity to unknown $p_0, k_1,$ and $k_2$.} All three of our tests require knowledge of $p_0$, the connection probability under the null hypothesis, which is typically unavailable in practice. In \cite{arias2014community}, the authors consider this problem in the case of non-bipartite graphs, and find that community detection without knowledge of $p_0$ exhibits substantially different behavior. We anticipate that similar phenomena will arise in the bipartite setting.

Furthermore, our max truncated degree test requires knowledge of the size of the community as encoded in $k_1$ and $k_2$. Optimal adaptivity to unknown $k_1$ and $k_2$ is achieved in previous works such as \cite{butucea2013detection, arias2014community} by scanning over possible values of $k_1$ and $k_2$ and performing a Bonferroni correction. However, this not an option for us, since the tails of the test statistic (\ref{eq_def_maxchisq}) are too heavy to still be optimal after a union bound. One possible strategy is to use Lepski-style adaption as in \cite{kotekal2023sparsity}; we leave this to future work. 

\bigskip 

\noindent \textbf{Beyond dense bipartite graphs.} The tightness of our upper bound in Theorem \ref{thm_ub} relies on Assumption \ref{asmp:dense}. Prior work in the non-bipartite community detection literature \cite{verzelen2015community} suggests that the detection landscape is very different for sparse graphs. Understanding optimal detection in bipartite graphs when $p_0$ is extremely small will potentially require testing procedures and new techniques for proving lower bounds, likely based on the truncated second moment method \cite{ingster2012nonparametric, butucea2013detection, arias2014community, verzelen2015community}.

\bigskip 

\noindent \textbf{Computational considerations.} Our  max truncated degree test is unfeasible to compute on datasets of even moderate size. It is well established that the community detection problem in non-bipartite graphs exhibits a statistical-computational gap \cite{ma2015computational, arias2014community, hajek2015computational,brennan2018reducibility, dadon2024detection, rotenberg2024planted, elimelech2025detecting}, meaning that it is not, in general, possible to optimally detect a planted community with a test that runs in polynomial time. We conjecture that this gap persists in the bipartite case, and establishing this phenomenon formally is an open problem of great interest.

\subsubsection*{Acknowledgments}
This paper has been funded by the Agence Nationale de la Recherche under grant ANR-17-EURE-0010 (Investissements d'Avenir program). Parker Knight is support by an NSF Graduate Research Fellowship. The authors thank the Rose Traveling Fellowship of the Harvard T.H. Chan School of Public Health for supporting an international visit between the authors, without which this work would not have been possible.

\printbibliography

\newpage 

\appendix

\section{Proofs for lower bound}

\subsection{Proof of Theorem 1}

In this section, we prove the lower bound on the minimax rate of separation $\delta^*$. Our general strategy is to lower bound the minimax risk of testing with the Bayes risk defined with respect to a well-chosen prior over the parameter space. We then invoke the Neyman-Pearson lemma and carefully control the second moment of the resulting likelihood ratio statistic. 

\bigskip \noindent Let $\eta \in (0,1)$ be given. Recall that the minimax risk is defined as
\begin{align*}
    \cR^*(k_1,k_2,n_1,n_2, \delta) &= \inf_\Delta \cR(\Delta, \delta) \\ 
    &= \inf_\Delta\Big\{\mathbb P_{0}(\Delta = 1) + \sup_{\bP \in \Theta(k_1,k_2,n_1,n_2, \mu)} \mathbb P_\bP(\Delta = 0)\Big\}.
\end{align*}
Define the reduced parameter space $\overline{\Theta}(k_1, k_2, n_1, n_2, \delta)$ as
\[\overline{\Theta}(k_1, k_2, n_1, n_2, \delta) = \{\bP \in \Theta(k_1, k_2, n_1, n_2, \delta) : P_{ij} \neq 0 \implies P_{ij} = p_0 + \delta\}.\]
We let $\pi$ denote the uniform distribution over $\overline{\Theta}(k_1, k_2, n_1, n_2, \delta)$, meaning that for any $\mathbf M \in \overline{\Theta}(k_1, k_2, n_1, n_2, \delta)$ it holds $\bbP_{\bP \sim \pi}(\bP = \mathbf M) = \frac{1}{{n_1 \choose k_1}{n_2 \choose k_2}}$. Using $\pi$ as our prior on $\overline{\Theta}(k_1, k_2, n_1, n_2, \delta)$, we define the mixture distribution $\mathbb{P}_\pi$ on $\R^{n_1\times n_2}$ as
\[\mathbb P_{\pi}(B)=\int_{\bP \in \overline{\Theta}(k_1, k_2, n_1, n_2, \delta)}\bbP_\bP(B)\pi(\D \bP),\]
where $B$ is a measurable set. With this infrastructure in hand, we can lower bound the minimax risk as follows:
\begin{align*}
    \cR^*(k_1,k_2,n_1,n_2, \delta) &=\inf_\Delta\Big\{\mathbb P_{0}(\Delta = 1) + \sup_{\bP \in \Theta(k_1,k_2,n_1,n_2, \delta)} \mathbb P_\bP(\Delta = 0)\Big\} \\
    &\geq \inf_\Delta\Big\{\mathbb P_{0}(\Delta = 1) + \sup_{\bP \in \overline{\Theta}(k_1,k_2,n_1,n_2, \mu)} \mathbb P_\bP(\Delta = 0)\Big\} \\
    &\geq \inf_\Delta\Big\{\mathbb P_{0}(\Delta = 1) + \mathbb P_\pi(\Delta = 0)\Big\}.
\end{align*}
The final expression above is the minimax risk of a simple versus simple hypothesis testing problem, and we can characterize it precisely using the Neyman-Pearson lemma (\cite{rigollet2023high}, Lemma 4.3). Combining this result with standard equivalent formulations of the total variation distance (\cite{rigollet2023high}, Proposition 4.4), we have
\begin{align*}
    \inf_\Delta\Big\{\mathbb P_{0}(\Delta = 1) + \mathbb P_\pi(\Delta = 0)\Big\} &= 1 - \text{TV}(\bbP_0, \mathbb P_{\pi}) \\
    &= 1 - \frac12\int_{\A \in \R^{n_1\times n_2}}|\D\bbP_0(\A) - \D \mathbb P_\pi (\A)| \\
    &= 1 - \frac12\int_{\A \in \R^{n_1\times n_2}}\Big|\frac{\D\mathbb P_\pi (\A)}{\D \bbP_0(\A)} - 1\Big|\D \bbP_0(\A),
\end{align*}
where $\D \bbP_0$ and $\D \mathbb P_\pi$ denote the Radon-Nikodym derivatives with respect to the counting measure of $\bbP_0$ and $\mathbb P_\pi$ respectively. Letting $L_\pi = \frac{\D\mathbb P_\pi }{\D \bbP_0}$ denote the likelihood ratio, we apply the Cauchy-Schwarz inequality to obtain
\begin{align*}
    \int_{\A \in \R^{n_1\times n_2}}\Big|\frac{\D\mathbb P_\pi (\A)}{\D \bbP_0(\A)} - 1\Big|\D \bbP_0(\A) &= \int_{\A \in \R^{n_1\times n_2}}\Big|L_\pi(\A) - 1\Big|\D \bbP_0(\A) \\
    &\leq \sqrt{\int_{\A \in \R^{n_1\times n_2}}\Big(L_\pi(\A) - 1\Big)^2\D \bbP_0(\A)} \\
    &= \sqrt{\EE_0\big[L_\pi^2(\A)\big] - 1},
\end{align*}
where we arrive at the final expression by expanding the square and using $\EE_0[L_\pi(\A)] = 1$. This chain of calculations reveals
\[\cR^*(k_1, k_2, n_1, n_2) \geq 1 - \frac12\sqrt{\EE_0\big[L_\pi^2(\A)\big] - 1}.\]
Therefore, to prove that $\cR*(k_1, k_2, n_1, n_2) \geq \eta$, it suffices to show
\begin{align*}
    \EE_0\big[L_\pi^2\big] &\leq 1 + 4(1 - \eta)^2 \\
    &= 1 + \eps,
\end{align*}
where we define $\eps := 4(1 - \eta)^2$. By direct calculation, we have
\begin{align*}
    \EE_0\big[L_\pi^2\big] &= \int_{\A}\frac{\big(\mathbf \D \mathbf P_\pi(\A)\big)^2}{\D \bbP_0(\A) } \\
    &= \int_{\A}\frac{\int \D \bbP_\bP(\A)\pi(\D \bP)\int \D \bbP_{\bP'}(\A)\pi(\D \bP')}{\D \bbP_0(\A) } \\
    &= \int_\bP \int_{\bP'}\Big[\int_\A\frac{\D \bbP_\bP(\A)\D \bbP_{\bP'}(\A)}{\D \bbP_0(\A)}\Big]\pi(\D \bP)\pi(\D \bP'), 
\end{align*}
where $\bP$ and $\bP'$ are two mean matrices drawn independently from $\pi$. Recalling that the entries of $\A$ are Bernoulli random variables, we have
\[\D \mathbb P_\bP(\A) = \prod_{i=1}^{n_1}\prod_{j=1}^{n_2} P_{ij}^{A_{ij}} (1-P_{ij})^{1-A_{ij}}.\]
Using this, as well as the definition of our prior $\pi$, we can compute the above triple integral as follows. Here, we let $\cA = \{0,1\}^{n_1 \times n_2}$ denote the set of all possible values of the adjacency matrix $\A$.
\begin{align*}
    \EE_0\big[L_\pi^2\big] &= \int_\bP \int_{\bP'}\Big[\int_\A\frac{\D \bbP_\bP(\A)\D \bbP_{\bP'}(\A)}{\D \bbP_0(\A)}\Big]\pi(\D \bP)\pi(\D \bP')\\ 
    &= \frac{1}{{n_1 \choose k_1}^2 {n_2 \choose k_2}^2}\sum_{\bP}\sum_{\bP'}\sum_{\A \in \cA}\prod_{i=1}^{n_1}\prod_{j=1}^{n_2}\frac{P_{ij}^{A_{ij}} (1-P_{ij})^{1-A_{ij}}{P'_{ij}}^{A_{ij}} (1-P'_{ij})^{1-A_{ij}}}{p_0^{A_{ij}} (1-p_0)^{1-A_{ij}}} \\
    &= \frac{1}{{n_1 \choose k_1}^2 {n_2 \choose k_2}^2}\sum_{\bP}\sum_{\bP'}\prod_{i=1}^{n_1}\prod_{j=1}^{n_2}\sum_{\A \in \cA: A_{ij} = 0}^1\left(\frac{P_{ij}^{A_{ij}} (1-P_{ij})^{1-A_{ij}}{P'_{ij}}^{A_{ij}} (1-P'_{ij})^{1-A_{ij}}}{p_0^{A_{ij}} (1-p_0)^{1-A_{ij}}}\right) \\
    &= \frac{1}{{n_1 \choose k_1}^2 {n_2 \choose k_2}^2}\sum_{\bP}\sum_{\bP'}\prod_{i=1}^{n_1}\prod_{j=1}^{n_2}\left(\frac{P_{ij}\cdot P'_{ij}}{p_{0}}
    +    
    \frac{ (1-P_{ij}) \cdot (1-P'_{ij})}{1-p_{0}}\right).
\end{align*}
Now, fix two mean matrices $\bP$ and $\bP'$ with elevated entries on $K_1 \times K_2$ and $K_1' \times K_2'$ respectively. Letting $(i,j) \in [n_1] \times [n_2]$, suppose that $(i,j) \notin (K_1 \cap K_1') \times (K_2 \cap K_2')$. Without loss of generality, say $(i,j) \notin K_1 \times K_2$. Then it holds that $P_{ij} = p_0$, and we have
\[\frac{P_{ij}\cdot P'_{ij}}{p_{0}}
    +    
    \frac{ (1-P_{ij}) \cdot (1-P'_{ij})}{1-p_{0}} = P'_{ij} + (1-P'_{ij}) = 1.\]
On the other hand, if $(i,j) \in (K_1 \cap K_1') \times (K_2 \cap K_2')$, it holds
\begin{align*}
    \frac{P_{ij}\cdot P'_{ij}}{p_{0}}
    +    
    \frac{ (1-P_{ij}) \cdot (1-P'_{ij})}{1-p_{0}} &= \frac{(p_0 + \delta)^2}{p_{0}}
    +    
    \frac{ (1-p_0 - \delta)^2}{1-p_{0}}\\
    & = p_0 - 2 \delta + \frac{\delta^2}{p_0} + 1-p_0 - 2 \delta + \frac{\delta^2}{1-p_0}\\
    & = 1 + \frac{\delta^2}{p_0(1-p_0)}.
\end{align*}
With these calculations in hand, we have
\begin{align*}
    \EE_0\big[L_\pi^2\big] &= \frac{1}{{n_1 \choose k_1}^2 {n_2 \choose k_2}^2}\sum_{\bP}\sum_{\bP'}\prod_{i=1}^{n_1}\prod_{j=1}^{n_2}\left(\frac{P_{ij}\cdot P'_{ij}}{p_{0}}
    +    
    \frac{ (1-P_{ij}) \cdot (1-P'_{ij})}{1-p_{0}}\right) \\
    &= \frac{1}{{n_1 \choose k_1}^2 {n_2 \choose k_2}^2}\sum_{\bP}\sum_{\bP'}\prod_{i=1}^{n_1}\prod_{j=1}^{n_2}\left(1 + \frac{\delta^2}{p_0(1-p_0)}\one \{ (i,j) \in (K_1 \cap K_1') \times (K_2 \cap K_2')\}\right) \\
    &\leq \frac{1}{{n_1 \choose k_1}^2 {n_2 \choose k_2}^2}\sum_{\bP}\sum_{\bP'}\prod_{i=1}^{n_1}\prod_{j=1}^{n_2}\exp\left(\frac{\delta^2}{p_0(1-p_0)}\one \{ (i,j) \in (K_1 \cap K_1') \times (K_2 \cap K_2')\}\right) \\
    &= \frac{1}{{n_1 \choose k_1}^2 {n_2 \choose k_2}^2}\sum_{\bP}\sum_{\bP'}\exp\left(\frac{\delta^2}{p_0(1-p_0)}|K_1 \cap K_1'| |K_2 \cap K_2'|\right) \\
    &= \EE_{\substack{U \sim \text{HypGeom}(n_1, k_1, k_1) \\ V \sim \text{HypGeom}(n_2, k_2, k_2)}}\left[\exp\left(\frac{\delta^2}{p_0(1-p_0)} UV\right)\right] \\
    &\leq \EE_{\substack{X \sim \text{Bin}(k_1, \frac{k_1}{n_1 - k_1}) \\ Y \sim \text{Bin}(k_2, \frac{k_2}{n_2 - k_2})}}\left[\exp\left(\frac{\delta^2}{p_0(1-p_0)} XY\right)\right],
\end{align*}
where the final inequality follows from Lemma 3 of \cite{arias2011global}.Therefore, it suffices to show that
\begin{equation}\label{eq_mgf_goal}
    \EE[\exp(\mu^2XY)] \leq 1+\eps,
\end{equation}
where $X \sim \text{Bin}(k_1, \frac{k_1}{n_1 - k_1}), Y \sim \text{Bin}(k_2, \frac{k_2}{n_2 - k_2}),$ and $\mu^2 = \frac{\delta^2}{p_0(1 - p))}$. The remainder of the proof is structured as follows.
\begin{enumerate}
    \item In the \textit{general sparsity} setting, meaning that $k_1 \leq c_1 n_1$ and $k_2 \leq c_2 n_2$ for sufficiently small constants $c_1, c_2 > 0$, we show that there exists a constant $c_\mu > 0$ such that if $\mu^2 \leq c_\mu R$, then (\ref{eq_mgf_goal}) holds. The proof of this claim constitutes the primary technical difficulty of the derivation of our lower bound, and relies on a precise analysis of the moment generating function in the left hand side of (\ref{eq_mgf_goal}). The details are given in Section \ref{sec_pf_lb_general}, with the key lemmas collected in Section \ref{sec_lemmak_lb}.
    \item Otherwise, we place ourselves in the \textit{very dense} setting and assume without loss of generality that $k_1 \geq c n_1$ for a constant $c \in (0,1)$. In this case, the problem roughly reduces to the sparse signal detection problem in a standard Gaussian sequence model. In Section \ref{sec_pf_lb_dense}, we prove that there exists a constant $c_\mu > 0$ such that if $\mu^2 \leq c_\mu \frac{n_1}{k_1^2}\log\big(1 + \frac{n_2}{k_2^2}\big)$, then (\ref{eq_mgf_goal}) holds. We then show that $\frac{n_1}{k_1^2}\log\big(1 + \frac{n_2}{k_2^2}\big) \asymp R$ in this setting, which completes the proof.
\end{enumerate}

\subsection{Proof of lower bound in the general sparsity setting}\label{sec_pf_lb_general}
Throughout this proof, we assume that there exist sufficiently small constants $c_1, c_2 \in (0,1)$ which depend on $\eta$ such that $k_1 \leq c_1n_1$ and $k_2 \leq c_2n_2$. In this case, $\frac{k_1}{n_1 - k_1} \asymp \frac{k_1}{n_1}$ and $\frac{k_2}{n_2 - k_2} \asymp \frac{k_2}{n_2}$, and it suffices to control $\EE[\exp(\mu^2XY)]$ for $X \sim \text{Bin}(k_1, \frac{k_1}{n_1})$ and $Y \sim \text{Bin}(k_2, \frac{k_2}{n_2})$. We will also assume that $\frac{n_1}{k_1} \geq e \log(\frac{n_2}{k_2})$. In fact, this assumption may be made without loss of generality due to Lemma \ref{lem_2.6}. Recall that we aim to show that there exists a constant $c_\mu > 0$ such that if $\mu^2 \leq c_\mu R$, then $\EE[\exp(\mu^2XY)] \leq 1 + \eps$. We divide our analysis into two cases.

\bigskip 

\noindent\underline{Case 1:} Suppose that $k_1^2 \geq (2e)^{-4}n_1k_2$. Let $\Ct \geq 1$ and $c_{1, \mu} > 0$ be the constants obtained from applying Lemma \ref{lem_simplify_dense_allrates} with $\alpha = \eps/2$. We can write
\begin{align*}
\EE[\exp(\mu^2XY)] &= \EE[\exp(\mu^2XY)\one(X \leq \Ct \frac{k_1^2}{n_1})] + \EE[\exp(\mu^2XY)\one(X > \Ct \frac{k_1^2}{n_1})]
\end{align*}
By Lemma \ref{lem_simplify_dense_allrates}, if $\mu^2 \leq c_{1, \mu}R$, then $\EE[\exp(\mu^2XY)\one(X > \Ct \frac{k_1^2}{n_1})] < \eps/2$.
Now we turn our attention to the first term. Suppose that $\Ct \frac{k_1^2}{n_1} < 1$. Then
\begin{align*}
    \EE[\exp(\mu^2XY)\one(X \leq \Ct \frac{k_1^2}{n_1})] &= \bbP(X = 0) \\
    &\leq 1,
\end{align*}
and it holds that $\EE[\exp(\mu^2XY)] \leq 1+\eps/2$ and the proof is complete. Otherwise, suppose that $\Ct \frac{k_1^2}{n_1} \geq 1$. Then by Lemma \ref{lemma:lb-principal-term}, there exists a constant $c_{2, \mu} > 0$ such that if $\mu^2 \leq c_{2, \mu}R$, it holds $\EE[\exp(\mu^2XY)\one(X \leq \Ct \frac{k_1^2}{n_1})] < 1 + \eps/2$. Letting $c_\mu = \min(c_{1, \mu}, c_{2, \mu})$, it follows that if $\mu^2 \leq c_\mu R$, then
\[\EE[\exp(\mu^2XY)] \leq 1+\eps.\]
The proof in this case is complete.

\bigskip 

\noindent \underline{Case 2:} Suppose that there exists a constant $\bar{c} \in (0, (2e)^{-4})$ such that $k_1^2 \leq \bar{c}n_1k_2$. We split our analysis into two sub-cases. Assume first that $k_2 < k_1\log\big(\frac{n_1k_2}{k_1^2}\big)$. For a constant $C_* \geq 1$ whose value will be determined later, we form the partition
\begin{align*}
    \EE[\exp(\mu^2XY)] &= \EE[\exp(\mu^2XY)\one(X \leq \Ct \frac{k_1^2}{n_1})] \\
    &\quad +  \EE[\exp(\mu^2XY)\one(\Ct \frac{k_1^2}{n_1} < X < k_2^{-1}\log\big(\frac{n_1k_2}{k_1^2}\big))] \\
    &\quad + \EE[\exp(\mu^2XY)\one(X \geq \Ct \frac{k_1^2}{n_1} \lor k_2^{-1}\log\big(\frac{n_1k_2}{k_1^2}\big) )] \\
    &= \text{I}(C_*) + \text{II}(C_*) + \text{III}(C_*).
\end{align*}
By Lemma \ref{lem_simplification_rate_max_test}, there exist constants $C_{1,*} \geq 1$ and $c_{1, \mu} > 0$ such that if $\mu^2 \leq c_{1, \mu}R$, then $\text{III}(C_{1,*}) < \eps/3$. Suppose that $\lceil C_{1,*}\frac{k_1^2}{n_1}\rceil \geq \lfloor k_2^{-1}\log \big(\frac{n_1k_2}{k_1^2}\big)\rfloor$. In this case, we take $C_* = C_{1,*}$, which gives us $\text{II}(C_*) = 0$. If $C_* \frac{k_1^2}{n_1} < 1$, it immediately follows that $\text{I}(C_*) \leq 1$ from an elementary calculation in the proof of Case 1, and thus $\EE[\exp(\mu^2XY)] \leq 1 + \eps/3$ which completes the proof. Otherwise, if $C_* \frac{k_1^2}{n_1} \geq 1$, then by Lemma \ref{lemma:lb-principal-term} there exists a constant $c_{2,\mu} > 0$ such that if $\mu^2 \leq c_{2,\mu}R$, then $\text{I}(C_*) \leq 1 + \frac23\eps$. In this case, we take $c_\mu = \min(c_{1, \mu}, c_{2, \mu})$ and for $\mu^2 \leq c_{\mu}R$ it holds that $\EE[\exp(\mu^2XY)]\leq 1 + \eps$. This completes the proof in the case $\lceil C_{1,*}\frac{k_1^2}{n_1}\rceil \geq \lfloor k_2^{-1}\log \big(\frac{n_1k_2}{k_1^2}\big)\rfloor$.

Now suppose that $\lceil C_{1,*}\frac{k_1^2}{n_1}\rceil \leq \lfloor k_2^{-1}\log \big(\frac{n_1k_2}{k_1^2}\big)\rfloor$. From here, we consider two further subcases. First, suppose that $2e\frac{k_1^2}{n_1} \geq \frac{k_2^2}{n_2}$. Then, by Lemma \ref{lem_simplify_s1d1big_truncchi2}, there exist constants $C_{2,*} \geq 1$ and $c_{2,\mu}$ such that if $\mu^2 \leq c_{2,\mu}R$, then $\text{II}(C_{2,*}) < \eps/3$. Now we take $C_* = \max(C_{1,*}, C_{2,*})$. If $C_* \frac{k_1^2}{n_1} < 1$, it immediately follows that $\text{I}(C_*) \leq 1$, and thus $\EE[\exp(\mu^2XY)] \leq 1 + \frac23\eps$ which completes the proof. Otherwise, if $C_* \frac{k_1^2}{n_1} \geq 1$, then by Lemma \ref{lemma:lb-principal-term} there exists a constant $c_{3,\mu} > 0$ such that if $\mu^2 \leq c_{3,\mu}R$, then $\text{I}(C_*) \leq 1 + \eps/3$. We then take $c_\mu = \min(c_{1, \mu}, c_{2, \mu}, c_{3,\mu})$ and for $\mu^2 \leq c_{\mu}R$ it holds that $\EE[\exp(\mu^2XY)]\leq 1 + \eps$. Next, suppose that $2e\frac{k_1^2}{n_1} \leq \frac{k_2^2}{n_2}$. In this case, we partition $\text{II}(C_*)$ as follows:
\begin{align*}
    \text{II}(C_*) &= \EE[\exp(\mu^2XY)\one(\Ct \frac{k_1^2}{n_1} < X < k_2^{-1}\log\big(\frac{n_1k_2}{k_1^2}\big))] \\
    &= \EE\left[\exp\big(\mu^2XY\big)\one\bigg(1 \vee \Ct \frac{k_1^2}{n_1} \leq X \leq \frac{\Ct k_2^2/n_2}{\log\big(\frac{n_1k_2^2}{k_1^2n_2}\big)}\wedge k_1\bigg)\right] \\
    &\quad +\EE\bigg[\exp(\mu^2XY)\one\bigg( 1 \vee \Ct \frac{k_1^2}{n_1} \vee  \frac{\Ct k_2^2/n_2}{\log\big(\frac{n_1k_2^2}{k_1^2n_2}\big)} \leq X \leq \frac{k_2}{\log\big(\frac{n_1k_2}{k_1^2}\big)}\bigg)\bigg] \\
    &= \text{II}^{(a)}(C_*) + \text{II}^{(b)}(C_*).
\end{align*}
By Lemmas \ref{lem_simplify_s2d2big_maxlin} and \ref{lem_simplify_s2d2big_truncchi2}, there exist constants $C_{2,*}, C_{3,*} \geq 1$ and $c_{2,\mu}, c_{3,\mu} > 0$ such that if $\mu^2 \leq \min(c_{2,\mu}, c_{3,\mu})R$, then $\text{II}^{(a)}(C_{2,*}) + \text{II}^{(b)}(C_{3,*}) < \eps/3$. We take $C_* = \max(C_{1,*}, C_{2,*}, C_{3,*})$. If $C_* \frac{k_1^2}{n_1} < 1$, then $\text{I}(C_*) \leq 1$ and thus $\EE[\exp(\mu^2XY)] \leq 1 + \frac23\eps$ which completes the proof. Otherwise, if $C_* \frac{k_1^2}{n_1} \geq 1$, then by Lemma \ref{lemma:lb-principal-term} there exists a constant $c_{4,\mu} > 0$ such that if $\mu^2 \leq c_{4,\mu}R$, then $\text{I}(C_*) \leq 1 + \eps/3$. We then take $c_\mu = \min(c_{1, \mu}, c_{2, \mu}, c_{3,\mu}, c_{4,\mu})$ and for $\mu^2 \leq c_{\mu}R$ it holds that $\EE[\exp(\mu^2XY)]\leq 1 + \eps$. This completes the proof in the case $k_2 < k_1 \log \big(\frac{n_1k_2}{k_1^2}\big)$.

We now turn our attention to the case $k_2 \geq k_1 \log \big(\frac{n_1k_2}{k_1^2}\big)$. In this case, we perform the partition
\begin{align*}
    \EE[\exp(\mu^2XY)] &= \EE[\exp(\mu^2XY)\one(X \leq \Ct \frac{k_1^2}{n_1})] \\
    &\quad +  \EE[\exp(\mu^2XY)\one(\Ct \frac{k_1^2}{n_1} < X \leq k_1\big))] \\
    &= \text{I}(C_*) + \text{II}(C_*)
\end{align*}
The remainder of the proof follows exactly as in the $k_2 < k_1 \log \big(\frac{n_1k_2}{k_1^2}\big)$ setting. If $2e\frac{k_1^2}{n_1} \geq \frac{k_2^2}{n_2}$, we control $\text{II}(C_*)$ using Lemma \ref{lem_simplify_s1d1big_truncchi2}; otherwise, we use Lemmas \ref{lem_simplify_s2d2big_maxlin} and \ref{lem_simplify_s2d2big_truncchi2}. We then control $\text{I}(C_*)$ using Lemma \ref{lemma:lb-principal-term} if needed. We omit the details for brevity. The proof of the lower bound in the general sparsity setting is complete.
\subsection{Proof of lower bound in the very dense setting}\label{sec_pf_lb_dense}
Now suppose that there exists a constant $c > 0$ such that $k_1 > cn_1$ and that $k_2 < \bar{c}n_2$ for a constant $\bar{c} < 1$. This allows us to consider $Y \sim \text{Bin}(k_2, \frac{k_2}{n_2})$ rather than $Y \sim \text{Bin}(k_2, \frac{k_2}{n_2 - k_2})$. If, on the other hand, $k_2 > cn_2$, the steps of the proof follow identically by plugging in $Y \leq k_2$ and noticing $\frac{n_2}{k_2^2} \asymp \log(1 + \frac{n_2}{k_2^2})$. We omit the details for brevity. Suppose that $\mu^2 \leq c\frac{n_1}{k_1^2}\log\big(1 + c'\frac{n_2}{k_2^2}\big)$ for $c' = \log(1 + \eps)$. By the definition of $c$, it immediately holds that
\[\mu^2 \leq c\frac{n_1}{k_1^2}\log\big(1 + c'\frac{n_2}{k_2^2}\big) \leq \frac{1}{k_1}\log\big(1 + c'\frac{n_2}{k_2^2}\big)\]
By direct calculation, we have
\begin{align*}
    \EE[\exp(\mu^2XY)] &\leq \EE[\exp(\mu^2k_1Y)] \quad \text{(since $X \leq k_1$ almost surely)} \\
    &\leq \EE\Big[\exp\Big(\log\big(1+c'\frac{n_2}{k_2^2}\big)Y\Big)\Big] \quad \text{(since $\mu^2\leq \frac{1}{k_1}\log\big(1 + c'\frac{n_2}{k_2^2}\big)$)} \\
    &= \Big(1 + \frac{k_2}{n_2}\big(e^{\log\big(1 + c'\frac{n_2}{k_2^2}\big)} - 1\big)\Big)^{k_2} \\
    &= \big(1 + \frac{c'}{k_2}\big)^{k_2} \\
    &\leq e^{c'} \\
    &= 1 + \eps.
\end{align*}

It just remains to show that $\frac{n_1}{k_1^2}\log\big(1 + c'\frac{n_2}{k_2^2}\big) \asymp R$. Let $C' = \frac1c$. First, note that $\frac{n_1}{k_1^2} \leq \frac{C'}{k_1} \leq C'$, so hence if $R$ is defined with $C \geq C'$, it holds $R \leq \frac{n_1}{k_1^2}\log\big(1 + c'\frac{n_2}{k_2^2}\big)$ by definition. It remains to show that $\frac{n_1}{k_1^2}\log\big(1 + c'\frac{n_2}{k_2^2}\big) \lesssim R$. Since $\frac{n_1}{k_1} \leq C'$, we have
\begin{align*}
    \frac{n_1}{k_1^2}\log\big(1 + c'\frac{n_2}{k_2^2}\big) &\leq \frac{C'}{k_1}\log\big(1 + c'\frac{n_2}{k_2^2}\big) \\
    &\leq  \frac{C'}{k_1}\log\left(1 + c'\frac{n_2}{k_2^2}\log \left(e {n_1 \choose k_1}\right)\right) \\
    &= \psi_{12} \\
    &\leq \psi_{12} + \psi_{21}.
\end{align*}
Furthermore, using the inequality $\log(1 + x) \leq x$ for any $x > 0$, we have
\begin{align*}
    \frac{n_1}{k_1^2}\log\big(1 + c'\frac{n_2}{k_2^2}\big) 
    &\leq c'\frac{n_1n_2}{k_1^2k_2^2} \\
    &\leq c' C' \frac{n_2}{k_1k_2^2} \\
    &\leq C''\frac{n_2}{k_2^2}\log\Big(1 + \frac{1}{k_1}\Big) \quad \text{(since $x \lesssim \log(1 + x)$ for $x < 1$)}\\
    &\leq C''\frac{n_2}{k_2^2}\log\Big(1 + \frac{n_1}{k^2_1}\Big) \\
    &\leq C'' \phi_{21}.
\end{align*}
Clearly, $\frac{n_1}{k_1^2}\log\big(1 + c'\frac{n_2}{k_2^2}\big) \leq \phi_{12}$. Therefore we have $\frac{n_1}{k_1^2}\log\big(1 + c'\frac{n_2}{k_2^2}\big) \lesssim R$, and the proof is complete.
\subsection{Lemmas for proof of lower bound}\label{sec_lemmak_lb}
\begin{lemma}\label{lemma:lb-principal-term}
    Let $C > 0$ be a constant such that that $k_1^2 \geq n_1 / C$. Then for any $\alpha > 0$, there exists a constant $c_{\mu} > 0$ such that  if $\mu^2 \leq c_{\mu}\frac{n_1}{k_1^2}\log\big(1 + \frac{n_2}{k_2^2}\big)$, then
    \[\EE\Big[\exp(\mu^2XY)\one\big(X \leq C \frac{k_1^2}{n_1}\big)\Big] < 1 + \alpha.\]
\end{lemma}

\begin{proof}
   By direct calculation, we have
    \begin{align*}
        \EE\Big[\exp(\mu^2XY)\one\big(X \leq C \frac{k_1^2}{n_1}\big)\Big] &\leq \EE\Big[\exp(\mu^2C \frac{k_1^2}{n_1}Y)\Big] \\ 
        &= \Big(1 + \frac{k_2}{n_2}\big(e^{\mu^2C (k_1^2/n_1)} - 1\big)\Big)^{k_2} \\ 
        &\leq \exp\Big(\frac{k_2^2}{n_2}\big(e^{\mu^2C (k_1^2/n_1)} - 1\big)\Big).
    \end{align*}
    Now since $\mu^2 \leq c_{\mu}\frac{n_1}{k_1^2}\log\big(1 + \frac{n_2}{k_2^2}\big)$, we have 
    \begin{align*}
        \exp\Big(\frac{k_2^2}{n_2}\big(e^{\mu^2C (k_1^2/n_1)} - 1\big)\Big) &\leq \exp\Big(\frac{k_2^2}{n_2}\big(e^{c_{\mu}\frac{n_1}{k_1^2}\log\big(1 + \frac{n_2}{k_2^2}\big)C (k_1^2/n_1)} - 1\big)\Big) \\ 
        &= \exp\Big(\frac{k_2^2}{n_2}\big(\big(1 + \frac{n_2}{k_2^2}\big)^{c_{\mu}C} - 1\big)\Big) \\ 
        &\leq \exp\Big(\frac{k_2^2}{n_2}\big(c_{\mu}C\frac{n_2}{k_2^2}\big)\Big) \\
        &= \exp\big(c_\mu C\big)
    \end{align*}
    where the second inequality holds for $c_\mu < C^{-1}$, by the inequality $(1+x)^y \leq 1+y x$ which holds for any $x\geq 0$ and $y \in [0,1]$. This final expression is at most $1 + \alpha$ for $c_\mu$ taken sufficiently small.
\end{proof}

For any $C \geq 1$, we define the set of indices
\begin{align}
    \bar{\cA}_C = \left\{ \Big\lceil C\frac{k_1^2}{n_1}\Big\rceil, \dots , k_1\right\}\label{eq_def_A_C}
\end{align}
The results in this section will depend largely on the following lemma.

\begin{lemma}\label{lemma:geom-series}
    Suppose that $k_1 \leq \frac12n_1$. Then for any $\alpha > 0$, there exist constants $\Ct \geq 1$ and $c_\mu > 0$ such that for any $\cA \subset \bar{\cA}_{\Ct}$, if
    \[\mu^2 \leq \min_{k \in \cA}\frac{1}{k}\log\Bigg(1 + \frac{n_2}{k_2}\bigg(\exp\Big[\frac{c_\mu k}{k_2}\log\big(\frac{kn_1}{2ek_1^2}\big)\Big] - 1\bigg)\Bigg)\]
    then 
    \[\EE\Big[\exp(\mu^2XY)\one\Big(X \in \cA \Big)\Big] < \alpha.\]

\end{lemma}

\begin{proof}
    We have 
    \begin{align*}
        \EE\Big[\exp(\mu^2XY)\one\Big(X  \in \cA\Big)\Big] &= \sum_{k \in \cA}\EE\Big[\exp(\mu^2kY)\Big]\Pr(X = k) \\
        &= \sum_{k \in \cA}\Big(1 + \frac{k_2}{n_2}\big(e^{k\mu^2} - 1\big)\Big)^{k_2}\Pr(X = k) \\
        &\leq \sum_{k \in \cA}\Big(1 + \frac{k_2}{n_2}\big(e^{k\mu^2} - 1\big)\Big)^{k_2}\Big(\frac{k^2_12e}{kn_1}\Big)^ke^{-k_1^2/n_1}
    \end{align*}
    where the inequality follows from Lemma \ref{lemma:probXk}. Since we have assumed that
     \[\mu^2 \leq \min_{k \in \cA}\frac{1}{k}\log\Bigg(1 + \frac{n_2}{k_2}\bigg(\exp\Big[\frac{c_\mu k}{k_2}\log\big(\frac{kn_1}{2ek_1^2}\big)\Big] - 1\bigg)\Bigg),\]
    we can substitute this into the above calculation to obtain
    \begin{align*}
        &\sum_{k \in \cA}\Big(1 + \frac{k_2}{n_2}\big(e^{k\mu^2} - 1\big)\Big)^{k_2}\Big(\frac{k^2_12e}{kn_1}\Big)^k \\ 
        &\leq \sum_{k \in \cA}\Bigg(\Big(\frac{kn_1}{k_1^22e}\Big)^{c_\mu}\Bigg)^k\Bigg(\frac{k^2_12e}{kn_1}\Bigg)^k \\ 
        &= \sum_{k \in \cA}\Bigg(\Big(\frac{k^2_12e}{kn_1}\Big)^{1-c_\mu}\Bigg)^k.
    \end{align*}
    Since $\cA \subset \bar{\cA}_{\Ct}$, for every $k \in \cA$ we have     
    $$\frac{k_1^2}{kn_1} \leq \frac{1}{\Ct}.$$
    Therefore, by taking $\Ct$ sufficiently large, we can control the above sum with a geometric series
    \begin{align*}
        &\sum_{k \in \cA}\Bigg(\Big(\frac{k^2_12e}{kn_1}\Big)^{1-c_\mu}\Bigg)^k \\
        &\leq \sum_{k \in \cA}\Bigg(\Big(\frac{2e}{\Ct}\Big)^{1-c_\mu}\Bigg)^k \\
        &\leq \frac{\tilde c}{1 - \tilde c}
    \end{align*}
    where $\tilde c = (2e/\Ct)^{1 - c_\mu}$. This quantity can be made arbitrarily small by taking $c_\mu$ sufficiently small and $\Ct$ sufficiently large. This completes the proof.
\end{proof}

\begin{lemma}\label{lem_max_test_lower_bound}
    Suppose that there exists a constant $\bar{c} \in (0,(2e)^{-4})$ such that $k_1^2 \leq \bar{c}n_1k_2$, and $k_1 \leq \frac12n_1$. 
    Furthermore, suppose that $k_2 < k_1\log\big(\frac{n_1k_2}{k_1^2}\big)$. Then for any $\alpha > 0$, there exist  constants $c_\mu > 0$ and $C_* \geq 1$ such that if 
    $$\mu^2 \leq c_\mu\frac{1}{k_2}\log\left(\frac{k_2n_1}{2ek_1^2} \log\left(\frac{n_2}{k_2}\right)\right),$$
    then 
    \[\EE\left[\exp(\mu^2XY)\one\Big(X \geq \frac{k_2}{\log\big(\frac{n_1k_2}{k_1^2}\big)}\vee   \frac{\Ct k_1^2}{n_1}\Big)\right] < \alpha.\]
    If we further assume that $\frac{n_2}{k_2} \geq 2(1 - e^{-c_\mu})^{-1}$ and that $k_1 < k_2\log(n_2/k_2)$, then if $\mu^2 \leq \frac{c_\mu}{k_2}\log\big(\frac{n_1}{2ek_1}\big) + \frac{1}{k_1}\log\big((1-e^{-c_\mu})\frac{n_2}{k_2}\big)$ it holds
\[\EE\left[\exp(\mu^2XY)\one\Big(X > \frac{k_2}{\log\big(\frac{n_1k_2}{k_1^2}\big)}\vee   \frac{\Ct k_1^2}{n_1}\Big)\right] < \alpha.\]
\end{lemma}

\begin{proof}[Proof of Lemma~\ref{lem_max_test_lower_bound}]
     Note that the set $\cA = \{1 \vee \big\lceil \frac{k_2}{\log\big((n_1k_2)/k_1^2\big)}\big\rceil \vee \big\lceil \Ct \frac{k^2_1}{n_1}\big\rceil, ..., k_1\}$ is a subset of $\bar{\cA}_{\Ct}$, and to prove the claim it suffices to show
    \[\mu^2 \leq \min_{k \in \cA}\frac{1}{k}\log\Bigg(1 + \frac{n_2}{k_2}\bigg(\exp\Big[\frac{c_\mu k}{k_2}\log\big(\frac{kn_1}{2ek_1^2}\big)\Big] - 1\bigg)\Bigg)\]
    and invoke Lemma \ref{lemma:geom-series}. To this end, for $k \in \cA$ we define 
    \begin{equation}\label{eq:g}
    g(k) = \frac{1}{k}\log\Bigg(1 + \frac{n_2}{k_2}\bigg(\exp\Big[\frac{c_\mu k}{k_2}\log\big(\frac{kn_1}{2ek_1^2}\big)\Big] - 1\bigg)\Bigg).
    \end{equation}
    Notice that for $k \in \cA$, we have 
    \begin{align*}
        \frac{c_\mu k}{k_2}\log\big(\frac{kn_1}{2ek_1^2}\big) &\geq \frac{c_\mu}{\log\big(\frac{k_2n_1}{k_1^2}\big)}\log\Big(\frac{k_2n_1}{2ek_1^2\log\big(\frac{k_2n_1}{k_1^2}\big)}\Big) \\
        &= c_\mu\Big(1 - \frac{\log\big(2e\log(\frac{n_1k_2}{k_1^2})\big)}{\log\big(\frac{n_1k_2}{k_1^2}\big)}\Big).
    \end{align*}
    Recall that there exists a constant $\bar{c} \in (0, (2e)^{-4})$ such that $k_1^2 \leq \bar{c}n_1k_2$. In particular, this implies that $n_1k_2/k_1^2 \geq 16e^4$. Using this, as well as the bound $\log (x) / x \leq 1/2$ for $x > 1$, we have 
    \begin{align*}
        \frac{\log\big(2e\log(\frac{n_1k_2}{k_1^2})\big)}{\log\big(\frac{n_1k_2}{k_1^2}\big)} &= \frac{\log(2e)}{\log\big(\frac{n_1k_2}{k_1^2}\big)} + \frac{\log\big(\log(\frac{n_1k_2}{k_1^2})\big)}{\log\big(\frac{n_1k_2}{k_1^2}\big)} \\
        &\leq \frac{\log(2e)}{4\log(2e)} + \frac12 \\
        &= \frac34.
    \end{align*}
    Combining this with our calculation above, we have that, for any $k \in \cA$,
    \begin{align*}
        \frac{c_\mu k}{k_2}\log\big(\frac{kn_1}{2ek_1^2}\big) &\geq \frac{c_\mu}{4} \\ 
        &> 0.
    \end{align*}
    Thus, $\exp\big(\frac{c_\mu k}{k_2}\log\big(\frac{kn_1}{2ek_1^2}\big)\big)$ is bounded away from 1, and hence  for $c = 1 - e^{-c_\mu}$ we have
    \[f(k) =: \frac{1}{k}\log\Bigg(c\frac{n_2}{k_2}\exp\Big[\frac{c_\mu k}{k_2}\log\big(\frac{kn_1}{2ek_1^2}\big)\Big]\Bigg) \leq g(k)\]
    for each $k \in \cA$. So the prove the claim, it suffices to show that $\mu^2$ is at most the minimum value of $f(k)$ over $\cA$. We can write $f(k)$ as 
    \[f(k) = \frac{1}{k}\log\Big(c\frac{n_2}{k_2}\Big) + \frac{c_\mu}{k_2}\log\Big(\frac{kn_1}{2ek_1^2}\Big).\]
    By direct calculation, we have
    \[\frac{\D f}{\D k}(k) = -\frac{\log\big(c\frac{n_2}{k_2}\big)}{k^2} + \frac{c_\mu}{k_2k}\]
    from which we deduce that $f$ is minimized at $k^* = c_\mu^{-1}k_2\log\big(c\frac{n_2}{k_2}\big)$ and decreasing for $k < k^*$. Then if $\mu^2 < c_\mu\frac{1}{k_2}\log\big(\frac{k_2n_1\log(n_2/k_2)}{k_1^2}\big)$, it follows that 
    \begin{align*}
        \mu^2 &< c_\mu\frac{1}{k_2}\log\big(\frac{k_2n_1\log(n_2/k_2)}{2ek_1^2}\big) \\
        &\leq c_\mu\frac{1}{k_2}\Big(\log\big(\frac{k_2n_1\log(n_2/k_2)}{2e c_\mu k_1^2}\big) + 1\Big) \\
        &= f(k^*) \\
        &\leq \min_{k \in \cA}f(k) \\
        &\leq \min_{k\in \cA}g(k)
    \end{align*}
    and we may apply Lemma \ref{lemma:geom-series} to complete the proof. Now if $c\frac{n_2}{k_2} \geq 2$ and $k_1 < k_2\log(n_2/k_2)$, then, for some sufficiently small $c_\mu$, we have
    \begin{align*}
        k_1 &< k_2\log(n_2/k_2) \\
        &< \frac{k_2}{c_\mu}\frac{\log (2)}{\log (2) + \log \big(\frac{2e}{c}\big)}\log\big(\frac{n_2}{k_2}\big) \\
        &\leq \frac{k_2\log\big(\frac{cn_2}{k_2}\big)}{c_\mu\log\big(\frac{2en_2}{k_2}\big)}\log\big(\frac{n_2}{k_2}\big) \\
        &\leq \frac{k_2}{c_\mu}\log\big(c\frac{n_2}{k_2}\big) \\
        &= k^*
    \end{align*}
    where the third inequality uses the fact that $x \mapsto \frac{x}{x + c}$ is increasing in $x$. Therefore since $k_1 < k^*$, the minimum value of $f(k)$ over the set $\cA$ is attained at $k = k_1$. Thus for $\mu^2 < \frac{c_\mu}{k_2}\log\big(\frac{n_1}{2ek_1}\big) + \frac{1}{k_1}\log\big(c\frac{n_2}{k_2}\big)$, we have 
    \begin{align*}
        \mu^2 &< \frac{c_\mu}{k_2}\log\big(\frac{n_1}{2ek_1}\big) + \frac{1}{k_1}\log\big(c\frac{n_2}{k_2}\big) \\ 
        &= f(k_1) \\
        &= \min_{k \in \cA}f(k) \\
        &\leq \min_{k\in \cA}g(k).
    \end{align*}
    Applying Lemma \ref{lemma:geom-series} completes the proof.    
\end{proof}

%
%
%
\begin{lemma}\label{lem_lb_s1d1big_truncchi2}
    Suppose that $k_1^2 \leq \bar{c}n_1k_2$ for some $\bar{c} \in (0,(2e)^{-4})$ and $k_1 \leq \frac12n_1$. Furthermore, suppose that $\frac{k^2_2}{n_2} \leq 2e\frac{k_1^2}{n_1}$. For any small constant $c > 0$, we define the quantity
    \begin{align*}
        M(c) = \begin{cases}
            \frac{1}{k_1}\log\Big(1 + c \frac{n_2k_1}{k_2^2}\log(\frac{n_1}{2ek_1})\Big) & \text{ if } k_1 \leq \frac{k_2}{\log\big(\frac{n_1k_2}{k_1^2}\big)}\\
            \frac{1}{k_2}\log\big(\frac{n_1k_2}{k_1^2}\big)\log\Big(1 + c \frac{n_2}{k_2\log\big(\frac{n_1k_2}{k_1^2}\big)}\log\big(\frac{n_1k_2}{2ek_1^2\log(\frac{n_1k_2}{k_1^2})}\big)\Big) & \text{ otherwise.}
        \end{cases}
    \end{align*}
    Then for any $\alpha > 0$, there exist constants $c_\mu > 0$ and $C_* \geq 1$ such that if $\mu \leq M(c_\mu)$ and $1 \vee \Ct \frac{k_1^2}{n_1}\leq k_1 \wedge\frac{k_2}{\log\big(\frac{n_1k_2}{k_1^2}\big)}$, then 
    \[\EE\Big[\exp(\mu^2XY)\one\Big( 1 \vee \Ct \frac{k_1^2}{n_1} \leq X \leq k_1\wedge\frac{k_2}{\log\big(\frac{n_1k_2}{k_1^2}\big)}\Big)\Big] < \alpha.\]
    Moreover, it holds that    
    $$2M(c_\mu) \geq \frac{1}{k_1}\log\Big(1 + c_\mu \frac{n_2k_1}{k_2^2}\log(\frac{n_1}{2ek_1})\Big) + \frac{\log\big(\frac{n_1k_2}{k_1^2}\big)}{k_2}\log\Big(1 + c_\mu \frac{n_2}{k_2\log\big(\frac{n_1k_2}{k_1^2}\big)}\log\big(\frac{n_1k_2}{2ek_1^2\log(\frac{n_1k_2}{k_1^2})}\big)\Big).$$
\end{lemma}

\begin{proof}[Proof of Lemma~\ref{lem_lb_s1d1big_truncchi2}]
    Define $\cA = \{ 1 \vee \lceil\Ct \frac{k_1^2}{n_1} \rceil, ...k_1 \wedge \big\lfloor\frac{k_2}{\log\big(\frac{n_1k_2}{k_1^2}\big)} \big\rfloor \}$. Clearly $\cA \subset \bar{\cA}_{\Ct}$, and hence it suffices to show that $\mu^2 \leq \min_{k \in \cA}g(k)$ where $g(k)$ is defined as in (\ref{eq:g}). Using the inequality $e^x \geq 1 + x $, we have 
    \begin{align*}
        g(k) &= \frac{1}{k}\log\Bigg(1 + \frac{n_2}{k_2}\bigg(\exp\Big[\frac{c_\mu k}{k_2}\log\big(\frac{kn_1}{2ek_1^2}\big)\Big] - 1\bigg)\Bigg) \\
        &\geq \frac{1}{k}\log\Bigg(1 + c_\mu\frac{n_2}{k^2_2}k\log\Big(\frac{kn_1}{2ek_1^2}\Big)\Bigg) \\
        &=: f^{(1)}(k).
    \end{align*}
    
%
%
    We define $\widetilde{\mathcal{A}} = \{ 1 \vee \lceil\Ct \frac{k_1^2}{n_1} \rceil, \dots \}$ and we will now show that $f^{(1)}(k)$ is decreasing in $k$ over the set $\widetilde{\mathcal{A}}$. To do so, it suffices to consider the function $f(k) = \frac{1}{k}\log\Big(c_\mu\frac{n_2}{k^2_2}k\log\big(\frac{kn_1}{2ek_1^2}\big)\Big)$ and apply Lemma \ref{lemma:log-plus-one}. Direct calculation reveals
    \[\frac{\D f}{\D k}(k) = -\frac{\log\big(c_\mu\frac{n_2}{k_2^2}\big)}{k^2} - \frac{\log(k)}{k^2} + \frac{1}{k^2} - \frac{\log \log \big(\frac{kn_1}{2ek_1^2}\big)}{k^2} + \frac{1}{k^2\log\big(\frac{kn_1}{2ek_1^2}\big)}.\]
    The right-hand side is less than zero if   
    \[1 + \frac{1}{\log\big(\frac{kn_1}{2ek_1^2}\big)} < \log\Big(c_\mu\frac{n_2k}{k_2^2} \log \big(\frac{kn_1}{2ek_1^2}\big)\Big).\]
    
    For $k \in \widetilde{\mathcal{A}}$, it holds 
    \begin{align*}
        \log\Big(c_\mu\frac{n_2k}{k_2^2}\log \big(\frac{kn_1}{2ek_1^2}\big)\Big) 
        &\geq \log\Big(c_\mu \Ct \frac{n_2k_1^2}{k_2^2n_1}\log \big(\frac{\Ct}{2e}\big)\Big) \\
        &\geq \log\Big(\frac{c_\mu C_*}{2e} \log \big(\frac{\Ct}{2e}\big)\Big)\\
        &\geq 2
    \end{align*}
    where the second inequality uses the assumption $\frac{k_2^2}{n_2} \leq 2e\frac{k_1^2}{n_1}$ and the third holds as long as $\Ct$ is taken sufficiently large. Furthermore, for $k \in \widetilde{\mathcal{A}}$ we have 

    \begin{align*}
        1 + \frac{1}{\log \big(\frac{kn_1}{2ek_1^2}\big)} &\leq 1 + \frac{1}{\log\big(\frac{\Ct}{2e}\big)} \\
        &< 2.
    \end{align*}

    Therefore, $f(k)$ is decreasing over $\widetilde{\mathcal{A}}$, and by Lemma \ref{lemma:log-plus-one}, $f^{(1)}(k)$ is decreasing over $\widetilde{\mathcal{A}}$ as well. 
    We also observe that 
    \begin{align*}
        M = \begin{cases} f^{(1)}(k_1) & \text{ if } k_1 \leq \frac{k_2}{\log(\frac{n_1k_2}{k_1^2})}\\
        f^{(1)}(\frac{k_2}{\log(\frac{n_1k_2}{k_1^2})}) & \text{ otherwise,}
        \end{cases}
    \end{align*}
    that is, $M= f^{(1)}\Big(k_1 \land \frac{k_2}{\log(\frac{n_1k_2}{k_1^2})}\Big)$. 
    Moreover, since $f^{(1)}$ is decreasing over $\widetilde{\mathcal{A}}$, we also have
    \begin{align*}
        M &= \max\bigg(f^{(1)}(k_1), ~f^{(1)}\bigg(\frac{k_2}{\log\big(\frac{n_1k_2}{k_1^2}\big)}\bigg)\bigg) \\
        &\geq \frac{1}{2} \left[f^{(1)}(k_1) + ~f^{(1)}\bigg(\frac{k_2}{\log\big(\frac{n_1k_2}{k_1^2}\big)}\bigg)\right]\\
        & = \frac{1}{2} \left[\frac{1}{k_1}\log\Big(1 + c_\mu \frac{n_2k_1}{k_2^2}\log(\frac{n_1}{2ek_1})\Big) + \frac{\log\big(\frac{n_1k_2}{k_1^2}\big)}{k_2}\log\Big(1 + c_\mu \frac{n_2}{k_2\log\big(\frac{n_1k_2}{k_1^2}\big)}\log\big(\frac{n_1k_2}{2ek_1^2\log(\frac{n_1k_2}{k_1^2})}\big)\Big)\right]
    \end{align*}
    as claimed. 
    Using our assumed upper bound on $\mu^2$, we have 
    \begin{align*}
        \mu^2 &\leq M =  f^{(1)}\bigg(k_1 \land \frac{k_2}{\log(\frac{n_1k_2}{k_1^2})}\bigg) \\
        &\leq \min_{k \in \cA}f^{(1)}(k) \\
        &\leq \min_{k\in \cA}g(k).
    \end{align*}
    We conclude the proof by invoking Lemma \ref{lemma:geom-series}. 
\end{proof}

\begin{lemma}\label{lem_lb_s2d2big_maxlin}
    Suppose that $k_1^2 \leq \bar{c}n_1k_2$ for some $\bar{c} \in (0, (2e)^{-4})$ and $k_1 < \frac12 n_1$. 
    Furthermore, suppose that $\frac{k^2_2}{n_2} \geq 2e\frac{k_1^2}{n_1}$ . Then for any $\alpha > 0$, there exist constants $c'_{\mu} > 0$ and $\Ct > 0$ such that if $1 \vee \lceil \Ct \frac{k_1^2}{n_1}\rceil \leq \big\lfloor \frac{\Ct k_2^2/n_2}{\log\big(\frac{n_1k_2^2}{k_1^2n_2}\big)}\big\rfloor\wedge k_1$ and
    \[\mu^2 \leq c'_\mu \frac{n_2}{k_2^2}\log\big(\frac{\Ct}{2e}\vee \frac{n_1}{2ek^2_1}\big),\]
    then
    \[\EE\left[\exp\big(\mu^2XY\big)\one\bigg(1 \vee \Ct \frac{k_1^2}{n_1} \leq X \leq \frac{\Ct k_2^2/n_2}{\log\big(\frac{n_1k_2^2}{k_1^2n_2}\big)}\wedge k_1\bigg)\right] < \alpha.\]
\end{lemma}

\begin{proof}[Proof of Lemma~\ref{lem_lb_s2d2big_maxlin}]
    Let $g(k)$ be defined as in (\ref{eq:g}). For $\cA = \{1 \vee \lceil \Ct \frac{k_1^2}{n_1}\rceil, \dots ,  \Big\lfloor \frac{k_2^2/n_2}{\log\big(\frac{n_1k_2^2}{k_1^2n_2}\big)}\Big\rfloor\wedge k_1\}$, it suffices to show that $\mu^2 \leq \min_{k\in\cA}g(k)$. Using $e^x \geq x + 1,\, \forall x \in \R$, we have 
    \begin{align*}
        g(k) &= \frac{1}{k}\log\Bigg(1 + \frac{n_2}{k_2}\bigg(\exp\Big[\frac{c_\mu k}{k_2}\log\big(\frac{kn_1}{2ek_1^2}\big)\Big] - 1\bigg)\Bigg) \\
        &\geq \frac{1}{k}\log\Bigg(1 + c_\mu\frac{n_2}{k^2_2}k\log\Big(\frac{kn_1}{2ek_1^2}\Big)\Bigg).
    \end{align*}
    For $k \in \cA$,  we can control the term within the logarithm as
    \begin{align*}
        c_\mu\frac{n_2}{k^2_2}k\log\Big(\frac{kn_1}{2ek_1^2}\Big) 
        &\leq c_\mu \Ct \frac{\log\Big(\frac{\Ct n_1k_2^2}{2ek_1^2n_2\log\big(n_1k_2^2/(k_1^2n_2)\big)} \Big)}{\log\big(\frac{n_1k_2^2}{k_1^2n_2}\big)} \\
        &\leq c_\mu\Ct \big(\log (\Ct) + 1 \big) \\
        &< 2c_\mu\Ct \log (\Ct)
    \end{align*}
    where the second inequality follows from $\log(x/\log(x))/\log(x) \leq 1$ for $x \geq 2e$, and the final inequality holds for $\Ct > e$. Therefore, there exists a constant $c \in (0,1)$ which depends on $c_\mu$ and $\Ct$ such that 
    \begin{align*}
        g(k) &\geq \frac{1}{k}\log\Bigg(1 + c_\mu\frac{n_2}{k^2_2}k\log\Big(\frac{kn_1}{2ek_1^2}\Big)\Bigg) \\
        &\geq cc_\mu \Ct \frac1k\frac{n_2}{k_2^2}k\log\big(\frac{kn_1}{2ek_1^2}\big) \\
        &= c'_\mu\frac{n_2}{k_2^2}\log\big(\frac{kn_1}{2ek_1^2}\big) \\
        &=: f(k)
    \end{align*}
    for $c'_\mu = c c_\mu \Ct$. The function $f(k)$ is clearly increasing in $k$, and hence is minimized at the smallest value of $k$ in $\cA$. 
    Now for $\mu^2 \leq c'_\mu \frac{n_2}{k_2^2}\log\big(\frac{\Ct}{2e}\vee \frac{n_1}{2ek^2_1}\big)$, we have
    \begin{align*}
        \mu^2 &\leq c'_\mu \frac{n_2}{k_2^2}\log\big(\frac{\Ct}{2e}\vee \frac{n_1}{2ek^2_1}\big) \\
        &\leq \min_{k \in \cA}f(k) \\
        &\leq \min_{k \in \cA}g(k).
    \end{align*}
    We apply Lemma \ref{lemma:geom-series} to complete the proof.
\end{proof}
\begin{lemma}\label{lem_lb_s2d2big_truncchi2}
    Suppose that $k_1^2 \leq \bar{c}n_1k_2$ for some $\bar{c} \in (0,(2e)^{-4})$ and $k_1 \leq \frac{1}{2e}n_1$. Furthermore, suppose that $\frac{k^2_2}{n_2} \geq 2e\frac{k_1^2}{n_1}$. Then for any $\alpha > 0$, there exist constants $c_\mu > 0$ and $C_* \geq 1$ such that if $1 \vee \Ct \frac{k_1^2}{n_1} \vee  \frac{\Ct k_2^2/n_2}{\log\big(\frac{n_1k_2^2}{k_1^2n_2}\big)} \leq k_1\wedge\frac{k_2}{\log\big(\frac{n_1k_2}{k_1^2}\big)}$ and
    $$\mu^2 \leq \frac12\bigg(\frac{1}{k_1}\log\Big(1 + c_\mu \frac{n_2k_1}{k_2^2}\log(\frac{n_1}{2ek_1})\Big) + \frac{\log\big(\frac{n_1k_2}{k_1^2}\big)}{k_2}\log\Big(1 + c_\mu \frac{n_2}{k_2\log\big(\frac{n_1k_2}{k_1^2}\big)}\log\big(\frac{n_1k_2}{2ek_1^2\log(\frac{n_1k_2}{k_1^2})}\big)\Big)\bigg),$$
    then 
    \[\EE\bigg[\exp(\mu^2XY)\one\bigg( 1 \vee \Ct \frac{k_1^2}{n_1} \vee  \frac{\Ct k_2^2/n_2}{\log\big(\frac{n_1k_2^2}{k_1^2n_2}\big)} \leq X \leq k_1\wedge \frac{k_2}{\log\big(\frac{n_1k_2}{k_1^2}\big)}\bigg)\bigg] < \alpha.\]
\end{lemma}

\begin{proof}[Proof of Lemma~\ref{lem_lb_s2d2big_truncchi2}]
    The structure of this proof is similar to that of Lemma \ref{lem_lb_s1d1big_truncchi2}. Define $\cA = \{ 1 \vee \lceil\Ct \frac{k_1^2}{n_1} \rceil \vee \Big\lceil \frac{\Ct k_2^2/n_2}{\log\big(\frac{n_1k_2^2}{k_1^2n_2}\big)}\Big\rceil, \dots , k_1 \wedge \lfloor\frac{k_2}{\log\big(\frac{n_1k_2}{k_1^2}\big)} \rfloor \}$. With $g(k)$ defined as in (\ref{eq:g}), we have 

    \begin{align*}
        g(k) &= \frac{1}{k}\log\Bigg(1 + \frac{n_2}{k_2}\bigg(\exp\Big[\frac{c_\mu k}{k_2}\log\big(\frac{kn_1}{2ek_1^2}\big)\Big] - 1\bigg)\Bigg) \\
        &\geq \frac{1}{k}\log\Bigg(1 + c_\mu\frac{n_2}{k^2_2}k\log\Big(\frac{kn_1}{2ek_1^2}\Big)\Bigg) \\
        &=: f^{(1)}(k).
    \end{align*}

    As in the proof of Lemma \ref{lem_lb_s1d1big_truncchi2}, will now show that $f^{(1)}(k)$ is decreasing in $k$ over the set $\cA$. It suffices to consider the function $f(k) =: \frac{1}{k}\log\Big(c_\mu\frac{n_2}{k^2_2}k\log\big(\frac{kn_1}{2ek_1^2}\big)\Big)$ and apply Lemma \ref{lemma:log-plus-one}. By direct calculation, we have
    \[\frac{\D f}{\D k}(k) = -\frac{\log\big(c_\mu\frac{n_2}{k_2^2}\big)}{k^2} - \frac{\log(k)}{k^2} + \frac{1}{k^2} - \frac{\log \log \big(\frac{kn_1}{2ek_1^2}\big)}{k^2} + \frac{1}{k^2\log\big(\frac{kn_1}{2ek_1^2}\big)}.\]
    The right hand side is less than zero if 
    \[1 + \frac{1}{\log\big(\frac{kn_1}{2ek_1^2}\big)} < \log\Big(c_\mu\frac{n_2k}{k_2^2}\log \big(\frac{kn_1}{2ek_1^2}\big)\Big)\]
    For $k \in \cA$, it holds 
    \begin{align*}
        \log\Big(c_\mu\frac{n_2k}{k_2^2} \log \big(\frac{kn_1}{2ek_1^2}\big)\Big) 
        &\geq \log\Bigg(\frac{c_\mu \Ct}{\log\big(\frac{n_1k_2^2}{k_1^2n_2}\big)} \log \Big(\frac{\Ct k_2^2n_1}{2ek_1^2n_2\log\big(\frac{k_2^2n_1}{k_1^2n_2}\big)}\Big)\Bigg) \\
        &\geq \log\Big(\frac{c_\mu \Ct }{2} \Big) \\
        &\geq 2
    \end{align*}
    where the second inequality uses that $\frac{\log \log x}{\log x} < \frac12$ for $x > 1$ to conclude 
    $$\frac{\log\big(\frac{\Ct k_2^2n_1}{2en_2k_1^2}\big) - \log \log \big(\frac{k_2^2n_1}{n_2k_1^2}\big)}{\log \big(\frac{n_1k_2^2}{k_1^2n_2}\big)} \geq \frac12,$$
    and the final inequality holds as long as $\Ct$ is taken sufficiently large such that $\Ct \geq 2e^2 / c_\mu$. Furthermore, for $k \in \cA$ we have 

    \begin{align*}
        1 + \frac{1}{\log \big(\frac{kn_1}{2ek_1^2}\big)} &\leq 1 + \frac{1}{\log\big(\frac{\Ct}{2e}\big)} \\
        &< 2
    \end{align*}

    Therefore, $f(k)$ is decreasing over $\cA$, and by Lemma \ref{lemma:log-plus-one}, $f^{(1)}(k)$ is decreasing over $\cA$ as well. Using our assumed upper bound on $\mu^2$, we have 
    \begin{align*}
        \mu^2 &\leq \frac12\bigg(\frac{1}{k_1}\log\Big(1 + c_\mu \frac{n_2k_1}{k_2^2}\log(\frac{n_1}{2ek_1})\Big) + \frac{\log\big(\frac{n_1k_2}{k_1^2}\big)}{k_2}\log\Big(1 + c_\mu \frac{n_2}{k_2\log\big(\frac{n_1k_2}{k_1^2}\big)}\log\big(\frac{n_1k_2}{2ek_1^2\log(\frac{n_1k_2}{k_1^2})}\big)\Big)\bigg) \\
        &= \frac12\Big(f^{(1)}(k_1) + f^{(1)}\big(\frac{k_2}{\log(\frac{n_1k_2}{k_1^2})}\big)\Big) \\
        &\leq \max\Big(f^{(1)}(k_1), f^{(1)}\big(\frac{k_2}{\log(\frac{n_1k_2}{k_1^2})}\big)\Big) \\
        &\leq \min_{k \in \cA}f^{(1)}(k) \\
        &\leq \min_{k\in \cA}g(k).
    \end{align*}
    We conclude the proof by invoking Lemma~\ref{lemma:geom-series}. 
\end{proof}

\begin{lemma}\label{lem_lb_dense_allrates}
    Suppose that $k_1^2 \geq (2e)^{-4}n_1k_2$ and that $k_1 < c'n_1$ for some small enough constant $c'>0$.
    Then for any $\alpha > 0$, there exist constants $c_\mu > 0$ and $\Ct \geq 1$ such that the following claims hold. In what follows, we define $c = 1 - \exp\Big(-c_\mu \Ct (2e)^{-4}\log \Big(\frac{\Ct}{2e}\Big)\Big)$.

    \begin{enumerate}
        \item Suppose that $c_\mu^{-1}k_2\log\big(c\frac{n_2}{k_2}\big) \leq \Ct \frac{k_1^2}{n_1}$. Then if 
        \[\mu^2 < \Ct^{-1}\frac{n_1}{k_1^2}\log\Big(c \frac{n_2}{k_2}\Big) + \frac{c_\mu}{k_2}\log\Big(\frac{\Ct}{2e}\Big),\]
         it holds 
        \[\EE\Big[\exp(\mu^2XY)\one\big(X \geq \lceil  \Ct \frac{k_1^2}{n_1}\rceil\big)\Big] < \alpha.\]
        \item Suppose that $\Ct \frac{k_1^2}{n_1} < c_\mu^{-1}k_2\log\big(c\frac{n_2}{k_2}\big) \leq k_1$. Then if
        \[\mu^2 < \frac{c_\mu}{k_2}\log\Big(\frac{k_2n_1\log(c \frac{n_2}{k_2})}{c_\mu 2e k_1^2}\Big),\]
        it holds 
        \[\EE\Big[\exp(\mu^2XY)\one\big(X \geq \lceil  \Ct \frac{k_1^2}{n_1}\rceil\big)\Big] < \alpha.\]
        \item Suppose that $k_1 < c_\mu^{-1}k_2\log\big(c\frac{n_2}{k_2}\big)$.Then if
        \[\mu^2 < \frac{1}{k_1}\log\big(c \frac{n_2}{k_2}\big) + \frac{c_\mu}{k_2}\log\big(\frac{n_1}{2ek_1}\big),\]
        it holds 
        \[\EE\Big[\exp(\mu^2XY)\one\big(X \geq \lceil  \Ct \frac{k_1^2}{n_1}\rceil\big)\Big] < \alpha.\]
    \end{enumerate}
\end{lemma}

\begin{proof}
    Define $\cA = \{\lceil  \Ct \frac{k_1^2}{n_1}\rceil, ..., k_1 \}$ and $g$ as in (\ref{eq:g}). For $k \in \cA$, it holds 
    \begin{align*}
        \frac{c_\mu k}{k_2}\log\Big(\frac{kn_1}{2ek_1^2}\Big) &\geq \frac{c_\mu \Ct k_1^2}{k_2n_1}\log\Big(\frac{\Ct}{2e}\Big) \\
        &\geq c_\mu \Ct (2e)^{-4}\log \Big(\frac{\Ct}{2e}\Big) \\
        &> 0
    \end{align*}
    where the second inequality follows from the assumption $k_1^2 \geq (2e)^{-4}n_1k_2$ and the final inequality holds for $\Ct > 2e$. This implies that $\exp\Big(\frac{c_\mu k}{k_2}\log\Big(\frac{kn_1}{2ek_1^2}\Big)\Big)$ is bounded away from 1 for $k \in \cA$. Therefore, for $c = 1 - \exp\Big(-c_\mu \Ct (2e)^{-4}\log \Big(\frac{\Ct}{2e}\Big)\Big)$, it holds
    \begin{align*}
        g(k) &= \frac{1}{k}\log\Bigg(1 + \frac{n_2}{k_2}\bigg(\exp\Big[\frac{c_\mu k}{k_2}\log\big(\frac{kn_1}{2ek_1^2}\big)\Big] - 1\bigg)\Bigg)\\
        &\geq \frac{1}{k}\log\Bigg(1 + c\frac{n_2}{k_2}\exp\Big[\frac{c_\mu k}{k_2}\log\big(\frac{kn_1}{2ek_1^2}\big)\Big]\Bigg)\\
        &\geq \frac{1}{k}\log\Bigg(c\frac{n_2}{k_2}\exp\Big[\frac{c_\mu k}{k_2}\log\big(\frac{kn_1}{2ek_1^2}\big)\Big]\Bigg) \\
        &=: f(k).
    \end{align*}
    Therefore, it suffices to show that $\mu^2$ is at most the minimum of $f$ over $\cA$. We write $f(k)$ as 
    \[f(k) = \frac{1}{k}\log\Big(c\frac{n_2}{k_2}\Big) + \frac{c_\mu}{k_2}\log\Big(\frac{kn_1}{2ek_1^2}\Big).\]
    By direct calculation, we have
    \[\frac{\D f}{\D k}(k) = -\frac{\log\big(c\frac{n_2}{k_2}\big)}{k^2} + \frac{c_\mu}{k_2k}\]
    from which we deduce that $f$ is minimized at $k^* = c_\mu^{-1}k_2\log\big(c\frac{n_2}{k_2}\big)$ over $\R^+$, decreasing for $k < k^*$, and increasing for $k > k^*$. We now proceed by cases.

    \begin{enumerate}
        \item Suppose that $c_\mu^{-1}k_2\log\big(c\frac{n_2}{k_2}\big) \leq \Ct \frac{k_1^2}{n_1}$. Then $f$ is increasing over $\mathcal{A}$, and by our assumption on $\mu^2$, it holds
        \begin{align*}\mu^2 &< \Ct^{-1}\frac{n_1}{k_1^2}\log\Big(c \frac{n_2}{k_2}\Big) + \frac{c_\mu}{k_2}\log\Big(\frac{\Ct}{2e}\Big) \\
        &\leq \min_{k \in \cA}f(k) \\
        &\leq \min_{k \in \cA}g(k).
        \end{align*}
        \item Suppose that $\Ct \frac{k_1^2}{n_1} < c_\mu^{-1}k_2\log\big(c\frac{n_2}{k_2}\big) \leq k_1$. Then by our assumption on $\mu^2$, we have
        \begin{align*}
            \mu^2 &< \frac{c_\mu}{k_2}\log\Big(\frac{k_2n_1\log(c \frac{n_2}{k_2})}{c_\mu 2e k_1^2}\Big)\\
        &\leq \min_{k \in \cA}f(k) \\
        &\leq \min_{k \in \cA}g(k).
        \end{align*}
        \item Suppose that $k_1 < c_\mu^{-1}k_2\log\big(c\frac{n_2}{k_2}\big)$. Then $f$ is decreasing over $\mathcal{A}$ and is minimized at $k = k_1$. Therefore by our assumption on $\mu^2$, we have 
        \begin{align*}
            \mu^2 &< \frac{1}{k_1}\log\big(c \frac{n_2}{k_2}\big) + \frac{c_\mu}{k_2}\log\big(\frac{n_1}{2ek_1}\big)\\
        &= \min_{k \in \cA}f(k) \\
        &\leq \min_{k \in \cA}g(k).
        \end{align*}
    \end{enumerate}
    By Lemma \ref{lemma:geom-series}, the proof is complete.
\end{proof}

\subsubsection{Simplification of the rate}
For any $k_1,k_2, n_1, n_2 \in \mathbb N$, we recall the following quantities:
\begin{align}
    \psi(k_1,k_2,n_1,n_2) &= \frac{1}{k_1} \log\left(1+ \frac{n_2k_1}{k_2^2} \log\left(\frac{n_1}{k_1}\right)\right) \label{eq_def_psi}\\[10pt]
    \phi(k_1,k_2, n_1,n_2) &= \begin{cases}
        \frac{n_1}{k_1^2} \log\left(1+\frac{n_2}{k_2^2}\right)&  \text{ if } \frac{n_1}{k_1^2} \leq 1\\
         \infty & \text{ otherwise.}
    \end{cases}\label{eq_def_phi}\\[10pt]
    \beta(k_1,k_2,n_1,n_2) &= \frac{1}{k_1} \log\left(\frac{n_2}{k_2}\right) \mathbf{1}_{\left\{\frac{n_1k_2}{k_1^2} \log\left(\frac{n_2}{k_2}\right)>1\right\}}.\label{eq_def_nu}
\end{align}
To alleviate the notation, we will write
\begin{align*}
    \phi_{12} &= \phi(k_1,k_2, n_1,n_2)\\
    \phi_{21} &= \phi(k_2,k_1, n_2,n_1)\\
    \psi_{12} &= \psi(k_1,k_2, n_1,n_2)\\
    \psi_{21} &= \psi(k_2,k_1, n_2,n_1)\\
    \beta_{12} &= \beta(k_1,k_2, n_1,n_2)\\
    \beta_{21} &= \beta(k_2,k_1, n_2,n_1).
\end{align*}
Below, we show that the lower bound can be rewritten as
\[R := R(k_1, k_2, n_1, n_2) = \Big(\psi_{12} + \psi_{21}\Big) \land \phi_{12}  \land \phi_{21}.\]

The following lemma simplifies the rate emerging from Lemma ~\ref{lem_lb_dense_allrates}.
\begin{lemma}\label{lem_simplify_dense_allrates}
    Assume that $k_1^2 \geq \bar c n_1 k_2$ for some constant $\bar c>0$ and $k_j \leq c' n_j, \,\forall j \in \{1,2\}$ for some sufficiently small constant $c'>0$. 
    Assume also that $\frac{n_1}{k_1} \geq e\log\left(\frac{n_2}{k_2}\right)$. 
    \begin{enumerate}
        \item 
        Then the following two properties hold    
    \begin{align}
        & \Big(\psi_{12} + \psi_{21}\Big) \land \phi_{12}
        \asymp \psi_{21} \land \phi_{12}\label{eq_simplification_dense_case}\\
        & \psi_{21} \land \phi_{12}
        \asymp \Big(\psi_{21} + \beta_{12}\Big) \land \phi_{12}.\label{eq_simplification_dense_case_upper_bound}
    \end{align}
    \item
    It follows that, for any $\alpha>0$, there exists a constant $c_\mu>0$ such that, whenever $\mu^2 \leq c_\mu R$, we have
    \begin{align*}
    \EE\Big[\exp(\mu^2XY)\; \one\Big(X \geq \Ct \frac{k_1^2}{n_1}\Big)\Big] < \alpha.
    \end{align*}
    \item Moreover, it holds that $R \asymp \big(\psi_{21} + \beta_{12}\big) \land \big(\psi_{12} + \beta_{21}\big) \land \phi_{12} \land \phi_{21}$.
    \end{enumerate}
\end{lemma}
\begin{proof}[Proof of Lemma ~\ref{lem_simplify_dense_allrates}]
\phantom{}

\begin{enumerate}
    \item 
For any $k_1,k_2, n_1, n_2$, we let 
\begin{align*}
    \widetilde \phi_{12} = \frac{n_1}{k_1^2} \log\left(1+\frac{n_2}{k_2^2}\right).
\end{align*}
We start by showing that 
\begin{align*}
    \psi_{21} \land \phi_{12} \asymp \psi_{21} \land \widetilde \phi_{12}.
\end{align*}
This is clear if $\frac{n_1}{k_1^2} \leq 1$ by definition of $\phi_{12}$.  
Assume now that $\frac{n_1}{k_1^2} >1$, which implies that $\phi_{12} = \infty$ and $\psi_{21} \land \phi_{12} = \psi_{21}$. 
By the assumption $k_1^2 \geq \bar c n_1 k_2$, we obtain $k_2 \leq \frac{k_1^2}{n_1\bar c} \leq \frac{1}{\bar c}$. 
Therefore, we have
\begin{align*}
    \phi_{12} = \frac{n_1}{k_1^2} \log\left(1+\frac{n_2}{k_2^2}\right) &\geq \frac{n_1}{k_1^2} \log\left(1+\bar c^2 n_2\right)& \\
    &\geq \bar c^2\frac{n_1}{k_1^2} \log\left(1+n_2\right) &\text{ by Lemma~\ref{lem_logk_constants}.(i)}\\
    & \geq \bar c^2\log\left(1+\frac{n_1}{k_1^2} k_2\log\left(\frac{n_2}{k_2}\right)\right)  & \text{ by Lemma~\ref{lem_logk_constants}.(ii)}\\
    & \geq \bar c^2\frac{1}{k_2}\log\left(1+\frac{n_1}{k_1^2} k_2\log\left(\frac{n_2}{k_2}\right)\right) &\\
    & = \bar c^2 \psi_{21}.&
\end{align*}
Therefore, it follows that
\begin{align*}
    \psi_{21} \land \phi_{12} &\asymp \psi_{21} \land \widetilde \phi_{12}\\
    \Big(\psi_{12} + \psi_{21}\Big) \land \phi_{12} &\asymp\Big(\psi_{12} + \psi_{21}\Big) \land \widetilde \phi_{12}\\
    \Big(\psi_{21} + \beta_{12}\Big) \land \phi_{12} & \asymp \Big(\psi_{21} + \beta_{12}\Big) \land \widetilde \phi_{12} 
\end{align*}
and the equations~\eqref{eq_simplification_dense_case} and~\eqref{eq_simplification_dense_case_upper_bound} become equivalent to proving
\begin{align}
\Big(\psi_{12} + \psi_{21}\Big) \land \widetilde \phi_{12}
        \asymp \psi_{21} \land \widetilde \phi_{12}.\label{eq_simplification_dense_case_tilde}\\
\psi_{21} \land \widetilde \phi_{12}
        \asymp \Big(\psi_{21} + \beta_{12}\Big) \land \widetilde \phi_{12}\label{eq_simplification_dense_case_upper_bounn_tilde}
\end{align}
In the rest of the proof, we focus on proving~\eqref{eq_simplification_dense_case_tilde} and~\eqref{eq_simplification_dense_case_upper_bounn_tilde}.
    
    Assume first that $\frac{n_1k_2}{k_1^2} \log\left(\frac{n_2}{k_2}\right) \leq 1$. 
    Using the inequality $\log(1+x) \geq x\log(2) $ for any $x \leq 1$, we have 
    \begin{align*}
        \psi_{21} = \frac{1}{k_2} \log\left(1+\frac{n_1k_2}{k_1^2} \log\left(\frac{n_2}{k_2}\right)\right) &\geq \log(2) \frac{n_1}{k_1^2} \log\left(\frac{n_2}{k_2}\right)\\
        & \geq \frac{\log(2)}{2} \frac{n_1}{k_1^2} \log\left(1+\frac{n_2}{k_2^2}\right) \quad \text{ if $c'$ is small enough}\\
        & = \frac{\log(2)}{2} \widetilde \phi_{12},
    \end{align*}
    which guarantees that both~\eqref{eq_simplification_dense_case_tilde} and~\eqref{eq_simplification_dense_case_upper_bounn_tilde} hold. 
    From now on, we will therefore assume that $\frac{n_1k_2}{k_1^2} \log\left(\frac{n_2}{k_2}\right) > 1$. 
    Now, assume that $\frac{n_2}{k_2^2} \leq 1$. 
    This implies the following inequalities   
    \begin{align*}
        \psi_{21} = \frac{1}{k_2} \log\left(1+\frac{n_1k_2}{k_1^2} \log\left(\frac{n_2}{k_2}\right)\right) &>\frac{\log(2)}{k_2} \\
        & \geq \frac{\log(2)}{k_2} \, \frac{\bar c n_1k_2}{k_1^2} \, \frac{n_2}{k_2^2} \\
        & = \bar c \log(2) \frac{n_1n_2}{k_1^2 k_2^2} \\
        & \geq \bar c \log(2) \frac{n_1}{k_1^2} \log\left(1+\frac{n_2}{k_2^2}\right)\\
        & = \bar c \log(2) \widetilde \phi_{12},
    \end{align*}
    which ensures that~\eqref{eq_simplification_dense_case_tilde} and~\eqref{eq_simplification_dense_case_upper_bounn_tilde} hold as well and yields the result. 
    Thus, we will now assume that $\frac{n_2}{k_2^2}>1$, which also implies $\frac{n_2k_1}{k_2^2} \log\left(\frac{n_1}{k_1}\right)>1$.

    Now, suppose for the sake of contradiction that $k_1 > \frac{n_1}{\log(n_1)}$. 
    Note that the assumption $\frac{n_1}{k_1} \geq e \log\left(\frac{n_2}{k_2}\right)$ implies that
    \begin{align*}
        \log(n_1) > e\log\left(\frac{n_2}{k_2}\right) > e\log(\sqrt{n_2}) = \frac{e}{2} \log(n_2),
    \end{align*}
    that is $n_1 > n_2^{e/2}.$ Therefore, 
    \begin{align*}
        \frac{k_2 n_1}{k_1^2} \log\left(\frac{n_2}{k_2}\right) &< \frac{k_2n_1}{n_1^2/\log^2(n_1)}\log\left(n_2\right) < \frac{\sqrt{n_2}}{n_1 }\cdot \log(n_2)\log^2(n_1)\\
        & < \frac{2}{e}\frac{\log^3(n_1)}{n_1^{1-1/e}} <1,
    \end{align*}
    which contradicts the assumed inequality $ \frac{k_2 n_1}{k_1^2} \log\left(\frac{n_2}{k_2}\right)>1$. 

    Therefore, it holds that $ k_1 \leq \frac{n_1}{\log(n_1)}$. 
    Since $\frac{n_1 k_2}{k_1^2} \log\left(\frac{n_2}{k_2}\right) > 1$, 
    we have
    \begin{align*}
        \psi_{12} = \frac{1}{k_1} \log\left(1+ \frac{n_2k_1}{k_2^2} \log\left(\frac{n_1}{k_1}\right)\right) &\leq \frac{1}{k_1} \log\left(2 \frac{n_2k_1}{k_2^2} \log\left(\frac{n_1}{k_1}\right)\right)\\
        & \leq \frac{\log(n_2/k_2) + 2\log(n_1)}{k_1}  \quad \text{ provided } k_1 \leq c' n_1\\
        & \leq \frac{1}{\bar ck_2} \frac{k_1}{ n_1}\left[\log\left(\frac{n_2}{k_2}\right) + 2\log(n_1)\right]\\
        & \leq \frac{1}{\bar c k_2} \left(1/e + 2\right)\\
        & \leq \frac{e^{-1} + 2}{\bar c \log(2)} \cdot \frac{1}{k_2} \log\left(1+ \frac{n_1 k_2}{k_1^2} \log\left(\frac{n_2}{k_2}\right)\right),\\
        & = \frac{e^{-1} + 2}{\bar c \log(2)} \psi_{21},
    \end{align*}
    which ensures that~\eqref{eq_simplification_dense_case_tilde} holds. 
    Moreover, we verify that equation~\eqref{eq_simplification_dense_case_upper_bounn_tilde} holds as well. This is clear if $\frac{n_1 k_2}{k_1^2} \log\left(\frac{n_2}{k_2}\right) \leq 1$ since we have $\beta_{12} = 0$. Otherwise, we have
    \begin{align*}
        \psi_{21} = \frac{1}{k_2} \log\left(1+ \frac{n_1 k_2}{k_1^2} \log\left(\frac{n_2}{k_2}\right)\right) > \log(2) \frac{1}{k_2} \geq \frac{\log(2)\bar c}{k_1} \frac{n_1}{k_1} \geq \frac{e\log(2) \bar c}{k_1} \log\left(\frac{n_2}{k_2}\right) \geq e \bar c \log(2) \beta_{12},
    \end{align*}
    which ensures~\eqref{eq_simplification_dense_case_upper_bounn_tilde} holds and concludes the proof of the first claim. 

    \item Now, let $\alpha>0$ and let $c_\mu, c>0$ be the constants defined in Lemma~\ref{lem_lb_dense_allrates}. 

    Assume first that $c_\mu^{-1}k_2\log\big(c\frac{n_2}{k_2}\big) \leq \Ct \frac{k_1^2}{n_1}$. 
    Recalling that $\frac{n_2}{k_2}\geq \frac{1}{c'}$ where $c'$ can be taken arbitrarily small, we have
    \begin{align*}
        R &\leq \phi_{12} = \frac{n_1}{k_1^2} \log\left(1+\frac{n_2}{k_2^2}\right) &\\
        & \leq \frac{n_1}{k_1^2} \log\left(2\frac{n_2}{k_2}\right) & \text{ provided $c'$ is small enough}\\
        & \leq 2\frac{n_1}{k_1^2} \log\left(c\frac{n_2}{k_2}\right) &\text{ provided $c'$ is small enough}\\
        &\leq \Ct^{-1}\frac{n_1}{k_1^2}\log\Big(c \frac{n_2}{k_2}\Big) + \frac{c_\mu}{k_2}\log\Big(\frac{\Ct}{2e}\Big),&
    \end{align*}
    which implies that if $\mu^2 \leq R$, then we have by Lemma~\ref{lem_lb_dense_allrates}.1. that 
    \[\EE\Big[\exp(\mu^2XY)\one\big(X \geq  \Ct \frac{k_1^2}{n_1} \big)\Big] < \alpha.\]
    Now, assume $\Ct \frac{k_1^2}{n_1} < c_\mu^{-1}k_2\log\big(c\frac{n_2}{k_2}\big) \leq k_1$. 
    Then we have 
    \begin{align*}
        R &\leq \psi_{21} = \frac{1}{k_2} \log\left( 1+ \frac{n_1k_2}{k_1^2} \log\left(\frac{n_2}{k_2}\right)\right)&\\
        & \leq \frac{1}{k_2} \log\left( 2\frac{n_1k_2}{k_1^2} \log\left(\frac{n_2}{k_2}\right)\right) & \text{ since } \frac{n_1k_2}{k_1^2} \geq 1\\
        &\leq \frac{1}{k_2}\log\Big(\frac{k_2n_1\log(c \frac{n_2}{k_2})}{c_\mu 2e k_1^2}\Big) &\text{ if $c'$ is small enough.}
    \end{align*}
    Therefore, whenever $\mu \leq c_\mu R$, we have by Lemma~\ref{lem_lb_dense_allrates}.2. that
    \[\EE\Big[\exp(\mu^2XY)\one\big(X \geq \lceil  \Ct \frac{k_1^2}{n_1}\rceil\big)\Big] < \alpha.\]

    Finally, assume that $k_1 < c_\mu^{-1}k_2\log\big(c\frac{n_2}{k_2}\big)$. 
    Then we have
    \begin{align*}
        \frac{1}{c_\mu}\log\left(c\frac{n_2}{k_2}\right) > \frac{k_1}{k_2} \geq \frac{1}{16e^4} \frac{n_1}{k_1} \geq \frac{1}{16e^3} \log\left(\frac{n_2}{k_2}\right),
    \end{align*}
    so that $\frac{k_2}{k_1} \log\left(\frac{n_2}{k_2}\right) \leq 16e^3$. 
    Therefore, we have
    \begin{align*}
        R &\leq \psi_{21} = \frac{1}{k_2} \log\left(1+\frac{n_1}{k_1^2}k_2\log\left(\frac{n_2}{k_2}\right)\right)\\
        & \leq \frac{1}{k_2} \log\left(1+16e^3\frac{n_1}{k_1}\right)\\
        & \leq \frac{1}{2k_2} \log\left(\frac{n_1}{2ek_1}\right) \qquad \text{ if $c'$ is small enough} \\
        & \leq \frac{1}{c_\mu}\left[\frac{1}{k_1}\log\big(c \frac{n_2}{k_2}\big) + \frac{c_\mu}{k_2}\log\big(\frac{n_1}{2ek_1}\big)\right]
    \end{align*}
    By Lemma~\ref{lem_lb_dense_allrates}.3, if $\mu^2 \leq c_\mu R \leq \frac{1}{k_1}\log\big(c \frac{n_2}{k_2}\big) + \frac{c_\mu}{k_2}\log\big(\frac{n_1}{2ek_1}\big)$, we obtain
    \[\EE\Big[\exp(\mu^2XY)\one\big(X \geq \lceil  \Ct \frac{k_1^2}{n_1}\rceil\big)\Big] < \alpha,\]
    and the proof is complete.
    \item We immediately deduce from~\eqref{eq_simplification_dense_case} that  $R \gtrsim \big(\psi_{21} + \beta_{12}\big) \land \big(\psi_{12} + \beta_{21}\big) \land \phi_{12} \land \phi_{21}$. 
    We now prove the converse bound. 
    Note that if $\frac{n_2k_1}{k_2^2} \log\left(\frac{n_1}{k_1}\right) \leq 1$, then we have $\beta_{21} = 0$ and $ \psi_{12} + \beta_{21} = \psi_{12} \lesssim \phi_{21}$ by Lemma~\ref{lem_psi_disappears} and the result follows. 
    We assume that $\frac{n_2k_1}{k_2^2} \log\left(\frac{n_1}{k_1}\right)>1$ from now on.
    
    To prove that $R \asymp \big(\psi_{21} + \beta_{12}\big) \land \big(\psi_{12} + \beta_{21}\big) \land \phi_{12} \land \phi_{21}$, it suffices to show that $\beta_{21} \gtrsim \psi_{21} + \beta_{12}$. 
    We have 
    \begin{align*}
        \psi_{21} &= \frac{1}{k_2} \log\left(1+ \frac{n_1}{k_1} \frac{k_2}{k_1} \log\left(\frac{n_2}{k_2}\right)\right)\\
        & \leq \frac{1}{k_2} \log\left(1+ \frac{n_1}{k_1} \frac{k_1}{\bar cn_1} \frac{n_1}{ek_1}\right)\\
        & \lesssim \frac{1}{k_2} \log\left(\frac{n_1}{k_1}\right)\\
        & = \beta_{21}
    \end{align*}
    which concludes the proof.
    \end{enumerate}
\end{proof}
The Lemma below shows that $R$ is a lower bound on the rate emerging from Lemma~\ref{lem_lb_s1d1big_truncchi2}.
\begin{lemma}\label{lem_simplify_s1d1big_truncchi2}
    Suppose that $k_1^2 \leq \bar{c}n_1k_2$ for some $\bar{c} \in (0,(2e)^{-4})$ and $k_1 \leq c'n_1$, $k_2 \leq c' n_2$ for some sufficiently small $c'>0$. Furthermore, suppose that $\frac{k^2_2}{n_2} \leq 2e\frac{k_1^2}{n_1}$. 
    For some small enough constant $c_\mu$, we define the quantity 
    \begin{align*}
        M = \begin{cases}
            \frac{1}{k_1} \log\Big(1 + c_\mu \frac{n_2k_1}{k_2^2}\log(\frac{n_1}{2ek_1})\Big) & \text{ if } k_1 \leq \frac{k_2}{\log\big(\frac{n_1k_2}{k_1^2}\big)}\\
            \frac{1}{k_2}\log\big(\frac{n_1k_2}{k_1^2}\big)\log\Big(1 + c_\mu \frac{n_2}{k_2\log\big(\frac{n_1k_2}{k_1^2}\big)}\log\big(\frac{n_1k_2}{2ek_1^2\log(\frac{n_1k_2}{k_1^2})}\big)\Big) & \text{ otherwise.}
        \end{cases}
    \end{align*}
    \begin{enumerate}
        \item 
    Then for some sufficiently small constant $c>0$ depending only on $\mu$, it holds that $M \geq cR$. 
    Hence, for any $\alpha>0$, there exists a constant $c'_\mu>0$ such that, whenever $\mu^2 \leq c'_\mu R$, we have
    \[\EE\Big[\exp(\mu^2XY)\one\Big( 1 \vee \Ct \frac{k_1^2}{n_1} \leq X \leq k_1\wedge\frac{k_2}{\log\big(\frac{n_1k_2}{k_1^2}\big)}\Big)\Big] < \alpha.\]
    \item Moreover, we have
    \begin{align}
        R \geq \begin{cases}
            \frac{1}{2}(\psi_{12} + \beta_{21}) \land \phi_{12} \land \phi_{21} & \text{ if } k_1 \leq \frac{k_2}{\log\big(\frac{n_1k_2}{k_1^2}\big)}\\
            \frac{1}{2}(\psi_{21} + \beta_{12}) \land \phi_{12} \land \phi_{21} & \text{ otherwise. }
        \end{cases} \label{eq_lower_bounn_M}
    \end{align}
    \item Moreover, it holds that $R \asymp \big(\psi_{21} + \beta_{12}\big) \land \big(\psi_{12} + \beta_{21}\big) \land \phi_{12} \land \phi_{21}$.
    \end{enumerate}
\end{lemma}

\begin{proof}[Proof of Lemma ~\ref{lem_simplify_s1d1big_truncchi2}]
\phantom{ }

\begin{enumerate}
\item We start by showing that $M \geq c(\psi_{21} + \psi_{12})$. 
    Using  Lemma~\ref{lem_lb_s1d1big_truncchi2} and Lemma~\ref{lem_logk_constants}.(i), we obtain
\begin{align*}
    2M &\geq\frac{1}{k_1}\log\Big(1 + c_\mu \frac{n_2k_1}{k_2^2}\log(\frac{n_1}{2ek_1})\Big) + \frac{\log\big(\frac{n_1k_2}{k_1^2}\big)}{k_2}\log\Big(1 + c_\mu \frac{n_2}{k_2\log\big(\frac{n_1k_2}{k_1^2}\big)}\log \big(\frac{n_1k_2}{2ek_1^2\log(\frac{n_1k_2}{k_1^2})}\big)\Big) \\
    &\geq\frac{c_\mu}{k_1}\log\Big(1 + \frac{n_2k_1}{k_2^2}\log(\frac{n_1}{2ek_1})\Big) + c_\mu\frac{\log\big(\frac{n_1k_2}{k_1^2}\big)}{k_2}\log\Big(1 +  \frac{n_2}{k_2\log\big(\frac{n_1k_2}{k_1^2}\big)}\log \big(\frac{n_1k_2}{2ek_1^2\log(\frac{n_1k_2}{k_1^2})}\big)\Big) \\
    &\geq c_\mu \left[\frac{1}{2}\psi_{12} + \frac{\log\big(\frac{n_1k_2}{k_1^2}\big)}{k_2}\log\Big(1 + \frac{n_2}{k_2\log\big(\frac{n_1k_2}{k_1^2}\big)}\log \big(\frac{n_1k_2}{2ek_1^2\log(\frac{n_1k_2}{k_1^2})}\big)\Big)\right].
\end{align*}
To obtain $M \geq c(\psi_{12} + \psi_{21})$, it now remains to show
\[c\psi_{21} \leq \frac{\log\big(\frac{n_1k_2}{k_1^2}\big)}{k_2}\log\Big(1 + \frac{n_2}{k_2\log\big(\frac{n_1k_2}{k_1^2}\big)}\log \big(\frac{n_1k_2}{2ek_1^2\log(\frac{n_1k_2}{k_1^2})}\big)\Big).\]

 Since $\frac{n_1k_2}{k_1^2} \geq \bar{c}^{-1} \geq 2$, it follows that $\frac{1}{2}\log(1 + \frac{n_1k_2}{k_1^2}) \leq \log(\frac{n_1k_2}{k_1^2})$ by the inequality $\sqrt{1+x} \leq x$ that holds true for any $x \geq 2$. Furthermore, the inequality $\frac{n_1k_2}{k_1^2} \geq \bar{c}^{-1} \geq 16e^4$ implies that 
 \begin{align*}
 \frac{\log\Big(\frac{n_1k_2}{2ek_1^2\log\big(\frac{n_1k_2}{k_1^2}\big)}\Big)}{\log \big(\frac{n_1k_2}{k_1^2}\big)} &= 1 - \frac{\log\Big(2e\log\big(\frac{n_1k_2}{k_1^2}\big)\Big)}{\log\big(\frac{n_1k_2}{k_1^2}\big)} \\
 &\geq \frac12.
 \end{align*}
Moreover, since $n_2/k_2 \geq 1/c' \geq 4$ provided $c' \leq 1/4$, we also have that $\log\left(\frac{n_2}{2k_2}\right) \geq \frac{1}{2} \log\left(\frac{n_2}{k_2}\right)$. Combining these observations yields 
\begin{align}
    \frac{\log\big(\frac{n_1k_2}{k_1^2}\big)}{k_2}\log\Big(1 +  \frac{n_2}{k_2\log\big(\frac{n_1k_2}{k_1^2}\big)}\log \big(\frac{n_1k_2}{2ek_1^2\log(\frac{n_1k_2}{k_1^2})}\big)\Big) &\geq \frac{\frac{1}{2}\log\big(1 + \frac{n_1k_2}{k_1^2}\big)}{k_2}\log\Big(
\frac12\frac{n_2}{k_2}\Big)\nonumber\\
& \geq \frac14\frac{\log\big(1 + \frac{n_1k_2}{k_1^2}\big)}{k_2} \log\Big(
\frac{n_2}{k_2}\Big)\nonumber\\
& \geq \frac{1}{4} \frac{1}{k_2} \log\left(1 + \frac{n_1k_2}{k_1^2}\log\Big(
\frac{n_2}{k_2}\Big)\right) \text{ by Lemma~\ref{lem_logk_constants}.(ii)}\nonumber\\
& = \frac{1}{4} \psi_{21}.\label{eq_psi12/4}
\end{align}
It remains to invoke Lemma~\ref{lem_lb_s1d1big_truncchi2} to conclude the proof of the first claim. 

    \item 
We now prove equation~\eqref{eq_lower_bounn_M}. 
Assume first that $k_1 \leq \frac{k_2}{\log\big(\frac{n_1k_2}{k_1^2}\big)}$. 
Then we obtain
\begin{align}
    \frac{n_1}{k_1} \leq \frac{\frac{n_1k_2}{k_1^2}}{\log\big(\frac{n_1k_2}{k_1^2}\big)},\quad  \text{ hence } \quad \frac{n_1k_2}{k_1^2} \geq \frac{n_1}{k_1} \log\left(\frac{n_1}{k_1}\right), \quad \text{ i.e. } \quad \frac{1}{k_1} \geq \frac{1}{k_2} \log\left(\frac{n_1}{k_1}\right)\label{eq_k2_geq_k1_log}
\end{align}
by Lemma~\ref{lem_logs}.(ii). 
Note that, since $\frac{n_2k_1}{k_2^2} \geq \frac{n_1}{k_1} \geq \frac{1}{c'} $ can be made arbitrarily large by taking $c'$ small enough, we have
\begin{align*}
    \beta_{21} &\leq \frac{1}{k_2} \log\left(\frac{n_1}{k_1}\right) \leq \frac{1}{k_1} \leq \frac{1}{k_1}\log\left(1 + \frac{n_2k_1}{k_2^2}\log\left(\frac{n_1}{k_1}\right)\right)  = \psi_{12}.
\end{align*}
Hence,
\begin{align*}
    R &= (\psi_{12} + \psi_{21}) \land \phi_{12} \land \phi_{21} \\
    &\geq \psi_{12}  \land \phi_{12} \land \phi_{21} \\
    &\geq \frac{1}{2}(\psi_{12} + \beta_{21}) \land \phi_{12} \land \phi_{21}
\end{align*}
as claimed.
Assume now that $k_1 > \frac{k_2}{\log\big(\frac{n_1k_2}{k_1^2}\big)}$. 
By assumption, we have that $\frac{n_1k_2}{k_1^2} \geq \bar{c}^{-1} \geq 2e$.
Therefore, 
\begin{align*}
    \frac{\frac{n_1 k_2}{k_1^2}}{\log\left(\frac{n_1 k_2}{k_1^2}\right)}\leq \frac{n_1}{k_1}, \quad  \text{ hence } \quad \frac{n_1 k_2}{k_1^2} \leq 2 \frac{n_1}{k_1} \log\left(\frac{2n_1}{k_1}\right), \quad \text{ i.e. } \quad k_2 \leq 3 k_1 \log\left(\frac{n_1}{k_1}\right)
\end{align*}
by Lemma~\ref{lem_logs}.(i) and for $c'$ sufficiently small.
This yields
\begin{align*}
    \psi_{12} = \frac{1}{k_1} \log\Big(1 +  \frac{n_2k_1}{k_2^2}\log\left(\frac{n_1}{k_1}\right)\Big)
    & \geq \frac{1}{k_1} \log\left(1+\frac{n_2}{k_2}\right)\geq \frac{1}{k_1} \log\left(\frac{n_2}{k_2}\right) \geq \beta_{12}.
\end{align*}
Therefore, we obtain
\begin{align*}
    R &= (\psi_{12} + \psi_{21}) \land \phi_{12} \land \phi_{21}\\
    &\geq \frac{1}{2}(\psi_{21} + \beta_{12}) \land \phi_{12} \land \phi_{21}.
\end{align*}
This concludes the proof of equation~\eqref{eq_lower_bounn_M}.
\item Assume that $k_1 \leq \frac{k_2}{\log\big(\frac{n_1k_2}{k_1^2}\big)}$, which implies $k_2 \geq k_1 \log\left(\frac{n_1}{k_1}\right)$ by~\eqref{eq_k2_geq_k1_log}. We aim to prove that $\psi_{21} + \beta_{12} \gtrsim \psi_{12} + \beta_{21}$. We have
\begin{align*}
    \psi_{12} &= \frac{1}{k_1} \log\left(1+ \frac{n_2k_1}{k_2^2} \log\left(\frac{n_1}{k_1}\right)\right)\\
    & \leq \frac{1}{k_1} \log\left(1+\frac{n_2}{k_2}\right)\\
    & \lesssim \beta_{12}.
\end{align*}
Moreover, 
\begin{align*}
    \psi_{21} = \frac{1}{k_2} \log\left(1+ \frac{n_1k_2}{k_1^2} \log\left(\frac{n_2}{k_2}\right)\right)
\end{align*}
    \end{enumerate}
\end{proof}

The lemma below shows that $R$ is a lower bound on the rate emerging from Lemma~\ref{lem_lb_s2d2big_maxlin}.

\begin{lemma}\label{lem_simplify_s2d2big_maxlin}
    Suppose that $k_1^2 \leq \bar{c}n_1k_2$ for some $\bar{c} \in (0, (2e)^{-4})$ and $k_1 < \frac12 n_1$. 
    Furthermore, suppose that $\frac{k^2_2}{n_2} \geq 2e\frac{k_1^2}{n_1}$ and that $\lceil \Ct \frac{k_1^2}{n_1}\rceil \leq \Big\lfloor \frac{\Ct k_2^2/n_2}{\log\big(\frac{n_1k_2^2}{k_1^2n_2}\big)}\Big\rfloor\wedge k_1$. Then for any $\alpha > 0$, there exists a constant $c > 0$ such that if $\mu^2 \leq cR$, 
    \[\EE\left[\exp\big(\mu^2XY\big)\one\bigg(\Ct \frac{k_1^2}{n_1}  \leq X \leq  \frac{\Ct k_2^2/n_2}{\log\big(\frac{n_1k_2^2}{k_1^2n_2}\big)}\wedge k_1\bigg)\right] < \alpha.\]
\end{lemma}

\begin{proof}[Proof of Lemma ~\ref{lem_simplify_s2d2big_maxlin}]
    First, observe that the assumption $\lceil \Ct \frac{k_1^2}{n_1}\rceil \leq \big\lfloor \frac{\Ct k_2^2/n_2}{\log\big(\frac{n_1k_2^2}{k_1^2n_2}\big)}\big\rfloor\wedge k_1$ implies 
    $1 \leq \frac{\Ct k_2^2/n_2}{\log\big(\frac{n_1k_2^2}{k_1^2n_2}\big)}$. Rearranging terms, this implies 
    $$\frac{n_2}{k_2^2} \leq \Ct\log^{-1} \big(\frac{n_1k_2^2}{n_2k_1^2}\big) \leq \frac{C_*}{\log(2e)},$$
     and therefore $\phi_{21} = \frac{n_2}{k_2^2}\log\big(1 + \frac{n_1}{k_1^2}\big)$ as long as we take the $C$ in the definition of $R$ to be at least $\frac{C_*}{\log(2e)}$. Now, assume that $\mu^2 \leq \widetilde c_\mu R$ for some $\widetilde c_\mu>0$. Then, we have
        \begin{align*}
            \frac{1}{\widetilde c_\mu}\mu^2 &\leq R \\
            &\leq \phi_{21} \\
            &= \frac{n_2}{k_2^2}\log\big(1 + \frac{n_1}{k_1^2}\big) \\
            &\leq \frac{n_2}{k_2^2}\log\big(4e\big) \quad \text{(this follows from $\frac{n_1}{k_1^2}\leq 2e \frac{n_2}{k_2^2}$)} \\
            &\leq \frac{n_2}{k_2^2}\log\big(4e\big)\log\big(\frac{\Ct}{2e}\vee \frac{n_1}{2ek^2_1}\big).
        \end{align*}
        Letting $c'_\mu$ denote the constant from Lemma~\ref{lem_lb_s2d2big_maxlin}, we can now choose $\widetilde c_\mu = \frac{c_\mu'}{ \log(4e)} $ sufficiently small so that, whenever $\mu^2 \leq \widetilde c_\mu R$, we have $\mu^2 \leq c_\mu' \frac{n_2}{k_2^2}\log\big(\frac{\Ct}{2e}\vee \frac{n_1}{2ek^2_1}\big)$.         
        The conclusion follows by Lemma ~\ref{lem_lb_s2d2big_maxlin}, and the proof is complete. 
\end{proof}

The lemma below ensure that $R$ is a lower bound on the rate obtained in Lemma~\ref{lem_lb_s2d2big_truncchi2}.
\begin{lemma}\label{lem_simplify_s2d2big_truncchi2}
Suppose that $k_1^2 \leq \bar{c}n_1k_2$ for some $\bar{c} \in (0, (2e)^{-4})$ and $k_1 < c' n_1$, $k_2 \leq c' n_2$ for some sufficiently small constant $c'$. 
    Furthermore, suppose that $\frac{k^2_2}{n_2} \geq 2e\frac{k_1^2}{n_1}$. 
    \begin{enumerate}
        \item Then for any $\alpha > 0$, there exist constants $c_\mu > 0$ and $C_* \geq 1$ such that if $1 \vee \big\lceil \Ct \frac{k_1^2}{n_1}\big\rceil \vee \Big\lceil \frac{\Ct k_2^2/n_2}{\log\big(\frac{n_1k_2^2}{k_1^2n_2}\big)}\Big\rceil \leq k_1\wedge\big\lfloor\frac{k_2}{\log\big(\frac{n_1k_2}{k_1^2}\big)}\big\rfloor$and $\mu^2 \leq c_\mu R$, 
    \[\EE\Big[\exp(\mu^2XY)\one\Big( \Ct \frac{k_1^2}{n_1} \vee \frac{\Ct k_2^2/n_2}{\log\big(\frac{n_1k_2^2}{k_1^2n_2}\big)} \leq X \leq k_1\wedge\frac{k_2}{\log\big(\frac{n_1k_2}{k_1^2}\big)}\Big)\Big] < \alpha.\]
    \item Moreover, we have
    \begin{align*}
        R \geq \begin{cases}
            \big(\psi_{21} + \beta_{12}\big)\land \phi_{12} \land \phi_{21} & \text{ if } k_2 \leq k_1 \log\left(\frac{n_1}{k_1}\right)\\
            \big(\psi_{12} + \beta_{21}\big)\land \phi_{12} \land \phi_{21} & \text{ otherwise. } 
        \end{cases}
    \end{align*}
    Consequently, we have $R \geq \big(\psi_{21} + \beta_{12}\big) \land \big(\psi_{12} + \beta_{21}\big) \land \phi_{12} \land \phi_{21}$.
    \end{enumerate}
    
\end{lemma}
\begin{proof}[Proof of Lemma ~\ref{lem_simplify_s2d2big_truncchi2}]
\phantom{}

\begin{enumerate}
    \item 
    By the assumption $\frac{n_1k_2}{k_1^2} \geq \bar{c}^{-1} \geq 16e^4$, 
    we can repeat the steps leading to equation~\eqref{eq_psi12/4} to obtain 
    \begin{align*}
         \frac{\log\big(\frac{n_1k_2}{k_1^2}\big)}{k_2}\log\Big(1 + c_\mu \frac{n_2}{k_2\log\big(\frac{n_1k_2}{k_1^2}\big)}\log \big(\frac{n_1k_2}{2ek_1^2\log(\frac{n_1k_2}{k_1^2})}\big)\Big)&\geq \frac{c_\mu}{4}\frac{\log\Big(1 + \frac{n_1k_2}{k_1^2}\log\big(\frac{n_2}{k_2}\big)\Big)}{k_2} \\
    &= \frac{c_\mu}{4} \psi_{21}.
    \end{align*}
    It immediately follows that, for some small enough constant $c$, we have 
    \begin{align*}
        \mu^2 &\leq cR \\
        &\leq c(\psi_{12} + \psi_{21}) \\
        &\leq \frac{1}{2}\frac{1}{k_1}\log\Big(1 + c_\mu \frac{n_2k_1}{k_2^2}\log(\frac{n_1}{2ek_1})\Big) + \frac12 \frac{\log\big(\frac{n_1k_2}{k_1^2}\big)}{k_2}\log\Big(1 + c_\mu \frac{n_2}{k_2\log\big(\frac{n_1k_2}{k_1^2}\big)}\log \big(\frac{n_1k_2}{2ek_1^2\log(\frac{n_1k_2}{k_1^2})}\big)\Big).
    \end{align*}
    The conclusion follows from Lemma ~\ref{lem_lb_s2d2big_truncchi2}.
    \item Assume now that $k_2 \leq k_1 \log\left(\frac{n_1}{k_1}\right)$. 
    Then we have
    \begin{align*}
        \psi_{12} &= \frac{1}{k_1} \log\left(1+\frac{n_2k_1}{k_2^2} \log\left(\frac{n_1}{k_1}\right)\right)\\
        & \geq \frac{1}{k_1} \log\left(1+\frac{n_2}{k_2} \right)\\
        & \geq \beta_{12}, 
    \end{align*}
    which yields the desired result. 
    Assume now that $k_2 > k_1 \log\left(\frac{n_1}{k_1}\right)$. 
    Then
    \begin{align*}
        \psi_{21} &= \frac{1}{k_2} \log\left(1+\frac{n_1k_2}{k_1^2} \log\left(\frac{n_2}{k_2}\right)\right)\\
        & > \frac{1}{k_2} \log\left(\frac{n_1}{k_1}\log\left(\frac{n_1}{k_1}\right)\log\left(\frac{n_2}{k_2}\right)\right)\\
        & \geq \frac{1}{k_2}\log\left(\frac{n_1}{k_1}\right)\\
        & \geq \beta_{21}.
    \end{align*}
    This concludes the proof.
    \end{enumerate}
\end{proof}

The lemma below ensure that $R$ is a lower bound on the rate obtained in Lemma~\ref{lem_max_test_lower_bound}.
\begin{lemma}\label{lem_simplification_rate_max_test}
    Suppose that $k_1^2 \leq \bar{c}n_1k_2$ for some $\bar{c} \in (0,(2e)^{-4})$ and $k_1 \leq c'n_1$ and $k_2 \leq c' n_2$ for some small enough constant $c'>0$. Furthermore, suppose that $k_2 < k_1\log\big((n_1k_2)/k_1^2\big)$ and $\frac{n_1}{k_1} \geq e \log\left(\frac{n_2}{k_2}\right)$. 
    \begin{enumerate}
        \item Then for any $\alpha_4 > 0$, there exist constants $c_\mu > 0$ and $C_* \geq 1$ such that if $\mu^2 \leq c_\mu R,$ then 
    \[\EE\Big[\exp(\mu^2XY)\one\Big(X \geq \frac{k_2}{\log\big(\frac{n_1k_2}{k_1^2}\big)} \lor  \Ct \frac{k_1^2}{n_1}\Big\rceil\Big)\Big] < \alpha_4.\]
    \item Moreover, it holds that 
    \begin{align*}
        R \geq \begin{cases}
            \frac{1}{4}\big((\psi_{21}+\beta_{12}) + (\psi_{12} + \beta_{21})\big) \land \phi_{12} \land \phi_{21}& \text{ if } k_1 \leq k_2\log\left(\frac{n_2}{k_2}\right)\\
            \frac{1}{2}\big(\psi_{21} + \beta_{12}\big) \land \phi_{12} \land \phi_{21} & \text{ otherwise.}
        \end{cases}
    \end{align*}
    Consequently, we have $R \geq \big(\psi_{21} + \beta_{12}\big) \land \big(\psi_{12} + \beta_{21}\big) \land \phi_{12} \land \phi_{21}$.
    \end{enumerate}
\end{lemma}
\begin{proof}[Proof of Lemma~\ref{lem_simplification_rate_max_test}]
\phantom{ }

\begin{enumerate}
    \item 

Note that the relation $k_2 < k_1\log\big((n_1k_2)/k_1^2\big)$ implies that 
\begin{align*}
    \frac{n_1k_2}{k_1^2} &< \frac{n_1}{k_1} \log\left(\frac{n_1k_2}{k_1^2}\right) \leq \frac{2n_1}{k_1} \log\left(\frac{2n_1}{k_1}\right) \quad \text{ by Lemma~\ref{lem_logs}.(i)}
\end{align*}
which yields 
\begin{align}
    k_2 \leq 4 k_1 \log(n_1/k_1). \label{eq_s2<s1_log(d1/s1)}
\end{align}
Now, we obtain 
\begin{align*}
    \frac{n_2k_1}{k_2^2} \log\left(\frac{n_1}{k_1}\right) = \frac{n_2}{k_2} \cdot \frac{k_1}{k_2} \log\left(\frac{n_1}{k_1}\right) \geq \frac{1}{4c'} \geq 1 
\end{align*}
provided $c' \leq 1/4$, and 
\begin{align*}
    \frac{n_1k_2}{k_1^2} \log\left(\frac{n_2}{k_2}\right) \geq \frac{1}{\bar c} \log\left(\frac{1}{c'}\right) \geq 1 
\end{align*}
provided $c'$ and $\bar c$ are small enough, which ensures that $\nu_{12 } = \frac{1}{k_1} \log\left(\frac{n_2}{k_2}\right)$ and $\beta_{21} = \frac{1}{k_2} \log\left(\frac{n_1}{k_1}\right)$. 

Assume first that  $k_1 < k_2\log(n_2/k_2)$. 
Then, we have 
\begin{align*}
    R &\leq \psi_{12} + \psi_{21}\\
    & = \frac{1}{k_1} \log\left(1+ \frac{n_2k_1}{k_2^2} \log\left(\frac{n_1}{k_1}\right)\right) + \frac{1}{k_2} \log\left(1+ \frac{n_1k_2}{k_1^2} \log\left(\frac{n_2}{k_2}\right)\right)\\
    & \leq \frac{1}{k_1} \log\left(2\frac{n_2}{k_2} \frac{k_1}{k_2} \log\left(\frac{n_1}{k_1}\right)\right) + \frac{1}{k_2} \log\left(1+ \frac{n_1}{k_1} 4\log\left(\frac{n_1}{k_1}\right)\log\left(\frac{n_2}{k_2}\right)\right)\\
    & \leq \frac{1}{k_1} \log\left(  \frac{n_2}{k_2} \right) + \frac{1}{k_1} \log\left(\frac{2k_1}{k_2} \log\left(\frac{n_1}{k_1}\right)\right) + \frac{1}{k_2} \log\left(\frac{8}{e} \left(\frac{n_1}{k_1} \right)^2\log\left(\frac{n_1}{k_1}\right)\right)\\
    & \leq \frac{1}{k_1} \log\left(  \frac{n_2}{k_2} \right) + \frac{1}{k_1} \log\left(\frac{2k_1}{k_2} \log\left(\frac{n_1}{k_1}\right)\right) + \frac{4}{k_2} \log\left(\frac{n_1}{k_1} \right)\\
    & \leq \frac{1}{k_1} \log\left(  \frac{n_2}{k_2} \right) + \frac{6}{k_2} \log\left(\frac{n_1}{k_1} \right)
    \\ &\leq 6 \beta_{12} + 6 \beta_{21}.
\end{align*}
In the second to last inequality, we used the fact that $\log\left(\frac{2k_1}{k_2} \log\left(\frac{n_1}{k_1}\right)\right) \leq \frac{2k_1}{k_2} \log\left(\frac{n_1}{k_1} \right)$ due to the inequality $\log(x) \leq x$ that holds true for any $x>0$. 
It follows from  Lemma~\ref{lem_max_test_lower_bound} that if $\mu^2 \leq c_\mu R$ for some sufficiently small constant $c_\mu>0$, then  
    $$\EE\Big[\exp(\mu^2XY)\one\Big(X \geq \frac{k_2}{\log\big(\frac{n_1k_2}{k_1^2}\big)}\lor  \Ct \frac{k_1^2}{n_1}\Big)\Big] < \alpha.$$

    Assume now that $k_1 > k_2\log\left(\frac{n_2}{k_2}\right)$. 
    Again, we have $R \leq \psi_{12} + \psi_{21}$. 
    We aim to show that
    \begin{align}
        \psi_{12} + \psi_{21} \leq \frac{4}{k_2}\log\big(\frac{k_2n_1\log(n_2/k_2)}{2ek_1^2}\big),\label{eq_target_simplification_max_rate}
    \end{align}
    which will conclude the proof since it ensures that, if $\mu^2 \leq c_\mu R$ for some small enough $c_\mu>0$, then 
    $$\EE\Big[\exp(\mu^2XY)\one\Big(X \geq \frac{k_2}{\log\big(\frac{n_1k_2}{k_1^2}\big)}\lor  \Ct \frac{k_1^2}{n_1}\Big)\Big] < \alpha.$$ 
    by Lemma~\ref{lem_max_test_lower_bound}.
    We observe that the assumption $k_1^2 \leq \bar c n_1 k_2$ ensures that $\frac{n_1k_2}{k_1^2}\log(n_2/k_2) > 2$, which yields
    \begin{align*}
        \psi_{21} &= \frac{1}{k_2} \log\left(1+ \frac{n_1k_2}{k_1^2} \log\left(\frac{n_2}{k_2}\right)\right) \\
        & \leq \frac{2}{k_2} \log\left(\frac{n_1k_2}{k_1^2} \log\left(\frac{n_2}{k_2}\right)\right),
    \end{align*}   
    and further implies that 
    \begin{align}
        \frac{1}{k_2} \log\left(\frac{n_1k_2}{k_1^2} \log\left(\frac{n_2}{k_2}\right)\right) \leq \psi_{21}  \leq \frac{2}{k_2} \log\left(\frac{n_1k_2}{k_1^2} \log\left(\frac{n_2}{k_2}\right)\right).\label{eq_equiv_psi21}
    \end{align}
    Now, by the relations $k_2 \leq 4 k_1\log(n_1/k_1)$ and $k_j \leq c' n_j$ for some small enough constant $c'>0$, we also have $\frac{n_2k_1}{k_2^2} \log\left(\frac{n_1}{k_1}\right)>2$, which yields
    \begin{align*}
        \psi_{12} = \frac{1}{k_1} \log\left(1+ \frac{n_2k_1}{k_2^2} \log\left(\frac{n_1}{k_1}\right)\right) \leq \frac{2}{k_1} \log\left( \frac{n_2k_1}{k_2^2} \log\left(\frac{n_1}{k_1}\right)\right),
    \end{align*}
    so that
    \begin{align}
        \frac{1}{k_1} \log\left( \frac{n_2k_1}{k_2^2} \log\left(\frac{n_1}{k_1}\right)\right) \leq \psi_{12} \leq \frac{2}{k_1} \log\left( \frac{n_2k_1}{k_2^2} \log\left(\frac{n_1}{k_1}\right)\right) \label{eq_equiv_psi12}
    \end{align}
    To obtain the desired inequality~\eqref{eq_target_simplification_max_rate}, it therefore suffices to show that 
    \begin{align}
        \frac{1}{k_1} \log\left( \frac{n_2k_1}{k_2^2} \log\left(\frac{n_1}{k_1}\right)\right) \leq \frac{1}{k_2} \log\left( \frac{n_1k_2}{k_1^2} \log\left(\frac{n_2}{k_2}\right)\right)\label{eq_new_target}
    \end{align}
    under the assumptions $\frac{k_1}{\bar c n_1} \leq \frac{k_2}{k_1} \leq \frac{1}{\log(n_2/k_2)}$ and $\frac{n_j}{k_j} \geq \frac{1}{c'}$ for $j\in\{1,2\}$, as well as $\frac{n_1}{k_1} \geq e \log\left(\frac{n_2}{k_2}\right)$. 
    
    To prove this, we let $x = \frac{k_2}{k_1}$, and $a_j = \frac{n_j}{k_j}$ for $j=1,2$ and introduce the function 
\begin{align*}
    f(x) = \log\left(a_1 x \log(a_2)\right) - x \log\left(\frac{a_2\log(a_1)}{x}\right).
\end{align*}
We note that the desired inequality~\eqref{eq_new_target} is equivalent to showing that $f(x) \geq 0$ for any $x \in \left[\frac{1}{\bar c a_1}, \frac{1}{\log(a_2)}\right]$ where $a_1 \geq e \log(a_2)$ and $a_1, a_2 \geq \frac{1}{c'}$.
We have, for any $x \in \left[\frac{1}{\bar c a_1}, \frac{1}{\log(a_2)}\right]$
\begin{align*}
    f'(x) &= \frac{1}{x} - \log(a_2 \log(a_1)) + \log(x) + 1 \\
    & = \log(x) + \frac{1}{x} - \log\left(\frac{a_2 \log(a_1)}{e}\right),\\
    f''(x) &= \frac{1}{x} - \frac{1}{x^2} \leq 0,
\end{align*}
since $x \leq \frac{1}{\log(a_2)} \leq 1$. 
Therefore, $f$ is concave, and to prove that $f \geq 0$ over $\left[\frac{1}{\bar c a_1}, \frac{1}{\log(a_2)}\right]$, it suffice to show that $f\left(\frac{1}{\bar c a_1}\right) \geq 0$ and $f \left(\frac{1}{\log(a_2)}\right)\geq 0$. 
We have
\begin{align*}
    f\left(\frac{1}{\bar c a_1}\right) &= \log\left(\frac{\log(a_2)}{\bar c}\right) - \frac{\log(\bar c a_1)}{\bar c a_1}  - \frac{\log(a_2)}{\bar c a_1} - \frac{\log\log(a_1)}{\bar c a_1}\\
    &\geq \log\left(\frac{\log(a_2)}{\bar c}\right)   - \frac{\log(a_2)}{\bar c a_1} - 1\\
    & \geq \frac{1}{2}\log\left(\frac{\log(a_2)}{\bar c}\right) - 1 \geq 0.
\end{align*}
In the last step, we used the fact that, since $a_1 \geq e \log(a_2)$, we also have $\bar c a_1 \geq \frac{1}{2} \frac{\log(a_2)}{\log \log(a_2/\bar c)}$
provided $a_2$ is large enough, which can be enforced by choosing $c'>0$ small enough. 
Similarly, we have
\begin{align*}
    f\left(\frac{1}{\log(a_2)}\right) &= \log(a_1) - \frac{1}{\log(a_2)} \log\Big(a_2\log(a_2)\log(a_1)\Big)\\
    & = \log(a_1) - 1 - \frac{\log \log(a_2)}{\log(a_2)} - \frac{\log(a_1)}{\log(a_2)}\\
    & \geq \log(a_1) - 2 - \frac{\log(a_1)}{\log(a_2)}\\
    &
    \geq 0,
\end{align*}
provided $a_1$ and $a_2$ are larger than suitably large constants, which can be enforced by choosing $c'$ small enough. 
This concludes the proof of the first claim.
\item Assume first that $k_1 < k_2 \log \left(\frac{n_2}{k_2}\right)$. 
Then we have 
    $$\begin{aligned}
    \psi_{12} + \psi_{21}
    & = \frac{1}{k_1} \log\left(1+ \frac{n_2k_1}{k_2^2} \log\left(\frac{n_1}{k_1}\right)\right) + \frac{1}{k_2} \log\left(1+ \frac{n_1k_2}{k_1^2} \log\left(\frac{n_2}{k_2}\right)\right)\\
    & \geq \frac{1}{k_1} \log\left( 1+ \frac{n_2}{4k_2} \right) + \frac{1}{k_2} \log\left(  \frac{n_1k_2}{k_1^2} \log\left(\frac{n_2}{k_2}\right)\right)\\
    & > \frac{1}{4k_1} \log\left( 1+ \frac{n_2}{k_2} \right) + \frac{1}{k_2} \log\left(\frac{n_1}{k_1}\right)\\
    & \geq \frac{1}{4} \big(\beta_{12} + \beta_{21}\big),
\end{aligned}
$$
which yields $R \geq \frac{1}{8} \big((\psi_{21} + \beta_{12}) + (\psi_{12}+\beta_{21})\big) \land \phi_{12} \land \phi_{21}$, as desired. 

Assume now that $k_1 \geq k_2 \log \left(\frac{n_2}{k_2}\right)$. 
Combining~\eqref{eq_equiv_psi21}, \eqref{eq_equiv_psi12} and~\eqref{eq_new_target} yields $\psi_{12} + \psi_{21} \asymp \psi_{21}$. 
We now show that $\psi_{21} \asymp \psi_{21} + \beta_{12}$. 
By assumption, we have
\begin{align*}
    \frac{n_1k_2}{k_1^2} \log\left(\frac{n_2}{k_2}\right)\geq \frac{1}{\bar c} >1,
\end{align*}
which implies that $\beta_{12} = \frac{1}{k_1} \log\left(\frac{n_2}{k_2}\right)$. 
Therefore, we obtain
\begin{align*}
        \beta_{12} & =\frac{1}{k_1} \log\left( \frac{n_2}{k_2} \right) & \\
        & \leq \frac{1}{k_1} \log\left( \frac{n_2}{k_2} \right) + \frac{1}{k_1} \log\left(\frac{4k_1\log\left(\frac{n_1}{k_1}\right)}{k_2}\right) & \text{by equation~\eqref{eq_s2<s1_log(d1/s1)}}\\
        &\leq \frac{2}{k_1} \log\left( \frac{n_2k_1}{k_2^2} \log\left(\frac{n_1}{k_1}\right)\right) & \text{ using $\log(4x) \leq 2 \log(x), ~\forall x\geq 4$}\\
        &\leq \frac{2}{k_2} \log\left( \frac{n_1k_2}{k_1^2} \log\left(\frac{n_2}{k_2}\right)\right) & \text{by equation~\eqref{eq_new_target}}\\
        & \leq 2 \psi_{21} & \text{ by equation~\eqref{eq_equiv_psi21}.}
    \end{align*}
    Therefore, we have $\psi_{12} + \psi_{21} \geq \psi_{21} + \beta_{12}/2$, hence $R \geq \frac{1}{2}\big(\psi_{21} + \beta_{12} \big) \land \phi_{12} \land \phi_{21}$ and the proof is complete.

\end{enumerate}
\end{proof}

\begin{lemma}\label{lem_psi_disappears}
        Assume that, for some constant $c>0$, we have $k_2 \leq c n_2$. 
        If $\frac{n_1k_2}{k_1^2} \log\left(\frac{n_2}{k_2}\right)\leq 1$, then we have $\psi_{21} \leq \phi_{12}$. 
    \end{lemma}

    \begin{proof}[Proof of Lemma~\ref{lem_psi_disappears}]
        Note that, when $\frac{n_1k_2}{k_1^2} \log\left(\frac{n_2}{k_2}\right)\leq 1$, we have $\frac{n_1}{k_1^2}\leq 1$, so that $\phi_{12} = \frac{n_1}{k_1^2}\log\left(1+\frac{n_2}{k_2^2}\right)$. We obtain
    \begin{align*}
        \psi_{21} &= \frac{1}{k_2} \log\left(1+\frac{n_1k_2}{k_1^2} \log\left(\frac{n_2}{k_2}\right)\right)\\
        & \asymp \frac{n_1}{k_1^2} \log\left(\frac{n_2}{k_2}\right)\\
        & \geq \phi_{12},
    \end{align*}
    which completes the proof.
    \end{proof}

    \begin{lemma}\label{lem_simplified_rate_for_UB}
        It holds that $R \gtrsim \big(\psi_{12} + \beta_{21}\big) \land \big(\psi_{21} + \beta_{12}\big) \land \phi_{12} \land \phi_{21}$.
    \end{lemma}

    \begin{proof}[Proof of Lemma~\ref{lem_simplified_rate_for_UB}]
    Assume first that, for some constant $c>0$, we have $k_1 \geq c n_1$ or $k_2 \geq c n_2$, and by symmetry, assume we have $k_1 \geq c n_1$. 
    Then 
    \begin{align*}
        \psi_{12} = \frac{1}{k_1} \log\left(1+ \frac{n_2}{k_2^2} \log\left(e {n_1 \choose k_1}\right)\right) \geq \frac{1}{k_1} \log\left(1+ \frac{n_2}{k_2^2} \right) \asymp \phi_{12},
    \end{align*}
    which yields that
    \begin{align*}
        R &= (\psi_{12} + \psi_{21}) \land \phi_{12} \land \phi_{21}\\
        & \asymp \phi_{12} \land \phi_{21} \\
        &\geq \big(\psi_{12} + \beta_{21}\big) \land \big(\psi_{21} + \beta_{12}\big) \land \phi_{12} \land \phi_{21}.
    \end{align*}
    From now on, assume that $k_1 \leq c n_1$ and $k_2 \leq c n_2$, and assume without loss of generality that $\frac{n_1}{k_1} \geq e\log\left(\frac{n_2}{k_2}\right)$, by lemma~\ref{lem_2.6}. 

    If $k_1^2 \geq \bar c n_1 k_2$, then the result follows by Lemma~\ref{lem_simplify_dense_allrates}. Now, assume that $k_1^2 < \bar c n_1 k_2$. If $\frac{k^2_2}{n_2} \leq 2e\frac{k_1^2}{n_1}$, then the result follows by Lemma~\ref{lem_simplify_s1d1big_truncchi2}. 
    Else, we have $k_1^2 \geq \bar c n_1 k_2$ and the result follows by Lemmas~\ref{lem_simplify_s2d2big_truncchi2} and~\ref{lem_simplification_rate_max_test}.
        
    \end{proof}
\section{Proofs for upper bound}

\subsection{Analysis of total degree test}
Recall that the total degree test is defined as $\Delta_\lin^h = \one(t_\lin > h)$ for a choice of threshold $h > 0$, where
\begin{equation}
        t_{\lin} = \frac{1}{\sqrt{n_1n_2p_0(1-p_0)}}\sum_{\substack{i = 1}}^{n_1}\sum_{j = 1}^{n_2}(A_{ij} - p_0)
\end{equation}
\begin{lemma}\label{lem_linear_test}
    Let $\alpha \in (0,1)$ be given and define $h_\alpha = \sqrt{4 \log(2/\alpha)}$. Suppose that  $n_1n_2p_0 \geq \frac{4}{27}h_\alpha^2$ and $p_0 \in (0,\frac14]$. Then there exists a constant $C_\delta > 0$ such that if
    \[\delta^2 \geq C_\delta p_0(1-p_0)\frac{n_1n_2}{k_1^2k_2^2},\]
    then the linear test with threshold $h = h_\alpha$ satisfies
    \[\cR(\Delta_\lin^h, \delta) \leq \alpha.\]
\end{lemma}
\begin{proof}[Proof of Lemma ~\ref{lem_linear_test}]
    We begin by an analysis of the Type 1 error of $\Delta_\lin^h$. Under the null hypothesis, it holds that $\EE[t_\lin] = 0$ and $\var(t_\lin) = 1$. Let $\sigma^2 = p_0(1-p_0)n_1n_2$. Note that under our assumptions $n_1n_2p_0 \geq \frac{4}{27}h_\alpha^2$ and $p_0 \leq \frac14$, it holds
    \begin{align*}
        \sigma &= \sqrt{p_0(1-p_0)n_1n_2} \\
        &\geq \sqrt{\frac34p_0n_1n_2} \\
        &\geq \sqrt{\frac34\frac{4}{27}h^2} \\
        &= \frac13 h
    \end{align*}
    
    By direct calculation, we have
    \begin{align*}
        \bbP_0(\Delta_\lin^h = 1) &= \bbP_0(t_\lin > h) \\
        &= \bbP_0\big(\sum_{i = 1}^{n_1}\sum_{j = 1}^{n_2}(A_{ij} - p_0) > \sigma h\big) \\
        &\leq \exp\Big(-\frac{\frac12h^2\sigma^2}{\sigma^2 + \frac13h\sigma}\Big) \quad \text{(by Bernstein's inequality)} \\
        &\leq \exp\Big(-\frac{\frac12 \sigma^2 h^2}{2\sigma^2}\Big) \quad \text{(since $\sigma \geq \frac13 h$)} \\
        &= \exp\Big(-\frac14 h^2\Big)\\
        &\leq \exp\big(- \log(2/\alpha)\big) \quad \text{(since $h = \sqrt{4\log(2/\alpha)}$)} \\
        &= \frac{\alpha}{2}.
    \end{align*}
    Now we turn our attention to the Type 2 error. Under the alternative hypothesis, there exist subsets of indices $K_1 \subset [n_1]$ and $K_1 \subset [n_2]$ of sizes $|K_1| = k_1$ and $|K_2| = k_2$ such that $\EE[A_{ij}] = p_0 + \delta$ for $(i,j) \in K_1 \times K_2$ and $\EE[A_{ij}] = p_0$ otherwise. Since the $\{A_{ij}\}_{i \in[n_1] ,j\in [n_2]}$ are independent Bernoulli random variables, we can directly compute
    \begin{align*}
        \EE_\bP[t_\lin] &= \frac{1}{\sigma}k_1k_2 \delta, \\
        \var_\bP(t_\lin) &= \frac{1}{\sigma^2}\big(k_1k_2 p_1(1 - p_1) + (n_1 - k_1)(n_2 - k_2)p_0(1 - p_0)\big),
    \end{align*}
    where we let $p_1 = p_0 + \delta$. Now under the assumption that $\delta^2 \geq C_\delta p_0(1-p_0)\frac{n_1n_2}{k_1^2k_2^2}$, it holds
    \begin{align*}
        \EE_\bP[t_\lin] &= \frac{1}{\sigma}k_1k_2 \delta \\
        &= \frac{k_1k_2}{\sqrt{n_1n_2p_0(1-p_0)}}\delta \\
        &\geq C_\delta.
    \end{align*}
    Therefore, for $C_\delta$ taken large enough, we can ensure that $\EE_\bP[t_\lin] \geq 2h = 2\sqrt{4 \log(2/\alpha)}$. Therefore, we may perform the following calculation:
    \begin{align*}
        \bbP_\bP(\Delta_\lin^h = 0) &= \bbP_\bP(t_\lin \leq h) \\
        &= \bbP_\bP(t_\lin - \EE_\bP t_\lin \leq h - \EE_\bP t_\lin) \\
        &= \bbP_\bP\big((t_\lin - \EE_\bP t_\lin)^2 \geq (h - \EE_\bP t_\lin)^2\big) \\
        &\leq \bbP_\bP\big((t_\lin - \EE_\bP t_\lin)^2 \geq \frac14(\EE_\bP t_\lin)^2\big) \\
        &\leq 4\frac{\var_\bP(t_\lin)}{(\EE_\bP t_\lin)^2},
    \end{align*}
    where the final inequality uses Markov's. From here, we apply our closed form expressions of $\var_\bP(t_\lin)$ and $\EE_\bP [t_\lin]$ derived above to compute:
    \begin{align*}
        \frac{\var_\bP(t_\lin)}{(\EE_\bP t_\lin)^2} &= \frac{k_1k_2 p_1(1 - p_1) + (n_1 - k_1)(n_2 - k_2)p_0(1 - p_0)}{k_1^2k_2^2\delta^2} \\
        &\leq \frac{k_1k_2(p_0 + \delta) + \sigma^2}{k_1^2k_2^2\delta^2} \\
        &= \text{I} + \text{II} + \text{III} ,  
\end{align*}
where we define $\text{I} = \frac{p_0}{k_1k_2\delta^2}$, $\text{II} = \frac{1}{k_1k_2\delta}$, and $\text{III} = \frac{\sigma^2}{k^2_1k^2_2 \delta^2}$. We control each of these terms separately. First, we have
\begin{align*}
    \text{I} &= \frac{p_0}{k_1k_2\delta^2} \\
    &\leq \frac{k_1k_2 p_0}{n_1n_2p_0(1-p_0)C_\delta} \quad \text{(since $\delta^2 \geq C_\delta\frac{n_1n_2p_0(1 - p_0)}{k^2_1k^2_2}$ )} \\
    &\leq \frac{4}{3C_\delta} \quad \text{(since $k_i \leq n_i$ for $i \in \{1,2\}$ and $1 - p_0 \geq \frac34$)} \\
    &\leq \frac{\alpha}{24},
\end{align*}
where the final inequality holds for $C_\delta$ taken sufficiently large. Now we consider $\text{II}$. Recall that $n_1 n_2 p_0 \geq \frac{4}{27} h^2 = \frac{8}{27}\log(2/\alpha)$. Therefore, $p_0 \geq \frac{C}{n_1n_2}$ where $C > 0$ is a constant that depends on $\alpha$. We thus have
\begin{align*}
    \delta &\geq \sqrt{C_\delta p_0 (1 - p_0)\frac{n_1n_2}{k_1^2k_2^2}} \\
    &\geq \sqrt{C_\delta \frac 34}\sqrt{p_0\frac{n_1n_2}{k_1^2k_2^2}} \quad \text{(since $1 - p_0 \geq \frac34$)} \\
    &\geq \frac{C'}{k_1k_2},
\end{align*}
where $C' = \sqrt{\frac 34 C_\delta C}$. Using this lower bound on $\delta$, we have
\begin{align*}
    \text{II} &= \frac{1}{k_1k_2\delta} \\
    &\leq \frac{1}{C'} \\
    &\leq \frac{\alpha}{24}.
\end{align*}
where the final inequality holds for $C_\delta$ taken sufficiently large. Finally, to control $\text{III}$ we have
\begin{align*}
    \text{III} &= \frac{\sigma^2}{k^2_1k^2_2 \delta^2} \\
    &\leq \frac{n_1n_2p_0(1- p_0)k_1^2k_2^2}{C_\delta n_1n_2p_0(1- p_0)k_1^2k_2^2} \quad \text{(since $\delta^2 \geq C_\delta\frac{n_1n_2p_0(1 - p_0)}{k^2_1k^2_2}$)}\\
    &= \frac{1}{C_\delta} \\
    &\leq \frac{\alpha}{24}
\end{align*}
where the final inequality holds for $C_\delta$ taken large enough. Combining these bounds on $\text{I}, \text{II},$ and $\text{III}$ gives us
\begin{align*}
    \bbP_\bP(\Delta_\lin^h = 0) &\leq4\frac{\var_\bP(t_\lin)}{(\EE_\bP t_\lin)^2} \\
    &\leq 4(\text{I} + \text{II} + \text{III}) \\
    &\leq 4(3 \frac{\alpha}{24}) \\
    &= \frac{\alpha}{2}.
\end{align*}
Combining this inequality with our bound on the Type I error, we have
\[\cR(\Delta_\lin^h, \delta) \leq \alpha,\]
and the proof is complete.
\end{proof}

\subsection{Analysis of the truncated degree test}

The following lemma controls the risk of $\Delta^{h}_{\chisqlin, 1}$. The analysis of $\Delta^{h}_{\chisqlin, 2}$ is analogous.
\begin{lemma}\label{lem_truncated_degree_test}
    Let $\alpha \in (0,1)$ and define $h_\alpha = C^*\left(\sqrt{n_2\exp\left(-c' \log(1 + \frac{n_2}{k_2^2})\right)\log(2/\alpha)} + \log(2/\alpha)\right)$ for constants $c', C^* > 0$ to be determined later. Suppose that $n_2 \geq c k_2^2$ for a constant $c > 0$. Then there exist constants $C_\delta, C_\alpha > 0$ and $C \geq \frac83$ such that if $\frac14 \geq p_0 \geq \frac{C_\alpha}{n_1} \log\left(1+ \frac{n_2}{k_2^2}\right)$ and
    \[\delta^2 \geq C_\delta p_0(1-p_0)\frac{n_1}{k_1^2}\log\left(1+\frac{n_2}{k_2^2}\right),\]
    then the truncated degree test with threshold $h = h_\alpha$ and $\tau = \sqrt{C\log\left(1 + \frac{n_2}{k_2^2}\right)}$ satisfies
    \[\cR(\Delta^{h_\alpha}_{\chisqlin, 1}, \delta) \leq \alpha.\]
\end{lemma}

\begin{proof}[Proof of Lemma~\ref{lem_truncated_degree_test}]
    First, we control the Type 1 error of $\Delta^{h_\alpha}_{\chisqlin, 1}$. Let $\sigma = \sqrt{n_1p_0(1-p_0)}$. By Lemma \ref{leq_truncchisq_concentration}, there exist constants $C', c, c', C^* > 0$ such that if $\tau \in [C', c\sigma]$, then
    \begin{align*}
    \bbP_0(\Delta^{h_\alpha}_{\chisqlin, 1} > h_\alpha) \leq \frac{\alpha}{2}.
    \end{align*}
    We impose that $c' \geq 1$ by taking $C$ in the definition of $\tau$ sufficiently large. Therefore we just need to verify that $\tau \in [C', c\sigma]$. Since $n_2 \geq ck_2^2$, it holds
    \begin{align*}
        \tau &= \sqrt{C\log(1 + \frac{n_2}{k_2^2})} \\
        &> \sqrt{C\log(1 + c)} \\
        &\geq C',
    \end{align*}
    where the final inequality holds for $C$ taken large enough. On the other hand, since $\frac14 \geq p_0 \geq \frac{C_\alpha}{n_1}\log\left(1 + \frac{n_2}{k_2^2}\right)$, we have
    \begin{align*}
        c\sigma &= c\sqrt{n_1p_0(1-p_0)} \\
        &\geq c\sqrt{C_\alpha \frac34 \log\left(1 + \frac{n_2}{k_2^2}\right)} \\
        &\geq \sqrt{C\log\left(1 + \frac{n_2}{k_2^2}\right)} \quad \text{(for $C_\alpha$ large enough)}\\
        &= \tau.
    \end{align*}
    Therefore $\tau \in [C', c\sigma]$, and the desired bound on the Type 1 error holds.

    Now we aim to control the Type 2 error. Under the alternative hypothesis, there exist sets $K_1 \in \mathcal P_{k_1}(n_1)$ and $K_2 \in \mathcal P_{k_2}(n_2)$ such that $\EE[\A] = \bP$ with entries $P_{ij} \geq p_0 + \delta$ if $(i,j) \in K_1 \times K_2$ and $P_{ij} = p_0$ otherwise. Let $\theta = \frac{k_1 \delta}{\sigma}$. Then by Lemma \ref{lem_truncchisq_mean}, there exist constants $\tilde{c}, c, C' > 0$ such that if $\tau \in [C', c\sigma]$, $k_1 (p_0 + \delta) \geq C'$, and $\theta \geq C'\tau$, then
    \begin{equation}\label{eq_trunc-mean-lb}
    \EE\left[\sum_{j = 1}^{n_2}\big(W_{j} - \nu^{k_1}_{\tau}\big)\one(\bar{A}_j > \tau)\right] \geq \tilde{c}k_2\theta \sigma \log\left(1 + \frac{\theta}{\sigma}\right).
    \end{equation}  
    As when we controlled the Type 1 error, we enforce $\tau \in [C', c\sigma]$ by taking $C$ in the definition of $\tau$ and $C_\alpha$ sufficiently large, in that order. To verify $k_1(p_0 + \delta) \geq C'$, we calculate
    \begin{align*}
    k_1(p_0 + \delta) &\geq k_1\delta \\
        &\geq \sqrt{n_1 p_0(1-p_0)C_\delta \log\left(1 + \frac{n_2}{k_2^2}\right)} \\
        &\geq \sqrt{C C_\delta \log\left(1 + \frac{n_2}{k_2^2}\right)} \quad \text{(since $\sigma \gtrsim 1$, as shown above)}\\
        &\geq \sqrt{C C_\delta \log\left(1 + c\right)} \\
        &\geq C',
    \end{align*}
    where the finally inequality holds for $C_\delta$ taken sufficiently large. Finally, to verify $\theta \geq C'\tau$, we have
    \begin{align*}
        \theta &= \frac{k_1\delta}{\sqrt{n_1p_0(1 - p_0)}} \\
        &\geq \sqrt{C_\delta\log\Big(1 + \frac{n_2}{k_2^2}\Big)} \\
        &\geq C'\sqrt{C\log\Big(1 + \frac{n_2}{k_2^2}\Big)} \quad \text{(for $C_\delta$ taken large enough)} \\
        &= C'\tau.
    \end{align*}
Therefore, (\ref{eq_trunc-mean-lb}) holds by Lemma \ref{lem_truncchisq_mean}. Additionally, by Lemma \ref{lem_variance_trunc_chi2}, there exist constants $C_1, c, C' > 0$ such that if $\tau \in [C', c\sigma]$,$k_1 (p_0 + \delta) \geq C'$,  and $\theta \geq 2\tau$, then
\begin{equation}\label{eq_trunc-var-ub}
    \var\left(\sum_{j = 1}^{n_2}\big(W_{j} - \nu^{k_1}_{\tau}\big)\one(\bar{A}_{j} > \tau)\right) \leq C_1 k_2 (\sigma^2 + \theta \sigma) \log^2\left(1 + \frac{\theta}{\sigma}\right) + C_1 (n_2 - k_2)\left(\sigma \tau \log(1 + \frac{\tau}{\sigma})\right)^2\exp\big(-\frac38 \tau^2\big).
\end{equation}
As before, we enforce $\tau \in [C', c\sigma]$,$k_1 (p_0 + \delta) \geq C'$,  and $\theta \geq 2\tau$ by taking $C$, $C_\alpha$, and $C_\delta$ sufficiently large in that order. To control the Type 2 error with (\ref{eq_trunc-mean-lb}) and (\ref{eq_trunc-var-ub}) and Chebyshev's inequality, we need to verify that the right-hand side of (\ref{eq_trunc-mean-lb}) is at least $2h_\alpha$. In what follows, we define
\[\bar{\delta}^2 := C_\delta p_0(1-p_0)\frac{n_1}{k^2_1}\log\Big(1 + \frac{n_2}{k_2^2}\Big),\]
so that $\delta \geq \bar{\delta}$. First, we have
\begin{align*}
    \log\Big(1 + \frac{\theta}{\sigma}\Big) &= \log\Big(1 + \frac{k_1\delta}{\sigma^2}\Big) \\
    &\geq \log\Big(1 + \frac{k_1\bar{\delta}}{\sigma^2}\Big).
\end{align*}
We now aim to show that there exists a constant $\bar{C} > 0$ such that $\frac{k_1\bar{\delta}}{\sigma^2} \leq \bar{C}$. By the definitions of $\bar{\delta}$ and $\sigma^2$, it suffices to show that 
\[\log\Big(1 + \frac{n_2}{k_2^2}\Big) \leq \bar{C}\sigma^2.\]
Recall that $\tau^2 = C\log\big(1 + \frac{n_2}{k_2^2}\big)$. Since we have already shown that $\tau^2 \lesssim \sigma^2$, it holds that there exists a constant $\bar{C} > 0$ such that $\log\Big(1 + \frac{n_2}{k_2^2}\Big) \leq \bar{C}\sigma^2$ as desired. Therefore, there exists a constant $\bar{c} > 0$ such that $\log\big(1 + \frac{\theta}{\sigma}\big) \geq \log\big(1 + \frac{k_1\bar{\delta}}{\sigma}\big) \geq \bar{c}\frac{k_1\bar{\delta}}{\sigma}$. Thus we have
\begin{align*}
    \tilde{c}k_2\theta \sigma \log\left(1 + \frac{\theta}{\sigma}\right) &\geq  ck_2\theta k_1 \bar{\delta} \quad \text{(for $c = \tilde{c}\bar{c}$)} \\
    &\geq c k_2 \frac{(k_1 \bar{\delta})^2}{\sigma} \\
    &= c C_\delta k_2 \sigma \log\Big(1 + \frac{n_2}{k_2^2}\Big) \\
    &\geq C'_\delta k_2 \log\Big(1 + \frac{n_2}{k_2^2}\Big),
\end{align*}
where the final inequality uses $\sigma \gtrsim 1$, and $C'_\delta > 0$ is a positive multiple of $C_\delta$. Since we have enforced $c' \geq 1$ in the definition of $h_\alpha$, it holds
\begin{align*}
h_\alpha &= C^*\left(\sqrt{n_2\exp\left(-c' \log(1 + \frac{n_2}{k_2^2})\right)\log\left(\frac{2}{\alpha}\right)}\right) \\
&\leq C^*\left(k_2\sqrt{\frac{\frac{n_2}{k_2^2}\log(2/\alpha)}{1 + \frac{n_2}{k_2^2}}}\right) \\
&\leq C^* \left(k_2\sqrt{\log(2/\alpha)} \right) \\
&\leq \frac{C'_\delta}{2} \left(k_2 \log\Big(1 + \frac{n_2}{k_2^2}\Big)\right),
\end{align*}
where the final inequality holds for $C_\delta$ taken sufficiently large, and uses that $\frac{n_2}{k_2^2} \geq c > 0$. Therefore, we have shown that for $C_\delta$ taken sufficiently large, it holds
\begin{align*}
    \EE\left[\sum_{j = 1}^{n_2}\big(W_{j} - \nu^{k_1}_{\tau}\big)\one(\bar{A}_j > \tau)\right] &\geq \tilde{c}k_2\theta \sigma \log\left(1 + \frac{\theta}{\sigma}\right) \\
    &\geq 2h_\alpha.
\end{align*}
Letting $U_ := \sum_{j = 1}^{n_2}\big(W_{j} - \nu^{k_1}_{\tau}\big)\one(\bar{A}_{j} > \tau)$, we therefore can conclude
\begin{align*}
\bbP_{\bP}\big(\Delta^{h_\alpha}_{\chisqlin, 1} = 0\big) &= \bbP_{\bP}\big(U \leq h_\alpha\big)\\
    &= \bbP_{\bP}\big(U - \EE[U]\leq h_\alpha - \EE[U]\big) \\
    &=\bbP_{\bP}\Big(\big(U - \EE[U]\big)^2\geq \big(h_\alpha - \EE[U]\big)^2\Big) \quad \text{(since $\EE[U] \geq 2h_\alpha$)} \\
    &\leq \bbP_{\bP}\Big(\big(U - \EE[U]\big)^2\geq \big(\EE[U]\big)^2\Big) \\
    &\leq \frac{\var\big(U\big)}{\big(\EE[U]\big)^2} \quad \text{(by Markov's inequality)} \\
    &\leq \frac{C_1 k_2 (\sigma^2 + \theta \sigma) \log^2\left(1 + \frac{\theta}{\sigma}\right) + C_1 (n_2 - k_2)\left(\sigma \tau \log(1 + \frac{\tau}{\sigma})\right)^2\exp\big(-\frac38 \tau^2\big)}{\left(\tilde{c}k_2\theta \sigma \log\left(1 + \frac{\theta}{\sigma}\right)\right)^2} \\
    &= \text{I} + \text{II},
\end{align*}
where we define
\[\text{I} = \frac{C_1 k_2 (\sigma^2 + \theta \sigma) \log^2\left(1 + \frac{\theta}{\sigma}\right)}{\left(\tilde{c}k_2\theta \sigma \log\left(1 + \frac{\theta}{\sigma}\right)\right)^2},\]
and 
\[\text{II} = \frac{C_1 (n_2 - k_2)\left(\sigma \tau \log(1 + \frac{\tau}{\sigma})\right)^2\exp\big(-\frac38 \tau^2\big)}{\left(\tilde{c}k_2\theta \sigma \log\left(1 + \frac{\theta}{\sigma}\right)\right)^2}.\]
We control these terms separately. First we have
\begin{align*}
    \text{I} &= \frac{C_1 k_2 (\sigma^2 + \theta \sigma) \log^2\left(1 + \frac{\theta}{\sigma}\right)}{\left(\tilde{c}k_2\theta \sigma \log\left(1 + \frac{\theta}{\sigma}\right)\right)^2} \\
    &= \frac{C_1(\sigma^2 + \theta \sigma)}{\tilde{c}^2k_2\theta^2 \sigma^2 } \\
    &= \frac{C_1}{\tilde{c}^2k_2\theta^2 }  + \frac{C_1}{\tilde{c}^2k_2\theta \sigma }\\
    &\leq \frac{C_1}{C'_\delta \tilde{c}^2 k_2\sqrt{\log\Big(1 + \frac{n_2}{k_2^2}\Big)}} \\
    &+ \frac{C_1}{C'_\delta \tilde{c}^2 k_2 \log\Big(1 + \frac{n_2}{k_2^2}\Big)}\\
    &\leq \frac{C_1}{C'_\delta \tilde{c}^2 k_2 \sqrt{\log(1 + c)}} + \frac{C_1}{C'_\delta \tilde{c}^2 k_2 \log(1 + c)} \\
    &\leq \frac{\alpha}{4},
\end{align*}
where the final inequality holds for $C_\delta$ taken large enough. Recall that $C \geq \frac{8}{3}$ in the definition of $\tau$. Then it holds
\begin{align*}
    \text{II} &= \frac{C_1 (n_2 - k_2)\left(\sigma \tau \log(1 + \frac{\tau}{\sigma})\right)^2\exp\big(-\frac38 \tau^2\big)}{\left(\tilde{c}k_2\theta \sigma \log\left(1 + \frac{\theta}{\sigma}\right)\right)^2} \\
    &\leq \frac{C_1n_2}{\tilde{c}^2k_2^2\Big(1 + \frac{n_2}{k_2^2}\Big)} \cdot \frac{\left(\tau \log\big(1 + \frac{\tau}{\sigma}\big)\right)^2}{\left(\theta \log\big(1 + \frac{\theta}{\sigma}\big)\right)^2} \\
    &\leq \frac{C_1}{\tilde{c}^2} \cdot \frac{\left(C\theta\log(1 + c)\right)^2}{\left(\theta \log\big(1 + \frac{\theta}{\sigma}\big)\right)^2} \quad \text{(since $\tau \leq c \sigma $ and $\tau \leq C \theta$)} \\
    &\leq \frac{C_1}{\tilde{c}^2} \cdot \frac{ \left(C\log(1 + c)\right)^2}{C'_\delta\left(k_2\log\Big(1 + \frac{n_2}{k_2^2}\Big)\right)} \\
    &\leq \frac{C_1}{\tilde{c}^2} \cdot \frac{ \left(C\log(1 + c)\right)^2}{C'_\delta\left(k_2\log\Big(1 + c\Big)\right)} \\
    &\leq \frac{\alpha}{4},
\end{align*}
where again the final inequality holds for $C_\delta$ taken sufficiently large. Therefore, we have shown
\[\bbP_{\bP}\big(\Delta^{h_\alpha}_{\chisqlin, 1} = 0\big) \leq \text{I} + \text{II} \leq \frac{\alpha}{2}.\]
Combining this with our bound on the Type 1 error, it holds
\[\cR(\Delta^{h_\alpha}_{\chisqlin, 1}, \delta) \leq \alpha,\]
and the proof is complete.
\end{proof}

\subsection{Analysis of the max truncated degree test}
The following lemma controls the risk of $\Delta^h_{\chisqmax, 1}$. The analysis of $\Delta^h_{\chisqmax, 2}$ is analogous.

\begin{lemma}\label{lem_max_truncated_degree_test}
    Let $\alpha \in (0,1)$ and define 
    $$h_\alpha = C^*\left(\sqrt{n_2\exp\left(-c' \log\left(1 + \frac{n_2}{k_2^2}\log{n_1 \choose k_1}\right)\right)\log\left(\frac{2}{\alpha}{n_1 \choose k_1}\right)} + \log\left(\frac{2}{\alpha}{n_1 \choose k_1}\right)\right)$$ 
    for constants $c', C^* > 0$ to be determined later. Suppose that $\frac{n_2}{k_2^2}\log {n_1 \choose k_1} > c$ for a constant $c \geq 0$. Then there exist constants $C_\delta, C_\alpha > 0$ and $C \geq \frac{8}{3}$ such that if $\frac14 \geq p_0 \geq \frac{C_\alpha}{k_1k_2} \log\left(e{n_1 \choose k_1}{n_2 \choose k_2}\right)$ and
    \[\delta^2 \geq C_\delta p_0(1-p_0)\Bigg(\frac{1}{k_1}\log\Big(1 + \frac{n_2}{k_2^2}\log {n_1 \choose k_1}\Big) + \frac{1}{k_2k_1}\log {n_1 \choose k_1}\Bigg),\]
    then the truncated degree test with threshold $h = h_\alpha$ and $\tau = \sqrt{C\log\left(1 + \frac{n_2}{k_2^2}\log{n_1 \choose k_1}\right)}$ satisfies
    \[\cR(\Delta^{h_\alpha}_{\chisqlin, 1}, \delta) \leq \alpha.\]
\end{lemma}
\begin{proof}[Proof of Lemma \ref{lem_max_truncated_degree_test}]
    Let $\sigma = \sqrt{k_1p_0(1-p_0)}$. We begin by controlling the Type 1 error.
    \begin{align*}
        \bbP_0 (\Delta^{h_\alpha}_{\chisqmax, 1} = 1) &= \bbP_0(t_{\chisqmax,1} > h_\alpha) \\
        &= \bbP_0\left(\max \left\{\sum_{j = 1}^{n_2}\big(W_{J_1,j} - \nu^{k_1}_{\tau}\big)\one(t_{J_1, j} > \tau)\,\Big|\, J_1 \in \mathcal{P}_{k_1}(n_1)\right\} > h_\alpha \right) \\
        &\leq {n_1 \choose k_1} \bbP_0\left(\sum_{j = 1}^{n_2}\big(W_{J_1,j} - \nu^{k_1}_{\tau}\big)\one(t_{J_1, j} > \tau) > h_\alpha\right) \\
        &\leq {n_1 \choose k_1}\exp\left(-\log\left(\frac{2}{\alpha}{n_1 \choose k_1}\right)\right) \\
        &= \frac{\alpha}{2},
    \end{align*}
    where the second inequality follows from Lemma \ref{leq_truncchisq_concentration} as long as $h_\alpha$ is defined with the constants $C^*$ and $c'$ obtained from Lemma \ref{leq_truncchisq_concentration} and $\tau \in [C', c\sigma]$ for some constants $C', c > 0$. We impose that the constant $c'$ is made to be greater than $1$ by taking the constant $C$ in the definition of $\tau$ sufficiently large. Therefore, we just need to verify that $\tau \in [C', c\sigma]$ to finalize our control of the Type 1 error. First, we have
    \begin{align*}
        \tau &= \sqrt{C\log\left(1 + \frac{n_2}{k_2^2}\log{n_1 \choose k_1}\right)} \\
        &> \sqrt{C\log\left(1 + c\right)} \\
        &\geq C',
    \end{align*}
    where the final inequality holds for $C$ taken sufficiently large. Additionally, using $\frac14 \geq p_0 \geq \frac{C_\alpha}{k_1k_2} \log\left(e{n_1 \choose k_1}{n_2 \choose k_2}\right)$ we have
    \begin{align*}
        c\sigma &= c\sqrt{k_1 p_0 (1 - p_0)} \\
        &\geq c\sqrt{\frac34 C_\alpha\log\left(e\frac{n_2}{k_2}{n_1 \choose k_1}\right)} \\
        &\geq \sqrt{C\log\left(1 + \frac{n_2}{k_2^2}\log{n_1 \choose k_1}\right)} \quad \text{(for $C_\alpha$ large enough)}\\
        &= \tau.
    \end{align*}
    Therefore  $\tau \in [C', c\sigma]$, and the bound on the Type 1 error holds. 

    Now we turn our attention to the Type 2 error. Under the alternative hypothesis, there exist sets $K_1 \in \mathcal P_{k_1}(n_1)$ and $K_2 \in \mathcal P_{k_2}(n_2)$ such that $\EE[\A] = \bP$ with entries $P_{ij} \geq p_0 + \delta$ if $(i,j) \in K_1 \times K_2$ and $P_{ij} = p_0$ otherwise. Let $\theta = \frac{k_1 \delta}{\sigma}$. Then by Lemma \ref{lem_truncchisq_mean}, there exist constants $\tilde{c}, c, C' > 0$ such that if $\tau \in [C', c\sigma]$, $k_1 (p_0 + \delta) \geq C'$, and $\theta \geq C'\tau$, then
    \begin{equation}\label{eq_max-trunc-mean-lb}
    \EE\left[\sum_{j = 1}^{n_2}\big(W_{K_1,j} - \nu^{k_1}_{\tau}\big)\one(t_{K_1, j} > \tau)\right] \geq \tilde{c}k_2\theta \sigma \log\left(1 + \frac{\theta}{\sigma}\right).
    \end{equation}
    As when we controlled the Type 1 error, we enforce $\tau \in [C', c\sigma]$ by taking $C$ in the definition of $\tau$ and $C_\alpha$ sufficiently large, in that order. To verify $k_1(p_0 + \delta) \geq C'$, we calculate
    \begin{align*}
        k_1(p_0 + \delta) &\geq k_1\delta \\
        &\geq \sqrt{k_1 p_0(1-p_0)C_\delta \log\left(1 + \frac{n_2}{k_2^2}\log {n_1 \choose k_1}\right)} \\
        &\geq \sqrt{\frac34 C_\alpha\log\left(e\frac{n_2}{k_2}{n_1 \choose k_1}\right) C_\delta \log\left(1 + \frac{n_2}{k_2^2}\log {n_1 \choose k_1}\right)} \\
        &\geq \sqrt{\frac34 C_\alpha C_\delta \log\left(1 + c\right)} \\
        &\geq C',
    \end{align*}
    where the finally inequality holds for $C_\delta$ taken sufficiently large. Finally, to verify $\theta \geq C'\tau$, we have
    \begin{align*}
        \theta &= \frac{k_1\delta}{\sqrt{k_1p_0(1 - p_0)}} \\
        &\geq \sqrt{C_\delta \left(\log\Big(1 + \frac{n_2}{k_2^2}\log {n_1 \choose k_1}\Big) + \frac{1}{k_2}\log {n_1 \choose k_1}\right)} \\
        &\geq C'\sqrt{C\log\Big(1 + \frac{n_2}{k_2^2}\log {n_1 \choose k_1}\Big)} \quad \text{(for $C_\delta$ taken large enough)} \\
        &= C'\tau.
    \end{align*}
Therefore, (\ref{eq_max-trunc-mean-lb}) holds due to Lemma \ref{lem_truncchisq_mean}. Additionally, by Lemma \ref{lem_variance_trunc_chi2}, there exist constants $C_1, c, C' > 0$ such that if $\tau \in [C', c\sigma]$, $k_1 (p_0 + \delta) \geq C'$,  and $\theta \geq 2\tau$, then
\begin{equation}\label{eq_max-trunc-var-ub}
    \var\left(\sum_{j = 1}^{n_2}\big(W_{K_1,j} - \nu^{k_1}_{\tau}\big)\one(t_{K_1, j} > \tau)\right) \leq C_1 k_2 (\sigma^2 + \theta \sigma) \log^2\left(1 + \frac{\theta}{\sigma}\right) + C_1 (n_2 - k_2)\left(\sigma \tau \log(1 + \frac{\tau}{\sigma})\right)^2\exp\big(-\frac38 \tau^2\big).
\end{equation}
As before, we enforce $\tau \in [C', c\sigma]$,$k_1 (p_0 + \delta) \geq C'$,  and $\theta \geq 2\tau$ by taking $C$, $C_\alpha$, and $C_\delta$ sufficiently large in that order. We aim to control the Type 2 error by combining (\ref{eq_max-trunc-mean-lb}) and (\ref{eq_max-trunc-var-ub}) in an application of Chebyshev's inequality. Before we can do this, we need to verify that the right-hand side of (\ref{eq_max-trunc-mean-lb}) is at least $2h_\alpha$. In what follows, we define
\[\bar{\delta}^2 := C_\delta p_0(1-p_0)\Bigg(\frac{1}{k_1}\log\Big(1 + \frac{n_2}{k_2^2}\log {n_1 \choose k_1}\Big) + \frac{1}{k_2k_1}\log {n_1 \choose k_1}\Bigg),\]
so that $\delta \geq \bar{\delta}$ by assumption. First, we have
\begin{align*}
    \log\Big(1 + \frac{\theta}{\sigma}\Big) &= \log\Big(1 + \frac{k_1\delta}{\sigma^2}\Big) \\
    &\geq \log\Big(1 + \frac{k_1\bar{\delta}}{\sigma^2}\Big).
\end{align*}
We now aim to show that there exists a constant $\bar{C} > 0$ such that $\frac{k_1\bar{\delta}}{\sigma^2} \leq \bar{C}$. By the definitions of $\bar{\delta}$ and $\sigma^2$, it suffices to show that 
\[\log\Big(1 + \frac{n_2}{k_2^2}\log {n_1 \choose k_1}\Big) + \frac{1}{k_2}\log {n_1 \choose k_1} \leq \bar{C}\sigma^2.\]
Recall that $\tau^2 = C\log\big(1 + \frac{n_2}{k_2^2}\log {n_1 \choose k_1}\big)$. 
Since we have already shown that $\tau \leq c \sigma$, it is clear that $\log\big(1 + \frac{n_2}{k_2^2}\log {n_1 \choose k_1}\big) \leq C'\sigma^2$ for some $C' > 0$. Furthermore, by the assumption $\frac14 \geq p_0 \geq \frac{C_\alpha}{k_1k_2} \log\left(e{n_1 \choose k_1}{n_2 \choose k_2}\right)$, it holds
\begin{align*}
    \sigma^2 &= k_1 p_0(1 - p_0) \\
    &\geq \frac34 \frac{C_\alpha}{k_2} \log\left(e{n_1 \choose k_1}{n_2 \choose k_2}\right) \\
    &\geq \frac34 \frac{C_\alpha}{k_2} \log {n_1 \choose k_1},
\end{align*}
and hence $k_2^{-1}\log {n_1 \choose k_1} \leq C'' \sigma^2$ for some $C'' > 0$. Taking $\bar{C} = (C' + C'')/C_\delta$, we have that $\frac{k_1\bar{\delta}}{\sigma^2} \leq \bar{C}$. Therefore, there exists a constant $\bar{c} > 0$ such that $\log\big(1 + \frac{\theta}{\sigma}\big) \geq \log\big(1 + \frac{k_1\bar{\delta}}{\sigma}\big) \geq \bar{c}\frac{k_1\bar{\delta}}{\sigma}$. With this result in hand, we have
\begin{align*}
    \tilde{c}k_2\theta \sigma \log\left(1 + \frac{\theta}{\sigma}\right) &\geq  ck_2\theta k_1 \bar{\delta} \quad \text{(for $c = \tilde{c}\bar{c}$)} \\
    &\geq c k_2 \frac{(k_1 \bar{\delta})^2}{\sigma} \\
    &= c C_\delta k_2 \sigma \left(\log\Big(1 + \frac{n_2}{k_2^2}\log {n_1 \choose k_1}\Big) + \frac{1}{k_2}\log {n_1 \choose k_1}\right) \\
    &\geq C'_\delta \left(k_2 \log\Big(1 + \frac{n_2}{k_2^2}\log {n_1 \choose k_1}\Big) + \log {n_1 \choose k_1}\right),
\end{align*}
where the final inequality uses $\sigma \gtrsim 1$, and $C'_\delta > 0$ is a positive multiple of $C_\delta$. Note that, since we have enforced $c' \geq 1$ in the definition of $h_\alpha$, it holds
\begin{align*}
h_\alpha &= C^*\left(\sqrt{n_2\exp\left(-c' \log(1 + \frac{n_2}{k_2^2}\log{n_1 \choose k_1})\right)\log\left(\frac{2}{\alpha}{n_1 \choose k_1}\right)} + \log\left(\frac{2}{\alpha}{n_1 \choose k_1}\right)\right) \\
&\leq C^*\left(k_2\sqrt{\frac{\frac{n_2}{k_2^2}\log {n_1 \choose k_1}}{1 + \frac{n_2}{k_2^2}\log {n_1 \choose k_1}}} + \log\left(\frac{2}{\alpha}{n_1 \choose k_1}\right)\right) \\
&\leq C^* \left(k_2 + \log\left(\frac{2}{\alpha}{n_1 \choose k_1}\right) \right) \\
&\leq \frac{C'_\delta}{2} \left(k_2 \log\Big(1 + \frac{n_2}{k_2^2}\log {n_1 \choose k_1}\Big) + \log {n_1 \choose k_1}\right),
\end{align*}
where the final inequality holds for $C_\delta$ taken sufficiently large, and uses that $\frac{n_2}{k_2^2}\log {n_1 \choose k_1} \geq c > 0$. To summarize, we have shown that for $C_\delta$ taken large enough, it holds
\begin{align*}
\EE\left[\sum_{j = 1}^{n_2}\big(W_{K_1,j} - \nu^{k_1}_{\tau}\big)\one(t_{K_1, j} > \tau)\right] &\geq \tilde{c}k_2\theta \sigma \log\left(1 + \frac{\theta}{\sigma}\right) \\
&\geq 2h_\alpha.
\end{align*}
Letting $U_{K_1} := \sum_{j = 1}^{n_2}\big(W_{K_1,j} - \nu^{k_1}_{\tau}\big)\one(t_{K_1, j} > \tau)$, we therefore can conclude
\begin{align*}
    \bbP_{\bP}\big(t_{\chisqmax} \leq h_\alpha\big) &\leq \bbP_{\bP}\big(U_{K_1} \leq h_\alpha\big)\\
    &= \bbP_{\bP}\big(U_{K_1} - \EE[U_{K_1}]\leq h_\alpha - \EE[U_{K_1}]\big) \\
    &=\bbP_{\bP}\Big(\big(U_{K_1} - \EE[U_{K_1}]\big)^2\geq \big(h_\alpha - \EE[U_{K_1}]\big)^2\Big) \quad \text{(since $\EE[U_{K_1}] \geq 2h_\alpha$)} \\
    &\leq \bbP_{\bP}\Big(\big(U_{K_1} - \EE[U_{K_1}]\big)^2\geq \big(\EE[U_{K_1}]\big)^2\Big) \\
    &\leq \frac{\var\big(U_{K_1}\big)}{\big(\EE[U_{K_1}]\big)^2} \quad \text{(by Markov's inequality)} \\
    &\leq \frac{C_1 k_2 (\sigma^2 + \theta \sigma) \log^2\left(1 + \frac{\theta}{\sigma}\right) + C_1 (n_2 - k_2)\left(\sigma \tau \log(1 + \frac{\tau}{\sigma})\right)^2\exp\big(-\frac38 \tau^2\big)}{\left(\tilde{c}k_2\theta \sigma \log\left(1 + \frac{\theta}{\sigma}\right)\right)^2} \\
    &= \text{I} + \text{II},
\end{align*}
where we define
\[\text{I} = \frac{C_1 k_2 (\sigma^2 + \theta \sigma) \log^2\left(1 + \frac{\theta}{\sigma}\right)}{\left(\tilde{c}k_2\theta \sigma \log\left(1 + \frac{\theta}{\sigma}\right)\right)^2},\]
and 
\[\text{II} = \frac{C_1 (n_2 - k_2)\left(\sigma \tau \log(1 + \frac{\tau}{\sigma})\right)^2\exp\big(-\frac38 \tau^2\big)}{\left(\tilde{c}k_2\theta \sigma \log\left(1 + \frac{\theta}{\sigma}\right)\right)^2}.\]
We control each of these terms separately. First we have
\begin{align*}
    \text{I} &= \frac{C_1 k_2 (\sigma^2 + \theta \sigma) \log^2\left(1 + \frac{\theta}{\sigma}\right)}{\left(\tilde{c}k_2\theta \sigma \log\left(1 + \frac{\theta}{\sigma}\right)\right)^2} \\
    &= \frac{C_1(\sigma^2 + \theta \sigma)}{\tilde{c}^2k_2\theta^2 \sigma^2 } \\
    &= \frac{C_1}{\tilde{c}^2k_2\theta^2 }  + \frac{C_1}{\tilde{c}^2k_2\theta \sigma }\\
    &\leq \frac{C_1}{C'_\delta \tilde{c}^2 k_2\sqrt{\log\Big(1 + \frac{n_2}{k_2^2}\log {n_1 \choose k_1}\Big) + \log {n_1 \choose k_1}}} \\
    &+ \frac{C_1}{C'_\delta \tilde{c}^2 k_2 \left(\log\Big(1 + \frac{n_2}{k_2^2}\log {n_1 \choose k_1}\Big) + \log {n_1 \choose k_1}\right)}\\
    &\leq \frac{C_1}{C'_\delta \tilde{c}^2 k_2 \sqrt{\log(1 + c)}} + \frac{C_1}{C'_\delta \tilde{c}^2 k_2 \log(1 + c)} \\
    &\leq \frac{\alpha}{4},
\end{align*}
where the final inequality holds for $C_\delta$ taken large enough. Now we control $\text{II}$.  Recall that $C \geq \frac{8}{3}$ in the definition of $\tau$. Then it holds
\begin{align*}
    \text{II} &= \frac{C_1 (n_2 - k_2)\left(\sigma \tau \log(1 + \frac{\tau}{\sigma})\right)^2\exp\big(-\frac38 \tau^2\big)}{\left(\tilde{c}k_2\theta \sigma \log\left(1 + \frac{\theta}{\sigma}\right)\right)^2} \\
    &\leq \frac{C_1n_2}{\tilde{c}^2k_2^2\Big(1 + \frac{n_2}{k_2^2}\log {n_1 \choose k_1}\Big)} \cdot \frac{\left(\tau \log\big(1 + \frac{\tau}{\sigma}\big)\right)^2}{\left(\theta \log\big(1 + \frac{\theta}{\sigma}\big)\right)^2} \\
    &\leq \frac{C_1}{\tilde{c}^2} \cdot \frac{\left(C\theta\log(1 + c)\right)^2}{\left(\theta \log\big(1 + \frac{\theta}{\sigma}\big)\right)^2} \quad \text{(since $\tau \leq c \sigma $ and $\tau \leq C \theta$)} \\
    &\leq \frac{C_1}{\tilde{c}^2} \cdot \frac{ \left(C\log(1 + c)\right)^2}{C'_\delta\left(k_2\log\Big(1 + \frac{n_2}{k_2^2}\log {n_1 \choose k_1}\Big) + \log {n_1 \choose k_1}\right)} \\
    &\leq \frac{C_1}{\tilde{c}^2} \cdot \frac{ \left(C\log(1 + c)\right)^2}{C'_\delta\left(k_2\log\Big(1 + c\Big)\right)} \\
    &\leq \frac{\alpha}{4},
\end{align*}
where again the final inequality holds for $C_\delta$ taken sufficiently large. Therefore, we have shown
\[\bbP_{\bP}\big(t_{\chisqmax, 1} \leq h_\alpha\big) \leq \text{I} + \text{II} \leq \frac{\alpha}{2}.\]
Combining this with our bound on the Type 1 error, it holds
\[\cR(\Delta^{h_\alpha}_{\chisqmax, 1}, \delta) \leq \alpha,\]
and the proof is complete.
\end{proof}

\subsection{Proof of Theorem 2}

Here, we provide the proof of Theorem 2, which gives an upper bound on the minimax rate of separation $\delta^*$.
We let 
\begin{align*}
    \tilde R = \big(\psi_{12} + \beta_{21}\big) \land \big(\psi_{21} + \beta_{12}\big) \land \phi_{12} \land \phi_{21}.
\end{align*}

We define our optimal test as
\begin{align*}
    \Delta^* = \begin{cases}
        \Delta^{h_3}_{\chisqmax, 1}& \text{ if } \tilde R = \psi_{12} + \beta_{21} \\
        \Delta^{h_3}_{\chisqmax, 2}& \text{ if } \tilde R = \psi_{21} + \beta_{12} \\
        \Delta_a^{h_1,h_2} & \text{ if } \tilde R = \phi_{12}\\
        \Delta_b^{h_1',h_2'} & \text{ if } \tilde R = \phi_{21},
    \end{cases}
\end{align*}
where
\begin{align*}
    &\Delta_a^{h_1,h_2} = \begin{cases}
        \Delta^{h_1}_{\chisqlin, 1} & \text{if  $\frac{n_2}{k_2^2} \geq c_1$,} \\ \Delta^{h_2}_{\lin} & \text{otherwise}
    \end{cases} \qquad \text{ and } 
    \qquad \Delta_b^{h_1',h_2'} = \begin{cases}
        \Delta^{h_1'}_{\chisqlin, 2} & \text{if  $\frac{n_1}{k_1^2} \geq c_1$,} \\ \Delta^{h_2'}_{\lin} & \text{otherwise.}
    \end{cases} 
\end{align*}
and, for some small enough $\alpha>0$, we let 
\begin{align*}
    h_3 &= C^*\left(\sqrt{n_2\exp\left(-c' \log\left(1 + \frac{n_2}{k_2^2}\log{n_1 \choose k_1}\right)\right)\log\left(\frac{2}{\alpha}{n_1 \choose k_1}\right)} + \log\left(\frac{2}{\alpha}{n_1 \choose k_1}\right)\right)\\
    h_4 &= C^*\left(\sqrt{n_1\exp\left(-c' \log\left(1 + \frac{n_1}{k_1^2}\log{n_2 \choose k_2}\right)\right)\log\left(\frac{2}{\alpha}{n_2 \choose k_2}\right)} + \log\left(\frac{2}{\alpha}{n_2 \choose k_2}\right)\right)\\
    h_1 &= C^*\left(\sqrt{n_2\exp\left(-c' \log(1 + \frac{n_2}{k_2^2})\right)\log(2/\alpha)} + \log(2/\alpha)\right)\\
    h_1' &= C^*\left(\sqrt{n_1\exp\left(-c' \log(1 + \frac{n_1}{k_1^2})\right)\log(2/\alpha)} + \log(2/\alpha)\right)\\
    h_2 &= h_2' = \sqrt{4 \log(2/\alpha)}.
\end{align*}
We recall the assumption
\begin{align*}
        p_0 \geq \begin{cases}
            \frac{C_\eta}{k_1k_2} \log\left(e{n_1 \choose k_1}{n_2 \choose k_2}\right) & \text{ if } \tilde R = \big(\psi_{12} + \beta_{21}\big) \land \big(\psi_{21} + \beta_{12}\big)\\[5pt]
            \frac{C_\eta}{n_1} \log\left(1+ \frac{n_2}{k_2^2}\right) & \text{ if } \tilde R = \phi_{12} \text{ and } n_2>k_2^2\\[5pt]
            \frac{C_\eta}{n_2} \log\left(1+ \frac{n_1}{k_1^2}\right) & \text{ if } \tilde R = \phi_{21} \text{ and } n_1>k_1^2\\[5pt]
            \frac{C_\eta}{n_1n_2} & \text{ otherwise.}
        \end{cases} 
    \end{align*}

By Lemma~\ref{lem_simplified_rate_for_UB}, we always have $R \gtrsim \tilde R$. 
Assume first that  $\tilde R = \big(\psi_{12} + \beta_{21}\big) \land \big(\psi_{21} + \beta_{12}\big)$, and, by symmetry, assume that we have  $\tilde R = \psi_{12} + \beta_{21}$. 
In this case, we have $\Delta^* = \Delta^{h_3}_{\chisqmax, 1}$, and the result follows by Lemma~\ref{lem_max_truncated_degree_test}. We can proceed similarly in the case where $\tilde R = \psi_{21} + \beta_{12}$. 

Assume now that $\tilde R = \phi_{12}$ and $n_2 > k_2^2$. Then the result follows by Lemma~\ref{lem_truncated_degree_test}. Similarly, if  
$\tilde R = \phi_{21}$ and $n_1 > k_1^2$, the result follows by Lemma~\ref{lem_truncated_degree_test}.
Finally, if none of the conditions above are satisfied, then we have $\Delta^* = \Delta^{h_2}_{\lin}$ and the result follows by Lemma~\ref{lem_linear_test}. 
The proof is complete.

\section{Additional results}

\subsection{Technical results for the analysis of the truncated chi-square tests}

\begin{lemma}\label{lemma17_LiuGaoSamworth}
Let $n \in \mathbb N$ and $p \in (0,1/4)$. 
Let $X \sim \operatorname{Bin}(n,p)$ and, for any $a>0$,  define 
\begin{align*}
    &Z \,=\, (X-np)/\sigma, \qquad \text{ where } \sigma = \sqrt{np(1-p)}\\
    &v(x)  = \sigma x \log\left(1+ \frac{x}{\sigma}\right), \quad \forall x > 0.
\end{align*}
For any $\alpha\geq 1$, there exist two constants $C_\alpha, \bar C>0$ such that, for any $a \in\big[C_\alpha, \sigma\big]$, we have $\mathbb E\left[v(Z)^{\alpha} \big|Z\geq a\right] \leq \bar C v(a)^{\alpha}$. 
\end{lemma}
\begin{proof}[Proof of Lemma~\ref{lemma17_LiuGaoSamworth}]

We first note that the event $\{Z \geq a\}$ is equivalent to $\{X \geq k_{\min}\}$ where $k_{\min} = \lceil np + a\sigma\rceil$. Let $c^*>0$ be a constant whose value will be adjusted later. Let also $\alpha > 0$. We have
\begin{align*}
\mathbb E\left[v(Z)^\alpha \one{\big( Z \geq a\big)}\right] &= \mathbb E\left[v(Z)^\alpha \one{\big( X \geq k_{\min}\big)}\right]\\[5pt]
& = \mathbb E\left[v(Z)^\alpha \one{\big( k_{\min} \leq X \leq k_{\min} + c^* a \sigma\big)}\right] \\[5pt]
& \quad + \mathbb E\left[v(Z)^\alpha \one{\big( k_{\min} + c^*a\sigma < X \leq n-2\big)}\right] \\[5pt]
& \quad + \mathbb E\left[v(Z)^\alpha \one{\big(X  \geq n-1\big)}\right]\\[5pt]
& =: \text{I} + \text{II} + \text{III}.\numberthis \label{eq_split_moment_alpha}
\end{align*}
We control each term separately. \\

\textit{Term III.} We use the inequality $\log(ab) \leq 2a \log(b)$ when $a >1$ and $b \geq 2$. Indeed, 
\begin{align*}
\log(b ) \geq \frac{1}{2}\log\left(b+1 \right) \geq \frac{1}{2} \log\left(\frac{ab}{a} + (1-\frac{1}{a})\cdot 1\right) \geq \frac{1}{2a} \log\left(ab\right)
\end{align*}
by concavity of the logarithm.  
Now, noting that $n(1-p) \geq p$, we have
\begin{align*}
\text{III} ~& =\mathbb E\left[v\!\left(\frac{X-np}{\sigma}\right)^\alpha \one{\big(X  \geq n-1 \big)}\right] \\
&= v\!\left(\frac{n(1-p)-1}{\sigma}\right)^\alpha n p^{n-1}(1-p) + v \!\left(\frac{n(1-p)}{\sigma}\right)^\alpha p^n\\
& \leq  2v\!\left(\frac{n(1-p)}{\sigma}\right)^\alpha n p^{n-1}(1-p) \\
& = 2 \left(\frac{\sigma}{p}\right)^{\alpha+2}\log^\alpha\left(1+\frac{1}{p}\right) p^{n}\\
& = 2 \left(\sqrt{\frac{n(1-p)}{p}}\right)^{\alpha+2}  p^{n-\alpha}\\
& =
\leq 2 \exp\left(\frac{\alpha + 2}{2}\left(\log(n(1-p)) + \log\Big(\frac{1}{p}\Big)\right) - (n-\alpha)\log\Big(\frac{1}{p}\Big)\right).
\end{align*}
Now, recall that we have $\sqrt{n/4} \geq \sigma \geq a \geq C_\alpha$. 
Therefore, we have $\frac{\alpha + 2}{2}\log(n(1-p))\leq \frac{n-\alpha}{4} \log\left(1/p\right)$ and $\frac{\alpha + 2}{2}\log(1/p) \leq \frac{n-\alpha}{4} \log\left(1/p\right)$ provided $C_\alpha$ is sufficiently large. Therefore,
\begin{align*}
    \text{III} \leq 2 \exp\left(-n/2\right) \leq 2 e^{-\frac{a^2}{2}}\numberthis \leq 4a \mathbb P\left(Z \geq a\right).\label{eq_term_III}
\end{align*}
In the last line, we used Theorem 2.1 from~\cite{slud1977distribution}, which states that whenever $p \leq \frac{1}{4}$, it holds that $\mathbb P(Z \geq a) \geq 1 - \Phi(a) \geq \frac{1}{2a} e^{-\frac{a^2}{2}}$ for $a \geq 2$ where $\Phi(a)$ denotes the cumulative distribution function of the standard normal distribution.

\textit{Term I.} Note that, if $k \in [k_{\min}, k_{\min} + c^*a\sigma]$, then 
\begin{align*}
v\left(\frac{k-np}{\sigma}\right)^\alpha &\leq v\left(\frac{k_{\min} + c^* a\sigma-np}{\sigma}\right)^\alpha \\
&= v\left(\frac{\lceil np + a\sigma\rceil + c^* a\sigma-np}{\sigma}\right)^\alpha\\
& \leq v\left(\frac{(1+c^*) a\sigma+1}{\sigma}\right)^\alpha\\
& \leq v\Big((1+\frac{3}{2}c^*) a\Big)^\alpha\\
& \leq (1+2c^*)^\alpha \, v(a)^\alpha
\end{align*}
provided the constant $C_\alpha>0$ is chosen to be large enough. 
Therefore, we obtain
\begin{align*}
\text{I} &= \mathbb E\left[v\!\left(\frac{X-np}{\sigma}\right)^\alpha \one{\big( k_{\min} \leq X \leq k_{\min} + c^* a \sigma\big)}\right]\\
& \leq (1+2c^*)^\alpha v(a)^\alpha \mathbb P \left(Z \geq a\right).\numberthis \label{eq_term_I}
\end{align*}

\textit{Term II.}
We let $j_0 = \lfloor c^* a \sigma\rfloor + 1$ and $k' = k_{\min} + j_0$, so that the event $X > k_{\min} + c^*a\sigma$ is equivalent to $X \geq k'$.
Therefore, we obtain
$$
\begin{aligned}
\text{II} &= \sum_{k = k'}^{n-2} v\left(\frac{k-np}{\sigma}\right)^\alpha \frac{n!}{k!(n-k)!} p^k (1-p)^{n-k}\\
&\leq \sum_{k = k'}^{n-2} v\left(\frac{k-np}{\sigma}\right)^\alpha \frac{C_0}{c_0^2\sqrt{2\pi}} \sqrt{\frac{n}{(n-k)k}} \left(\frac{n(1-p)}{n-k}\right)^{n-k} \left(\frac{np}{k}\right)^k\qquad \text{ by Lemma~\ref{lem_Stirling}}\\
& \leq \frac{C_0}{c_0^2\sqrt{2\pi}} \sum_{k = k'}^{n-2} v\left(\frac{k-np}{\sigma}\right)^\alpha \sqrt{\frac{n}{(n-k)k}} \left(1 - \frac{k-np}{n(1-p)}\right)^{-n+k} \left(1+ \frac{k-np}{np}\right)^{-k}\\
&\leq \frac{C_0}{c_0^2\sqrt{2\pi}} \sum_{j=j_0}^{n-k_{\min}-2} v\left(\frac{a\sigma + j + 1}{\sigma}\right)^\alpha \sqrt{\frac{n}{(n-k_{\min}-j)(k_{\min} + j)}} \\
& \qquad \left(1 - \frac{k_{\min} + j-np}{n(1-p)}\right)^{-n+k_{\min} + j} \left(1+ \frac{k_{\min} + j-np}{np}\right)^{-k_{\min} - j} \\
& =: \tilde C \sum_{j=j_0}^{n-k_{\min}-2} v\left(\frac{a\sigma + j}{\sigma}\right)^\alpha \sqrt{\frac{n}{(n-k_{\min}-j)(k_{\min} + j)}} \left(1 - \frac{a \sigma + j + 1}{n(1-p)}\right)^{-n+k_{\min} + j} \left(1+ \frac{a\sigma + j}{np}\right)^{-k_{\min} - j},
\end{aligned}
$$
where $\tilde C = 2\frac{C_0}{c_0^2\sqrt{2\pi}} $ and where we used the fact that $v\left(\frac{a\sigma + j + 1}{\sigma}\right)^\alpha  \leq 2 v\left(\frac{a\sigma + j}{\sigma}\right)^\alpha $
provided that $\sigma \geq a \geq C_\alpha$ for some large enough constant $C_\alpha$. 
Moreover, we note that
\begin{align*}
& \quad \left(1 - \frac{a \sigma + j + 1}{n(1-p)}\right)^{-n+k_{\min} + j} \left(1+ \frac{a\sigma + j}{np}\right)^{-k_{\min} - j} \\
& = \exp\left(-(n-k_{\min} - j) \log\left(1 - \frac{a \sigma + j + 1}{n(1-p)}\right) - (k_{\min} + j) \log\left(1 + \frac{a\sigma + j}{np}\right)\right)\\
& \leq \exp\left(-(n-np - a\sigma - j) \log\left(1 - \frac{a \sigma + j + 1}{n(1-p)}\right) - (np + a\sigma + j) \log\left(1 + \frac{a\sigma + j}{np}\right)\right)\\
& = \exp\left(-n(1-p)\left(1 - \frac{a\sigma + j}{n(1-p)}\right) \log\left(1 - \frac{a \sigma + j + 1}{n(1-p)}\right) - np\left(1 + \frac{a\sigma + j}{np}\right) \log\left(1 + \frac{a\sigma + j}{np}\right)\right).
\end{align*}
We now argue that the latter quantity is at most a constant times the quantity below:
\begin{align*}
\exp\left(-n(1-p)\left(1 - \frac{a\sigma + j}{n(1-p)}\right) \log\left(1 - \frac{a \sigma + j}{n(1-p)}\right) - np\left(1 + \frac{a\sigma + j}{np}\right) \log\left(1 + \frac{a\sigma + j}{np}\right)\right).
\end{align*}

Indeed, the ratio of the two is 
\begin{align*}
\exp\left(-n(1-p)\left(1 - \frac{a\sigma + j}{n(1-p)}\right) \log\left(1- \frac{\frac{1}{n(1-p)} }{1 - \frac{a \sigma + j}{n(1-p)}}\right)\right),
\end{align*}
and since $j \leq n - k_{\min} - 2 \leq n(1-p) - a\sigma - 2$, 
it follows that
\begin{align*}
\frac{\frac{1}{n(1-p)} }{1 - \frac{a \sigma + j}{n(1-p)}} \leq \frac{\frac{1}{n(1-p)} }{\frac{2}{n(1-p)}} = \frac{1}{2}.
\end{align*}
Using the inequality $\log(1-x) \geq -2x$ for $x \in (0,1/2)$, the above ratio is controlled as
\begin{align*}
&\exp\left(-n(1-p)\left(1 - \frac{a\sigma + j}{n(1-p)}\right) \log\left(1- \frac{\frac{1}{n(1-p)} }{1 - \frac{a \sigma + j}{n(1-p)}}\right)\right) \leq \exp(2).
\end{align*}
Combining these calculations, we obtain
\begin{align*}
\text{II} &\leq \tilde C \sum_{j=j_0}^{n-k_{\min} - 2} v\left(\frac{a\sigma + j}{\sigma}\right)^\alpha \sqrt{\frac{n}{(n-k_{\min}-j)(k_{\min} + j)}} \\
& \hspace{2cm}\exp\left(-n(1-p)\left(1 - \frac{a\sigma + j}{n(1-p)}\right) \log\left(1 - \frac{a \sigma + j}{n(1-p)}\right) - np\left(1 + \frac{a\sigma + j}{np}\right) \log\left(1 + \frac{a\sigma + j}{np}\right)\right)\\
& = \tilde C\sum_{j=j_0}^{n-k_{\min} - 2} v\left(\frac{a\sigma + j}{\sigma}\right)^\alpha \sqrt{\frac{n}{(n-k_{\min}-j)(k_{\min} + j)}} \exp\left(-n(1-p) h\left(-\frac{a\sigma + j}{n(1-p)}\right) - np \, h\left(\frac{a\sigma + j}{np}\right)\right)
\end{align*}
where we have defined the Bennett function
\begin{align*}
h(x) = (1+x) \log\left(1+x\right) - x, \qquad \forall x > -1.
\end{align*}
Now, we define the sets of indices $J_1 = \{j_0, \dots, \lfloor 2 \sigma^2\rfloor\}$ and $J_2 = \{\lfloor 2 \sigma^2 \rfloor + 1, \dots, n - k_{\min} - 2\rfloor\}$.
We have 
\begin{align*}
    \text{II} & \leq \tilde C\sum_{j\in J_1} v\left(\frac{a\sigma + j}{\sigma}\right)^\alpha \sqrt{\frac{n}{(n-k_{\min}-j)(k_{\min} + j)}} \exp\left(-n(1-p) h\left(-\frac{a\sigma + j}{n(1-p)}\right) - np \, h\left(\frac{a\sigma + j}{np}\right)\right)\\
    & \quad + \tilde C\sum_{j\in J_2} v\left(\frac{a\sigma + j}{\sigma}\right)^\alpha \sqrt{\frac{n}{(n-k_{\min}-j)(k_{\min} + j)}} \exp\left(-n(1-p) h\left(-\frac{a\sigma + j}{n(1-p)}\right) - np \, h\left(\frac{a\sigma + j}{np}\right)\right)\\
    & =: \text{I\hspace{.2mm}I}_{J_1} + \text{I\hspace{.2mm}I}_{J_2}\numberthis \label{eq_split_term_II}
\end{align*}
We now control these two terms separately. 
Suppose first that $j \in J_1$. Then we have 
\begin{align*}
\frac{a\sigma+j}{np} \land \frac{a\sigma+j}{n(1-p)} \leq \frac{a\sigma + 2\sigma^2}{\sigma^2} \leq 3,
\end{align*}
and, since $j \leq \sigma^2 \leq n/4$, we also have
\begin{align*}
    \sqrt{\frac{n}{(n-k_{\min}-j)(k_{\min} + j)}} &\leq \sqrt{\frac{n}{(n-np - a\sigma - 1-n/4)(k_{\min} + j)}}\\
    & \leq \sqrt{\frac{n}{(n/4-1)(k_{\min} + j)}}\\
    & \leq 3 \sqrt{\frac{1}{k_{\min}+j}}.
\end{align*}
In the last line, we used the fact that, since $\sqrt{n/4} \geq \sigma \geq a \geq C_\alpha$, we have $n\geq 4C_\alpha^2$, so that $\frac{n}{n/4-1} \leq 9$ for $C_\alpha$ large enough.
Moreover, we have $h(x) \geq x^2/5$ for any $x \in (-1,3]$. Therefore,
\begin{align*}
\frac{1}{3\tilde C}\text{I\hspace{.2mm}I}_{J_1} &\leq\sum_{j\in J_1}v\left(\frac{a\sigma + j}{\sigma}\right)^\alpha \sqrt{\frac{1}{k_{\min} + j}} \exp\left(-n(1-p) h\left(-\frac{a\sigma + j}{n(1-p)}\right) - np \, h\left(\frac{a\sigma + j}{np}\right)\right)\\
& \leq \sum_{j\in J_1}v\left(\frac{a\sigma + j}{\sigma}\right)^\alpha \sqrt{\frac{1}{k_{\min}}} \exp\left(-\frac15\frac{(a\sigma + j)^2}{n(1-p)} - \frac15\frac{(a\sigma + j)^2}{np}\right)\\
& = \sum_{j\in J_1} v\left(\frac{a\sigma + j}{\sigma}\right)^\alpha \sqrt{\frac{1}{k_{\min}}} \exp\left(-\frac15\frac{(a\sigma + j)^2}{\sigma^2}\right).
\end{align*}
Note that the function $x \mapsto x^\alpha e^{-\frac{x^2}{5}}$ is decreasing over $[\sqrt{5\alpha/2}, \infty)$. Therefore, choosing $C_\alpha$ large enough again, and letting $\tilde C$ denote a constant whose value may change in each appearance, we obtain 
\begin{align*}
\frac{1}{3 \tilde C}\text{I\hspace{.2mm}I}_{J_1}&  \sum_{j\in J_1} v\left(\frac{a\sigma + j}{\sigma}\right)^\alpha \sqrt{\frac{1}{k_{\min}}} \exp\left(-\frac15\frac{(a\sigma + j)^2}{\sigma^2}\right) \\
&\leq \frac{1}{\sqrt{k_{\min}}} \int_{j_0}^\infty v\left(\frac{a\sigma + x}{\sigma}\right)^\alpha e^{-\frac{(a\sigma + x)^2}{5 \sigma^2}} dx  \qquad \text{ where } j_0 - 1 = \lfloor c^* a \sigma \rfloor\\
& \leq \frac{\sigma}{\sqrt{k_{\min}}} \int_{\sqrt{\frac{2}{5}}\left(a + \frac{1}{\sigma}\lfloor c^*a \sigma\rfloor\right)}^\infty \left(\frac{5}{2}\right)^{1/2}v \left(\sqrt{\frac{5}{2}}y\right)^\alpha e^{-\frac{y^2}{2}} dy\\
& \leq \tilde C\frac{\sigma}{\sqrt{k_{\min}}} \int_{a}^\infty v \left(y\right)^\alpha e^{-\frac{y^2}{2}} dy \qquad  \text{ by taking $c^*$ large enough}\\
& \leq \tilde C \frac{\sigma}{\sqrt{k_{\min}}} a^{\alpha-1} \log(1+a)^\alpha \exp\left(-\frac{a^2}{2}\right) \qquad \text{ by Lemma~\ref{lem:gaussian_tail_momentk_ellt}}
\end{align*}
so that
\begin{align}
    \text{I\hspace{.2mm}I}_{J_1} \leq C_{J_1}v(a)^\alpha \mathbb P\left(Z \geq a\right),\label{eq_term_II_J1}
\end{align}
for some constant $C_{J_1}>0$, where we used the fact that $\sigma \leq \sqrt{k_{\min}}$ and Theorem 2.1 from~\cite{slud1977distribution}.

Now, we turn to the term $\text{I\hspace{.2mm}I}_{J_2}$. 
Note that, when $j \in J_2$, 
we have
\begin{align*}
    \frac{a\sigma + j}{np} \geq \frac{a\sigma + 2np(1-p)}{np} \geq 2(1-p) \geq \frac{3}{2}.
\end{align*}
Moreover, since $j \geq \sigma^2$, we also have
\begin{align*}
    \sqrt{\frac{n}{n-k_{\min} - j}} e^{-j/4} &= \left(1 - \frac{k_{\min} + j}{ n}\right)^{-1/2} e^{-j/4} = \left(1-x\right)^{-1/2} e^{(-nx + k_{\min})/4}
\end{align*}
where $x = (k_{\min}+j)/n \in [ 2p, 1 - \frac{2}{n}]$. 
We now prove that the function $g: x \mapsto \left(1-x\right)^{-1/2} e^{-nx + k_{\min}}$ is bounded over the interval $[2p,1-\frac{2}{n}]$ independently of $n$. 
Indeed, the derivative of $\log g$ satisfies
\begin{align*}
    (\log g)'(x) = \frac{1}{2(1-x)} - n/4,
\end{align*}
so that $g$ is decreasing over $[2 p, 1-\frac{2}{n}]$. 
It now suffices to evaluate $g(2p)$
\begin{align*}
    g(2p) &= \left(1 - 2p \right)^{-1/2} \exp\left(-2n p + k_{\min}\right) \leq \left(1/2 \right)^{-1/2} \exp\left(-2n p + np + a\sigma + 1\right)\\
    & \leq \left(1/2 \right)^{-1/2} \exp\left(-\sigma^2/2 \right) \leq \left(1/2 \right)^{-1/2} \exp\left(-C_\alpha^2/2 \right).
\end{align*}
Therefore, there exists a constant $\bar c$ such that, for any $n$ and any $j \in J_2$, we have
\begin{align*}
    \sqrt{\frac{n}{n-k_{\min} - j}} e^{-j} \leq \bar c.
\end{align*}
By the inequality $h(x) 
 \geq \frac{x}{2}$ 
that holds for any $x \geq 3/2$, we have
\begin{align*}
    \frac{1}{\tilde C}\text{I\hspace{.2mm}I}_{J_2} &= \sum_{j\in J_2}v\left(a +\frac{j}{\sigma}\right)^\alpha \sqrt{\frac{n}{(k_{\min} + j)(n-k_{\min} - j)}} \exp\left(-n(1-p) h\left(-\frac{a\sigma + j}{n(1-p)}\right) - np \, h\left(\frac{a\sigma + j}{np}\right)\right)\\
    & \leq \sum_{j\in J_2}v\left(a+\frac{j}{\sigma}\right)^\alpha \frac{1}{ \sqrt{j}} \sqrt{\frac{n}{n-k_{\min} - j}}\exp\left(- np \, h\left(\frac{a\sigma + j}{np}\right)\right)\\
    & \leq  \sum_{j\in J_2} v\left(\frac{2j}{\sigma}\right)^\alpha \frac{1}{\sqrt{j}} \sqrt{\frac{n}{n-k_{\min} - j}} \exp\left(-  \frac{a\sigma + j}{2} \right) \\
    & \leq \tilde C \bar c\sum_{j\in J_2}v\left(\frac{j}{\sigma}\right)^\alpha \frac{1}{\sqrt{j}} \exp\left(- \frac{j}{4}\right)\\
    & \leq \frac{\tilde C}{\sigma^\alpha}\int_{\lfloor 2 \sigma^2\rfloor}^\infty x^{\alpha - 1/2} \log(1+x)^\alpha \exp\left(-x/4\right) dx\\
    & = \frac{ \tilde C}{\sigma^\alpha} \int_{\lfloor 2 \sigma^2\rfloor/4} \left(4y\right)^{\alpha-1/2}  \log(1+4y)^\alpha\exp\left(-y\right) 4 dy\\
    & = \frac{ \tilde C}{\sigma^\alpha} \int_{\lfloor 2 \sigma^2\rfloor/4} y^{\alpha-1/2}  \log(1+y)^\alpha\exp\left(-y\right) 4 dy\\
    & \leq \frac{\tilde C}{\sigma^\alpha} \left(\frac{\lfloor 2 \sigma^2\rfloor}{4}\right)^{\alpha - 1/2} \log\left(1+\frac{\lfloor2\sigma^2\rfloor}{4}\right)^\alpha \exp\left( - \frac{\lfloor 2 \sigma^2\rfloor}{4}\right) \qquad \text{ by Lemma~\ref{lem_moment_exponential}}\\
    &\leq \frac{\tilde C}{\sigma^\alpha} \left(\frac{ 2 \sigma^2}{4}\right)^{\alpha - 1/2} \log\left(1+\frac{2\sigma^2}{4}\right)^\alpha \exp\left( - \frac{2 \sigma^2}{4} + \frac{1}{2}\right)\\
    & \leq C_{J_2}\frac{v(\sigma)^\alpha}{\sigma} \exp\left(-\frac{\sigma^2}{2}\right) \qquad \text{ for some constant } C_{J_2}\\
    & \leq C_{J_2} \frac{v(a)^\alpha}{a} \exp\left(-\frac{a^2}{2}\right) \\
    & \leq C_{J_2} v(a)^\alpha \mathbb P\left(Z \geq a\right)\numberthis\label{eq_term_II_J2}
\end{align*}
where the last inequality follows from Theorem 2.1 from~\cite{slud1977distribution}.

To conclude the proof of the lemma, it remains to combine equations~\eqref{eq_split_moment_alpha}, \eqref{eq_term_III}, \eqref{eq_term_I}, \eqref{eq_split_term_II}, \eqref{eq_term_II_J1} and \eqref{eq_term_II_J2} to obtain
\begin{align*}
    \mathbb E\left[Z^\alpha | Z \geq a\right] &= \frac{\mathbb E\left[Z^\alpha \,\one{\big( Z \geq a\big)}\right]}{\mathbb P(Z \geq a)}\\
    & = \frac{\text{I}+\text{II}+\text{III}}{\mathbb P(Z\geq A)}\\
    & \leq \frac{\text{I}+\text{II}_{J_1} + \text{II}_{J_2} +\text{III}}{\mathbb P(Z\geq A)}\\
    & \leq \frac{(1+2c^*)^\alpha v(a)^\alpha \mathbb P \left(Z \geq v(a)\right)+ C_{J_1}v(a)^\alpha \mathbb P\left(Z \geq a\right) + C_{J_2}v(a)^\alpha \mathbb P\left(Z \geq a\right)+4 v(a) \mathbb P(Z \geq a)}{\mathbb P(Z \geq a)}\\
    & \leq \bar C v(a)^\alpha
\end{align*}
for some constant $\bar C$ and
where we used $v(a) \leq a^\alpha$ since $a,\alpha\geq 1$. 
This completes the proof.

\end{proof}

We recall that the Bennett function is defined as
\begin{align}
    &h_B(x) = (1+x) \log(1+x) - x,\qquad \forall x>-1.\label{eq_def_h_bennett}\\
    &h_B(-1) = 1.
\end{align}
Now, we let $Y \sim \operatorname{Bin}(n,p_0)$ for some $n \in \mathbb N$ and define 
\begin{align*}
    w(x) &= n(1-p_0) \, h_B\!\left(-\frac{x-np_0}{n(1-p_0)}\right) + np_0 \, h_B\!\left(\frac{x-np_0}{np_0}\right), \forall x \in [0, n]\numberthis \label{eq_def_W_appendix}\\
    Z &= (Y - np_0)/\sigma, \qquad \text{ where } \sigma = \sqrt{np_0(1-p_0)}\\[5pt]
    W &= w(Y)\\
    \nu_a &=  \mathbb 
    E\left[W| Z \geq a\right]\numberthis \label{eq_def_nu_a}\\
\gamma_a &= \mathbb E\left[W^2 | Z \geq a\right].\numberthis\label{eq_def_gamma}
\end{align*}
We deduce the following corollary.

\begin{corollary}\label{cor_lemma17_LiuGaoSamworth}
    Recall the definitions of the functions $\nu$ and $\gamma$ from~\eqref{eq_def_nu_a} and~\eqref{eq_def_gamma}. There exist sufficiently large constants $C, \bar C$ such that, for any $a \in [C, \sigma]$, it holds that $\nu_a \leq \bar C \sigma a \log\left(1+\frac{a}{\sigma}\right) \leq \bar C a^2$ and $\gamma_a \leq \bar C \sigma a^2 \log^2\left(1+\frac{a}{\sigma}\right) \leq \bar C a^4$.
\end{corollary}

\begin{lemma}\label{leq_truncchisq_concentration} Let $m,n \in \mathbb N$ and $p_0 \in (0,1/4)$. 
Let $Y_1,\dots, Y_m \overset{iid}{\sim} \operatorname{Bin}(n,p_0)$ and for any $j \in \{1,\dots,m\}$ define 
$Z_j = (Y_j-np_0)/\sigma, \text{ where } \sigma = \sqrt{np_0(1-p_0)}.$ 
For any $j \in \{1,\dots,m\}$, define $W_j = w(Y_j)$ where $w$ is defined in~\eqref{eq_def_W_appendix} and recall the definition of $\nu_a$ from~\eqref{eq_def_nu_a} for $a > 0$. 
Then there exist universal constants $C,C^{*}, c, c'>0$ such that for any $a\in [C,c\sigma]$ and $x>0$, we have
$$
\mathbb{P}\left(\sum_{j=1}^{m}\left(W_j-\nu_{a}\right) \one_{\left\{Z_{j} \geq a\right\}} \geq C^{*}\left(\sqrt{m e^{-c'a^2} x}+x\right)\right) \leq e^{-x}.
$$
\end{lemma}
\begin{proof}[Proof of Lemma~\ref{leq_truncchisq_concentration}]
In this proof, we will write $p$ rather than $p_0$ to alleviate the notation.  
Let $Y \sim \operatorname{Bin}(n,p)$ and $Z = (Y-np)/\sigma$, and consider $X=\left(W-\nu_{a}\right) \one_{\{Z \geq a\}}$ where
\begin{align*}
    W = n(1-p)\,h_B\!\left(-\frac{Y-np}{n(1-p)}\right) + np \,h_B\!\left(\frac{Y-np}{np}\right).
\end{align*}
We first derive a bound for the moment-generating function $\mathbb{E}\left(e^{\lambda X}\right)$. Since $\mathbb{E}(X)=0$, we have
$$
\mathbb{E}\left(e^{\lambda X}\right)=1+\mathbb{E}\left(e^{\lambda X}-1-\lambda X\right).
$$

By the deterministic bound
$$
e^{x}-1-x \leq \begin{cases}(-x) \wedge x^{2}, & \text { if } x<0 \\ x^{2}, & \text { if } 0 \leq x \leq 1 \\ e^{x}, & \text { if } x \geq 1\end{cases}
$$
we have
$$
\mathbb{E}\left(e^{\lambda X}-1-\lambda X\right) \leq \lambda^{2} \mathbb{E}\left(X^{2} \one_{\{X<0\}}\right)+\lambda^{2} \mathbb{E}\left(X^{2} \one{\{0 \leq X \leq 1 / \lambda\}}\right)+\mathbb{E}\left(e^{\lambda X} \one_{\{X\geq1 / \lambda\}}\right)
$$
and we will bound the three terms separately. By Corollary~\ref{cor_lemma17_LiuGaoSamworth}, the first term is bounded as
\begin{align*}
    \mathbb{E}\left(\big(|X|\land X^{2}\big) \one_{\{X<0\}}\right) & \leq \mathbb{P}(X<0)\Big[\big|\nu_{a}-w(np+a\sigma)\big|\land \big(\nu_{a}-w(np+a\sigma)\big)^{2}  \Big]\\
    & \leq \mathbb{P}(X<0)\Big[\big|\nu_{a}-\frac{a^2}{8}\big|\land \big(\nu_{a}-\frac{a^2}{8}\big)^{2}  \Big] \qquad \text{ by Lemma~\ref{lem_w_quadratic}.2}\\[5pt]
    & \leq \mathbb{P}(Z \geq a)\left(\nu_{a}-\frac{a^{2}}{8}\right)^{2} \\[5pt]
    &\leq (\bar C-1/8)^2 a^4 \mathbb P\left(Y-np \geq a \sigma\right) \qquad \text{ by Corollary~\ref{cor_lemma17_LiuGaoSamworth}}\\
    & \leq (\bar C-1/8)^2 a^4\exp \left(-\frac{a^2/2}{1 + \frac{a/\sigma}{3}}\right)\qquad \text{ by Lemma~\ref{lem_rescalen_Bernstein}}\\
    & = (\bar C-1/8)^2 a^4\exp \left(-\frac{3a^2}{8}\right)\\
    & \leq \exp\left(- \frac{a^2}{4}\right)
\end{align*}
provided the universal constant $C$ such that $a \geq C$ is chosen large enough depending on $\bar C$. We now bound the second term. Recalling the definition of the function $\gamma$ from~\eqref{eq_def_gamma}, we have
$$
\begin{aligned}
    \mathbb{E}\left(X^{2} \one_{\{0<X \leq 1 / \lambda\}}\right)&\leq \mathbb{E}\left(X^{2} \one_{\{X>0\}}\right) = \mathbb E\left[\left(W-\nu_a\right)^2 \one_{\{Z >\sqrt{\nu_a}\}}\right]\\[5pt]
    & \leq 2 \mathbb E\left[W^2 \one_{\{Z >\sqrt{\nu_a}\}}\right] + 2 \nu_a^2 \mathbb P\left(Z>\sqrt{\nu_a}\right)\\[10pt]
    & = 2 \left(\gamma_{\sqrt{\nu_a}} + \nu_a^2\right) \mathbb P\left(Z>\sqrt{\nu_a}\right) \\
    & \leq 2\left(\bar C \nu_a^2 + \nu_a^2\right) \exp\left(- \frac{\nu_a/2}{1+\frac{\sqrt{\nu_a}}{3\sigma}}\right) \qquad \text{ by Lemmas~\ref{lemma17_LiuGaoSamworth} and~\ref{lem_rescalen_Bernstein} }\\
    & \leq 2 \left(\bar C^3 + \bar C^2\right)a^4 \exp\left(-\frac{\nu_a/2}{1+ \sqrt{\bar C}/3}\right) \quad \text{ by Lemma~\ref{lemma17_LiuGaoSamworth} and } a \leq \sigma\\
    & \leq \exp\left(-\frac{\nu_a}{2\bar C}\right)\\
    & \leq \exp\left(-\frac{a^2}{2\bar C}\right)
\end{aligned}
$$
provided the universal constant $C$ such that $a \geq C$ is chosen large enough depending on $\bar C$.

Finally we control the third term. 
We note that the event $X\geq\frac{1}{\lambda}$ is equivalent to $W \geq \nu_a + \frac{1}{\lambda}$. 
We let $k_0 \in \mathbb N$ be the unique integer larger than or equal to $np$ such that the event $\big\{W \geq \nu_a + \frac{1}{\lambda}$ and $Z \geq 0 \big\}$ is equivalent to $\{Y \geq k_0\}$. 
This integer exists by Lemma~\ref{lem_W_cvx_increasing} since $w$ is increasing over $[np,n]$ by Lemma~\ref{lem_W_cvx_increasing}, $w(np + a\sigma) = 0$, $\lim_{x \to n} w(x) = 2 n \log(1/p) > \nu_a + \frac{1}{\lambda}$ for $\lambda \leq \frac{1}{2}$ and $\nu_a \leq \bar C a^2 \leq \bar C c^2 \sigma^2 \leq \bar C c n < 2 n \log(1/p)$ for $c$ chosen sufficiently small depending on $\bar C$. 
Moreover, we also have $k_0 \geq np + a\sigma$ since 
$w(np+a\sigma) \leq a^2 < \nu_a + \frac{1}{\lambda}$.
Now, we have
\begin{align*}
\mathbb E\left[e^{\lambda X} \one{\big(X\geq1/\lambda\big)}\right] &= \mathbb E\left[\exp\left(\lambda(W - \nu_a)\right)\one{\big(W \geq \nu_a + \frac{1}{\lambda}\big)}\right]\\
& = e^{-\lambda \nu_a}  \sum_{k = k_0}^n e^{\lambda w(k)}\mathbb P(Y = k)\\
& \leq e^{-\lambda \nu_a}   e^{\lambda w(n)}\mathbb P(Y = n) + e^{-\lambda \nu_a}   e^{\lambda w(k_0)}\mathbb P(Y = k_0) \\
& \qquad + \frac{C_0}{\sqrt{2\pi}c_0^2} e^{-\lambda \nu_a} \sum_{k = k_0}^{n-1} \sqrt{\frac{n}{(n-k)k}} \\
& \qquad \times \exp\left((2\lambda - 1) \left\{n(1-p) h_B\left(-\frac{k-np}{n(1-p)}\right) + np \, h_B\left(\frac{k-np}{np}\right)\right\}\right) \\
&\hspace{3cm} \text{ (by Lemma~\ref{lem_Stirling})}\\
& = e^{-\lambda \nu_a}   e^{\lambda w(n)}\mathbb P(Y = n) +  e^{-\lambda \nu_a}   e^{\lambda w(k_0)}\mathbb P(Y = k_0) \\
& \qquad + \frac{C_0}{\sqrt{2\pi}c_0^2} e^{-\lambda \nu_a} \sum_{I_1 \cup I_2} \sqrt{\frac{n}{(n-k)k}} e^{\left(\lambda - \frac{1}{2}\right)w(k)}\numberthis \label{eq_thirn_term_moment_generating_function}
\end{align*}
where we have defined $I_1 = \{k_0+1, \dots, \lfloor C^* \sigma^2\rfloor \}$ and $I_2 = \{\lfloor C^* \sigma^2\rfloor + 1, \dots, n-1\}$ for some constant $C^*$ that will be adjusted later. 
We first compute the sum over the set of indices $I_1$. 
Below, we use the notation $C'$ to denote a constant whose value may change in each appearance. 
Using the change of variables 
\begin{align*}
    & y = w(x)\\
    & dx = \frac{dy}{w' \circ w^{-1}(y)} = \left(\log\left(\frac{w^{-1}(y)}{np} \frac{n(1-p)}{n-w^{-1}(y)}\right)\right)^{-1} dy
\end{align*}
we have  
\begin{align*}
    & \quad \sum_{I_1} \sqrt{\frac{n}{(n-k)k}} e^{\left(\lambda - \frac{1}{2}\right)w(k)}\\
    & \leq C'\int_{k_0}^{\lfloor C^*\sigma\rfloor +1} \sqrt{\frac{n}{(n-x) x}} e^{(\lambda - \frac{1}{2}) w(x)} dx\\
    & \leq \frac{C'}{\sigma}\int_{k_0}^{\lfloor C^*\sigma\rfloor +1}  e^{(\lambda - \frac{1}{2}) w(x)} dx\\
    & = \frac{C'}{\sigma} \int_{w(k_0)}^{w(\lfloor C^*\sigma^2\rfloor + 1)} e^{(\lambda - \frac{1}{2}) y} \left(\log\left(\frac{w^{-1}(y)}{np} \frac{n(1-p)}{n-w^{-1}(y)}\right)\right)^{-1} dy\\  
    & = \frac{C'}{\sigma} \int_{w(k_0)}^{w(\lfloor C^*\sigma^2\rfloor + 1)} e^{(\lambda - \frac{1}{2}) y} \left(\log\left(\frac{w^{-1}(y)}{np} \frac{1}{1 - \frac{w^{-1}(y) - np}{n(1-p)}}\right)\right)^{-1} dy\\
    & \leq \frac{C'}{\sigma} \int_{w(k_0)}^{w(\lfloor C^*\sigma^2\rfloor + 1)} e^{(\lambda - \frac{1}{2}) y} \left(\log\left(\frac{w^{-1}(y)}{np} \left(1+\frac{ w^{-1}(y) - np}{n(1-p)}\right)\right)\right)^{-1} dy.
\end{align*}
Moreover, we have 
$w(x) \leq \frac{(x-np)^2}{2\sigma^2}$ for any $x \in [np,n]$. 
Therefore, $w^{-1}(y) \geq 
 np + \sigma \sqrt{2y}$. Therefore, we can further control the display above as
 \begin{align*}
     & \quad \sum_{I_1} \sqrt{\frac{n}{(n-k)k}} e^{\left(\lambda - \frac{1}{2}\right)w(k)}\\
     & \leq \frac{C'}{\sigma} \int_{w(k_0)}^{w(\lfloor C^*\sigma^2\rfloor + 1)} e^{(\lambda - \frac{1}{2}) y} \left(\log\left(\left(1+ \frac{\sigma\sqrt{2y}}{np}\right) \left(1+\frac{\sigma\sqrt{2y}}{n(1-p)}\right)\right)\right)^{-1} dy\\
     & \leq \frac{C'}{\sigma} \int_{w(k_0)}^{w(\lfloor C^*\sigma^2\rfloor + 1)} e^{(\lambda - \frac{1}{2}) y} \left(\log\left(1+ \frac{\sqrt{2y}}{\sigma}\right)\right)^{-1} dy.
 \end{align*}
Moreover, we have $w(x) \leq \frac{(x-np)^2}{\sigma^2} $ for any $x >0$ by Lemma~\ref{lem_w_quadratic} so that $w(\lfloor C^*\sigma^2\rfloor + 1) \leq C' \sigma^2$, which yields that, for some constant $c'$ that depends on $C^*$, we have
\begin{align*}
    \log\left(1+ \frac{\sqrt{2y}}{\sigma}\right) \geq c'\frac{\sqrt{y}}{\sigma}, \qquad \forall y \in \left[{w(k_0)},{w(\lfloor C^*\sigma^2\rfloor + 1)} \right].
\end{align*}
We obtain
\begin{align*}
    \sum_{I_1} \sqrt{\frac{n}{(n-k)k}} e^{\left(\lambda - \frac{1}{2}\right)w(k)}
    & \leq \frac{C'}{\sigma} \int_{w(k_0)}^{w(\lfloor C^*\sigma^2\rfloor + 1)} e^{(\lambda - \frac{1}{2}) y} \frac{\sigma}{\sqrt{y}} dy\\[5pt]
    & \leq C' \int_{(\frac{1}{2} - \lambda)w(k_0)}^\infty e^{-z} \frac{1}{\sqrt{z}} \frac{1}{\sqrt{1/2 - \lambda}} dz\\[5pt]
    & \leq C' \int_{(\frac{1}{2} - \lambda)(\nu_a + \frac{1}{\lambda})}^\infty e^{-z} \frac{1}{\sqrt{z}} \frac{1}{\sqrt{1/2 - \lambda}} dz\\[5pt]
    & \leq \frac{C' \exp\left(- (\frac{1}{2} - \lambda)(\nu_a + \frac{1}{\lambda})\right)}{\sqrt{(\frac{1}{2} - \lambda)(\nu_a + \frac{1}{\lambda})}},
\end{align*}
which yields
\begin{align}
    \frac{C_0}{\sqrt{2\pi}c_0^2} e^{-\lambda \nu_a} \sum_{I_1} \sqrt{\frac{n}{(n-k)k}} e^{\left(\lambda - \frac{1}{2}\right)w(k)} \leq \frac{C_0}{\sqrt{2\pi}c_0^2} e^{-\lambda \nu_a} \frac{C' \exp\left(- (\frac{1}{2} - \lambda)(\nu_a + \frac{1}{\lambda})\right)}{\sqrt{(\frac{1}{2} - \lambda)(\nu_a + \frac{1}{\lambda})}}\label{eq_thirn_term_I1}
\end{align}

We now show that, for any $k \in I_2 \setminus \{n-1\}$, we have
\begin{align*}
    \sqrt{\frac{n}{(n-k-1)(k+1)}} e^{\left(\lambda - \frac{1}{2}\right)w(k+1)} \leq \sqrt{2} e^{-2} \sqrt{\frac{n}{(n-k)k}} e^{\left(\lambda - \frac{1}{2}\right)w(k)}.
\end{align*}
Indeed, we have
\begin{align*}
    \sqrt{\frac{n}{(n-k-1)(k+1)}}  \leq \sqrt{2}\sqrt{\frac{n}{(n-k)k}}.  
\end{align*}
Moreover, since the function $w$ is increasing and convex over $[np,n)$ by Lemma~\ref{lem_W_cvx_increasing}, we have 
\begin{align*}
    w(k+1) - w(k)\geq w'(k) &\geq  2 \log\left(\frac{k}{np} \frac{n(1-p)}{n-k}\right)\\
    & \geq 2 \log\left(\frac{k_0}{np} \frac{n(1-p)}{n-k_0}\right) \qquad \text{ since $w'$ is increasing over }[np,n]\\
    & \geq \log\left(\frac{np + \lfloor C^*\sigma^2 \rfloor}{np} \frac{n(1-p)}{n-np - \lfloor C^*\sigma^2 \rfloor}\right)\\
    & = \log\left(\left(1 + \frac{\lfloor C^*\sigma^2 \rfloor}{np} \right)\frac{1}{1 - \frac{\lfloor C^*\sigma^2 \rfloor}{n (1-p)}}\right)\\
    & \geq \log\left(\left(1 + \frac{\lfloor C^*\sigma^2 \rfloor}{np} \right)\left(1 + \frac{\lfloor C^*\sigma^2 \rfloor}{n (1-p)}\right)\right)\\
    & \geq \log \left(1+\frac{\lfloor C^*\sigma^2 \rfloor}{\sigma^2}\right)\\
    & \geq 8
\end{align*}
for $C^*$ larger than an absolute constant. 
For $\lambda \leq 1/4$, we therefore obtain
\begin{align*}
    \sqrt{\frac{n}{(n-k-1)(k+1)}} e^{\left(\lambda - \frac{1}{2}\right)w(k+1)} \leq \sqrt{2} e^{-2} \sqrt{\frac{n}{(n-k)k}} e^{\left(\lambda - \frac{1}{2}\right)w(k)}.
\end{align*}
Therefore, the sum over $I_2$ can be controlled as by a geometric series, and we obtain 
\begin{align*}
    \frac{C_0}{\sqrt{2\pi}c_0^2} e^{-\lambda \nu_a} \sum_{k \in I_2} \sqrt{\frac{n}{(n-k)k}} e^{\left(\lambda - \frac{1}{2}\right)w(k)} & \leq C' e^{-\lambda \nu_a} \sqrt{\frac{n}{(n-\lfloor C^* \sigma^2 \rfloor)\lfloor C^* \sigma^2 \rfloor}} e^{\left(\lambda - \frac{1}{2}\right)w(\lfloor C^* \sigma^2 \rfloor)}\\
    & \leq C'' e^{-\lambda \nu_a} \frac{1}{\sigma} e^{\left(\lambda - \frac{1}{2}\right)\left(\nu_a + 1/\lambda\right)}\numberthis\label{eq_thirn_term_I2}
\end{align*}
provided $C^*$ is large enough. 

To conclude, it remains to control the terms $e^{-\lambda \nu_a}   e^{\lambda w(n)}\mathbb P(Y = n) $ and $ e^{-\lambda \nu_a}   e^{\lambda w(k_0)}\mathbb P(Y = k_0)$. 
For $k \in \{k_0,n\}$, we have
\begin{align*}
    & \quad e^{-\lambda \nu_a}   e^{\lambda w(k)}\mathbb P(Y = k) \\
    & \leq e^{-\lambda \nu_a}   e^{\lambda w(k)}\frac{C_0}{c_0^2\sqrt{2\pi}} \sqrt{\frac{n}{(n-k)k}} \exp\left(-n(1-p) \,h_B\!\left(-\frac{k-np}{n(1-p)}\right) - np \, \, h_B\!\left(\frac{k-np}{np}\right)\right) \text{ by Lemma~\ref{lem_Stirling}}\\
    & \leq \frac{C_0}{c_0^2\sqrt{2\pi}} e^{-\lambda \nu_a}\sqrt{\frac{2}{k}} e^{(\lambda - \frac{1}{2}) w(k)}\\
    & \leq \frac{C_0}{c_0^2\sqrt{2\pi}} e^{-\lambda \nu_a}\sqrt{\frac{2}{\sigma^2}} e^{(\lambda - \frac{1}{2}) (\nu_a + \frac{1}{\lambda})}\numberthis \label{eq_thirn_term_remainder}
\end{align*}
by definition of $k_0$. 
Combining equations~\eqref{eq_thirn_term_moment_generating_function}, \eqref{eq_thirn_term_I1}, \eqref{eq_thirn_term_I2} and~\eqref{eq_thirn_term_remainder}, we obtain
\begin{align*}
    \mathbb E\left[e^{\lambda X} \one{\big(X\geq1/\lambda\big)}\right] &\leq C e^{-\lambda \nu_a + (\lambda - \frac{1}{2})(\nu_a + \frac{1}{\lambda})}\\
    & \leq C e^{-\frac{\nu_a}{2} - \frac{1}{2\lambda}} \\
    & \leq C \lambda^2 e^{- \frac{\nu_a}{2}}.
\end{align*}

Assembling these calculations, we have for $\lambda \in [0,M]$
\[\mathbb{E} e^{\lambda X} \leq 1+C' \lambda^{2} e^{- ca^{2}} \leq \exp \left(C' \lambda^{2} e^{-ca^{2}}\right)\]
where $C' > 0$ is a constant that depends on $\bar{C}$. Then, for any $t>0$, we have
$$
\begin{aligned}
\mathbb{P}\left(\sum_{j=1}^{m}\left(W_{j}-\nu_{a}\right) \one_{\left\{Z_{j} \geq a\right\}}>t\right) & \leq \inf _{\lambda \leq M} e^{-\lambda t}\left(\mathbb{E} e^{\lambda X}\right)^{m} \leq \inf _{\lambda \leq M} \exp \left(-\lambda t + C' m \lambda^{2} e^{-ca^{2}}\right) \\
& =\exp \left\{-\left(\frac{Mt^{2} e^{ca^{2}}}{C' m} \wedge \frac{t}{2M}\right)\right\}
\end{aligned}
$$

Setting $t=C^*\left(\sqrt{m e^{-ca^{2}} x}+x\right)$ for $C^*$ sufficiently large, we obtain the desired result.
\end{proof}

\begin{lemma}\label{lem_truncchisq_mean}
Let $n,k \in \mathbb N$ such that $k \leq n/2$, $p_0 \in (0,1/4)$ and $p_1 \in [0,1]$. 
Let $Y = A+B$ where $A \sim \operatorname{Bin}(k,p_1)$ and $B \sim \operatorname{Bin}(n-k,p_0)$ are independent and define 
$Z = (Y-np_0)/\sigma, \text{ where } \sigma = \sqrt{np_0(1-p_0)}.$ 
Let $w$ denote the function defined in~\eqref{eq_def_W_appendix} and let $W = w(Y)$.
Let $\theta = \frac{k(p_1 - p_0)}{\sigma}$ and recall the definition of $\nu_a$ from~\eqref{eq_def_nu_a} for $a > 0$. 
Then there exist universal constants $\tilde c, c,C>0$ such that, if $kp_1 \geq C$, then for any $a \in [C,c\sigma]$,
$$
\mathbb{E}\left\{\left(W-\nu_{a}\right) \one{\{Z \geq a\}}\right\} \begin{cases}=0, & \text { if } \theta = 0   \\[5pt]
\geq \tilde c \, \sigma \theta \log\left(1+ \frac{\theta}{\sigma}\right), & \text { if } \theta \geq C a.
\end{cases}
$$

\end{lemma}

\begin{proof}[Proof of Lemma~\ref{lem_truncchisq_mean}]
    The fact that $\mathbb{E}\left\{\left(W-\nu_{a}\right) \one{\{Z \geq a\}}\right\}=0$ when $\theta=0$ follows by definition of $\nu_{a}$. 
    
    Suppose now that $\theta \geq C a$. 
    Recalling that the function $w$ is increasing over $[np_0,n]$ by Lemma~\ref{lem_W_cvx_increasing}, we have
    \begin{align*}
        &\quad \mathbb E \left[(W- \nu_a) \one{\big(Z \geq a\big)}\right] \\
        & \geq \mathbb E \left[W \one{\big(Z \geq \theta\big)}\right] - \nu_a\\
        & = E \left[w(Y) \one{\big(Y \geq kp_1 + (n-k)p_0\big)}\right] - \nu_a\\
        & \geq w(kp_1 + (n-k)p_0) \mathbb P\left(Y \geq kp_1 + (n-k)p_0\right) - \bar C a^2 \qquad \text{ by Corollary~\ref{cor_lemma17_LiuGaoSamworth}}\\
        & = c'\mathbb P\left(Y \geq kp_1 + (n-k)p_0\right)\sigma \theta \log\left(1+ \frac{\theta}{\sigma}\right) - \bar C a^2 \qquad \text{ for some $c'>0$ by Corollary~\ref{cor_w_growth_w}.}
    \end{align*}
    Moreover, the median $m_{n,p}$ of a binomial distribution with parameters $n$ and $p$ satisfies $|m_{n,p} - np|< \log(2)$ (see~\cite{HAMZA199521}, Theorem 2). 
    Therefore, we have
    \begin{align*}
        \mathbb P\left(Y \geq np_1\right) & \geq \mathbb P\left(A \geq kp_1\right)\mathbb P\left(B \geq (n-k)p_0\right) 
    \end{align*}
    We show that $P\left(A \geq kp_1\right) \geq \frac{1}{4}$ provided $kp_1 \geq C$ for some large enough constant $C>0$. 
    We have
    \begin{align*}
        P\left(A \geq kp_1\right)&\geq \mathbb P\left(A \geq m_{k,p_1}\right) - \mathbb P\left(A = \lfloor kp_1\rfloor - 1\right)\\
        & \geq \frac{1}{2} - \mathbb P\left(A = \lfloor kp_1\rfloor - 1\right) .
    \end{align*}
    Moreover, letting $j = \lfloor kp_1\rfloor - 1$, we have
    \begin{align*}
        \mathbb P\left(A = j\right) &= {k \choose j} p_1^j (1-p_1)^{k-j} \\
        & \leq C_0 \sqrt{\frac{k}{(k-j)j}} \left(\frac{k(1-p_1)}{k-j}\right)^{k-j} \left(\frac{kp_1}{j}\right)^j \qquad \text{ by Lemma~\ref{lem_Stirling}}\\
        &\leq C_0 \frac{2}{\sqrt{kp_1}}\left(1+\frac{kp_1-j}{j}\right)^j \\
        & \leq C_0 \frac{2}{\sqrt{kp_1}}\exp\left(kp_1-j\right)\\
        & \leq 2C_0 e^2 \frac{1}{\sqrt{kp_1}}\\
        & \leq \frac{1}{4}
    \end{align*}
    provided $kp_1 \geq C$, as claimed. 
    Similarly, we can show that $P\left(B \geq (n-k)p_0\right) \geq \frac{1}{4}$ provided $(n-k)p_0 \geq C$ for some large enough constant $C>0$.  
    Therefore, we have proved $\mathbb P(Y \geq np_1) \geq \frac{1}{16}$, which yields
    \begin{align*}
        \mathbb E \left[(W- \nu_a) \one{\big(Z \geq a\big)}\right] \geq \frac{c'}{16}\sigma \theta \log\left(1+ \frac{\theta}{\sigma}\right) - \bar C a^2.
    \end{align*}
    Moreover, since $\theta \geq Ca$, we have
    \begin{align*}
        \sigma \theta \log\left(1+ \frac{\theta}{\sigma}\right)&\geq\sigma Ca \log\left(1+ \frac{Ca}{\sigma}\right) \\
        & \geq \sigma Ca \log(2) \cdot \frac{Ca}{\sigma} \qquad \text{ provided $c \leq 1/C$}\\
        & \geq C^2 \log(2) a^2.
    \end{align*}
    We may now choose $C$ large enough that $\bar C a^2 \leq \frac{c'}{32}\sigma \theta \log\left(1+ \frac{\theta}{\sigma}\right)$,
    which yields
    \begin{align*}
        \mathbb E \left[(W- \nu_a) \one{\big(Z \geq a\big)}\right] \geq \frac{c'}{32}\sigma \theta \log\left(1+ \frac{\theta}{\sigma}\right) 
    \end{align*}
    and concludes the proof.

\end{proof}

\begin{lemma}\label{lem_variance_trunc_chi2}
Let $n,k \in \mathbb N$ such that $k \leq n/2$, $p_0 \in (0,1/4)$ and $p_1 \in [0,1]$. 
Let $Y = A+B$ where $A \sim \operatorname{Bin}(k,p_1)$ and $B \sim \operatorname{Bin}(n-k,p_0)$ are independent and define 
$Z = (Y-np_0)/\sigma, \text{ where } \sigma = \sqrt{np_0(1-p_0)}.$ 
Let $w$ denote the function defined in~\eqref{eq_def_W_appendix} and let $W = w(Y)$.
Let $\theta = \frac{k(p_1 - p_0)}{\sigma}$ and recall the definition of $\nu_a$ from~\eqref{eq_def_nu_a} for $a > 0$. 
Then there exist universal constants $C_1, c,C>0$ such that, if $kp_1 \geq C$, then for any $a \in [C,c\sigma]$,
$$
\operatorname{Var}\left\{\left(W-\nu_{a}\right) \one{\{Z \geq a\}}\right\} \leq \begin{cases}C_{1} \left( \sigma a \log\left(1+ \frac{a}{\sigma}\right)\right)^2 \exp\left(-\frac{3}{8}a^2\right) & \text { if } \theta=0 \\[5pt]
C_{1} (\sigma^2 + \sigma \theta )\log^2\left(1+\frac{\theta}{\sigma}\right) & \text { if }\theta \geq 2 a.\end{cases}
$$
\end{lemma}

\begin{proof}[Proof of Lemma~\ref{lem_variance_trunc_chi2}]
    
Let $a \geq C$. In this proof, we use the notation $\tilde C$ to denote a constant whose value may change in each appearance. 
When $\theta=0$, we have 
$$
\begin{aligned}
\operatorname{Var}\left\{\left(W-\nu_{a}\right) \one{\{Z \geq a\}}\right\} & \leq \mathbb{E}\left\{\left(W-\nu_{a}\right)^{2} \one{\{Z \geq a\}}\right\}\\
& = \mathbb E\left[W^2\one{\{Z \geq a\}}\right] - 2 \nu_a \mathbb E\left[W \one{\{Z \geq a\}}\right] + \nu_a^2 \mathbb P\left(Z\geq a\right)\\
& = \left(\gamma_a - \nu_a^2\right)\mathbb P\left(Z\geq a\right)\\
& \leq \bar C \left( \sigma a \log\left(1+ \frac{a}{\sigma}\right)\right)^2 \exp\left(-\frac{\frac{1}{2} a^2}{1+ \frac{a/\sigma}{3}}\right) \qquad \text{ by Corollary~\ref{cor_lemma17_LiuGaoSamworth} and Lemma~\ref{lem_rescalen_Bernstein}}\\
& \leq \bar C \left( \sigma a \log\left(1+ \frac{a}{\sigma}\right)\right)^2 \exp\left(-\frac{3}{8}a^2\right).
\end{aligned}
$$
For $\theta \neq 0$, we have
\begin{align*}
& \quad \operatorname{Var}\left\{\left(W-\nu_{a}\right) \one{\{Z \geq a\}}\right\} \\[5pt]
& =\mathbb{E}\left[\operatorname{Var}\left\{\left(W-\nu_{a}\right) \one{\{Z \geq a\}} \mid \one{\{Z \geq a\}}\right\}\right] \\[5pt]
& \quad+\operatorname{Var}\left[\mathbb{E}\left\{\left(W-\nu_{a}\right) \one{\{Z \geq a\}} \mid \one{\{Z \geq a\}}\right\}\right] \\[5pt]
& =\mathbb{P}(Z \geq a) \operatorname{Var}\left(W| Z \geq a\right)  \\[5pt]
& \quad+\mathbb{P}(Z<a) \mathbb{P}(Z \geq a)\left\{\mathbb{E}\left(W-\nu_{a}|Z \geq a\right)\right\}^{2}.\numberthis\label{eq_decomp_variance_H1}
\end{align*}

Moreover, we have
\begin{align*}
\mathbb{P}(Z \geq a) \operatorname{Var}\left(W| Z \geq a\right) & \leq \mathbb{E}\left\{\operatorname{Var}\left(W \mid \one{\{Z \geq a\}}\right)\right\} \\[9pt]
& \leq \operatorname{Var}\left(W\right) = \mathbb E [W^2] - \mathbb E^2[W].\numberthis\label{eq_first_part_variance_H1}
\end{align*}
Now, we recall that $\mathbb EY = kp_1 + (n-k)p_0$. 
We observe that, for any $x \in [0,n]$
\begin{align*}
    w(x) &= (n-x) \log\left(\frac{n-x}{n(1-p_0)}\right) + x \log\left(\frac{x}{np_0}\right)\\
    & = (n-x) \log\left(\frac{n-\mathbb E Y}{n(1-p_0)}\right) + x \log\left(\frac{\mathbb E Y}{np_0}\right) + (n-x) \log\left(\frac{n-x}{n - \mathbb E Y}\right) + x \log\left(\frac{x}{\mathbb EY}\right)\\
    & =: nD_{KL}\left(\frac{\mathbb EY}{n} \, \Big|\Big|\,p_0\right) + (x-\mathbb EY) \log\left(\frac{\mathbb EY}{np_0}\frac{n(1-p_0)}{n-\mathbb EY}\right) + w_{p_1}(x),
\end{align*}
where $D_{KL}(p||q) = (1-p) \log\left(\frac{1-p}{1-q}\right) + p \log\left(\frac{p}{q}\right)$ for any $p,q \in (0,1)$.
Therefore, we obtain
\begin{align*}
    &\mathbb E(W) = nD_{KL}\left(\frac{\mathbb EY}{n} \, \Big|\Big|\,p_0\right) + \mathbb E(w_{p_1}(Y))\\
    \implies & \mathbb E^2(W) \geq \left(nD_{KL}\left(\frac{\mathbb EY}{n} \, \Big|\Big|\,p_0\right)\right)^2 + 2nD_{KL}\left(\frac{\mathbb EY}{n} \, \Big|\Big|\,p_0\right)\mathbb E(w_{p_1}(Y)).
\end{align*}
Moreover, using the inequality $ab \leq a^2+b^2$ that holds true for any $a,b\in \R$, we obtain
\begin{align*}
    \mathbb E [W^2] &= \left(nD_{KL}\left(\frac{\mathbb EY}{n} \, \Big|\Big|\,p_0\right)\right)^2 + \mathbb V(Y) \log^2\left(\frac{\mathbb EY}{np_0}\frac{n(1-p_0)}{n-\mathbb EY}\right) + \mathbb E (w_{p_1}(Y)^2) + 2 nD_{KL}\left(\frac{\mathbb EY}{n} \, \Big|\Big|\,p_0\right)\mathbb E(w_{p_1}(Y)) \\
    & + \mathbb E \left[(Y-\mathbb EY) \log\left(\frac{\mathbb EY}{np_0}\frac{n(1-p_0)}{n-\mathbb EY}\right)  w_{p_1}(Y)\right]\\
    &\leq \left(nD_{KL}\left(\frac{\mathbb EY}{n} \, \Big|\Big|\,p_0\right)\right)^2 + 2\mathbb V(Y) \log^2\left(\frac{\mathbb EY}{np_0}\frac{n(1-p_0)}{n-\mathbb EY}\right) + 2 \mathbb E (w_{p_1}(Y)^2) + 2 nD_{KL}\left(\frac{\mathbb EY}{n} \, \Big|\Big|\,p_0\right)\mathbb E(w_{p_1}(Y)).
\end{align*}
Recalling $k \leq n/2$, we have $\frac{k(p_1-p_0)}{n} \leq 1/2$. We obtain
\begin{align*}
    \mathbb V(W) &\leq 2\mathbb V(Y) \log^2\left(\frac{\mathbb EY}{np_0}\frac{n(1-p_0)}{n-\mathbb EY}\right) + 2 \mathbb E (w_{p_1}(Y)^2) \\
    & \leq 2\mathbb V(Y) \log^2\left(\frac{\mathbb EY}{np_0}\frac{n(1-p_0)}{n-\mathbb EY}\right) + 2 \mathbb E \left(4\left(\frac{Y-\mathbb EY}{\sigma(Y)}\right)^4\right) \qquad \text{ by Lemma~\ref{lem_w_quadratic}.1}\\
    & \leq 2\left(kp_1 + (n-k) p_0\right) \log^2 \left(\left(1 + \frac{k(p_1 - p_0)}{np_0}\right)\frac{n(1-p_0)}{n(1-p_0) - k(p_1-p_0))}\right) + C\\
    & \leq C np_0\left(1+\frac{k(p_1-p_0)}{np_0}\right) \log\left(1+\frac{k(p_1-p_0)}{np_0(1-p_0)}\right) + C \qquad \text{since } \frac{k(p_1-p_0)}{n(1-p_0)} \leq 1/2\\
    & \leq C\left(\sigma^2 + \theta \sigma\right) \log^2\left(1+\frac{\theta}{\sigma}\right).
\end{align*}

Finally, writing $Z=\theta+X$ where $X = \frac{Y-kp_1 - (n-k)p_0}{\sigma}$, we have by Lemma~\ref{lemma17_LiuGaoSamworth}
\begin{align*}
& \quad \mathbb{E}\left(W-\nu_{a}| Z \geq a\right) \\
& = \mathbb E \Big[w(kp_1 + (n-k)p_0 + \sigma X) \, \big| \, \theta + X \geq a\Big]\\
& \leq \mathbb E \Big[w(kp_1 + (n-k)p_0 + \sigma X) \, \big| \, X \geq a \Big]\\
& \leq \tilde C \, \mathbb E \Big[(k(p_1-p_0) + \sigma X) \log\left(1+ \frac{k(p_1-p_0) + \sigma X}{\sigma^2}\right) \, \big| \, X \geq a \Big] \quad \text{ by Corollary~\ref{cor_w_growth_w}}\\
& \leq \tilde C \left(\sigma \theta \log\left(1+ \frac{\theta}{\sigma}\right) + \mathbb E \left[\sigma X \log\left(1+ \frac{X}{\sigma}\right) \, \big| \, X \geq a \right]\right)\\
& \leq \tilde C \left(\sigma \theta \log\left(1+ \frac{\theta}{\sigma}\right) + \sigma a \log\left(1+ \frac{a}{\sigma}\right)\right)\\
& \leq \tilde C  \sigma  \theta\log\left(1+ \frac{\theta}{\sigma}\right). \numberthis\label{eq_seconn_part_variance_H1}
\end{align*}

If $|\theta| \geq 2 a$, notice that, by Bernstein's inequality, we have
\begin{align*}
    \mathbb{P}(Z<a) &= \mathbb P \left(X < -(\theta-a)\right)\leq \exp\left(-\frac{\frac{1}{2} (a-\theta)^2}{1 + \frac{(\theta - a)/\sigma}{3}}\right)\\
    & \leq \exp\left(-\frac{\frac{1}{2} \big(1-{\bar C}^{-1}\big)^2\theta^2}{1 + \frac{(1-{\bar C}^{-1})\theta/\sigma}{3}}\right)\\
    &\leq \tilde C / \theta^{2}\\
    & \leq \frac{\tilde C}{\sigma \theta \log\left(1+\frac{\theta}{\sigma}\right)}\numberthis\label{eq_Bernstein_variance_H1}
\end{align*}
if $\theta \geq C$ for some large enough $C>0$. 
The result follows by combining~\eqref{eq_decomp_variance_H1}, \eqref{eq_first_part_variance_H1}, \eqref{eq_seconn_part_variance_H1} and~\eqref{eq_Bernstein_variance_H1}.

\end{proof}

\begin{lemma}\label{lem:gaussian_tail_momentk_ellt}
Let $Y \sim \mathcal{N}(0,1)$ and let $ \alpha, \beta \geq 1 $ and $x>0$ such that $x^2 \geq 2(\alpha-1) $. Then we have 
\begin{align*}
    \mathbb E \big[ Y^\alpha \one_{Y\geq x}\big] &\leq 2 x^{\alpha -1} e^{-x^2/2} & \text{ provided $x^2 \geq 2(\alpha-1) $}\\
    \mathbb E \big[ Y^\alpha \log(1+Y)^\beta\one_{Y\geq x}\big] &\leq 2 x^{\alpha -1} \log(1+x)^\beta e^{-x^2/2}& \text{ provided $x^2 \geq 2(\alpha+\beta-1) $.}
\end{align*}

\end{lemma}

\begin{proof}[Proof of Lemma~\ref{lem:gaussian_tail_momentk_ellt}]
By integration by parts, we have 
\begin{align*}
    \int_x^\infty y^\alpha e^{-y^2/2} dy &= x^{\alpha-1} e^{-x^2/2} + (\alpha\!-\!1)\! \int_x^{\infty} \! y^{\alpha-2} e^{-y^2/2} dy\\
    &\leq x^{\alpha-1} e^{-x^2/2} + \frac{1}{2} \int_x^\infty y^\alpha e^{-y^2/2} dy ~~ \text{using that $y^2 \geq 2(\alpha-1)$ over $[x,+\infty)$},\\
    \text{so that } \quad 
        \mathbb E \big[ Y^\alpha \one_{Y\geq x}\big] &= \int_x^\infty y^\alpha e^{-y^2/2} dy \leq 2x^{\alpha-1} e^{-x^2/2}.
\end{align*}
Similarly,
\begin{align*}
    \int_x^\infty y^\alpha\log(1+y)^\beta e^{-y^2/2} dy &= x^{\alpha-1}\log(1+y)^\beta e^{-x^2/2} \\
    & \qquad + \! \int_x^{\infty} \! \left((\alpha\!-\!1)y^{\alpha-2} \log(1+y)^\beta + \beta \frac{y^{\alpha-1} }{1+y}\log(1+y)^{\beta-1} \right) e^{-y^2/2} dy\\
    & \leq x^{\alpha-1}\log(1+y)^\beta e^{-x^2/2} \\
    & \qquad + \! \int_x^{\infty} \! (\alpha + \beta - 1)y^{\alpha-2} \log(1+y)^\beta e^{-y^2/2} dy\\
    &\leq x^{\alpha-1}\log(1+y)^\beta e^{-x^2/2} \\
    & \qquad + \frac{1}{2} \int_x^\infty y^\alpha \log(1+y)^\beta e^{-y^2/2} dy ~~ \text{using that $y^2 \geq 2(\alpha+\beta-1)$ over $[x,+\infty)$},\\
    \text{so that } \quad 
        \mathbb E \big[ v(Y)^\alpha \one_{Y\geq x}\big] &= \int_x^\infty y^\alpha \log(1+y)^\beta e^{-y^2/2} dy \leq 2x^{\alpha-1}\log(1+y)^\beta e^{-x^2/2}.
\end{align*}
\end{proof}

\begin{lemma}\label{lem_moment_exponential}
    For any $\alpha>0$ and $x \geq 2\alpha$, it holds that
    \begin{align*}
        \int_{x}^\infty y^\alpha e^{-y} dy \leq 2 x^\alpha e^{-x}.
    \end{align*}
\end{lemma}
\begin{proof}[Proof of Lemma~\ref{lem_moment_exponential}]
    By integration by parts, we have
    \begin{align*}
        \int_{x}^\infty y^\alpha e^{-y} dy \leq 2 x^\alpha e^{-x} &= x^\alpha e^{-x} + \int_{x}^\infty \alpha x^{\alpha-1} e^{-x} dx\\
        & \leq x^\alpha e^{-x} + \frac{1}{2}\int_{x}^\infty  x^{\alpha} e^{-x} dx \qquad \text{ since } y \geq 2\alpha \text{ over } [x,\infty),
    \end{align*}
    from which it follows that
    \begin{align*}
        \int_{x}^\infty y^\alpha e^{-y} dy \leq 2 x^\alpha e^{-x},
    \end{align*}
    as claimed.
\end{proof}

\begin{lemma}\label{lem_rescalen_Bernstein}
    Let $n\in \mathbb N$ and $p \in (0,1/2)$. 
    Let $X \sim \operatorname{Bin}(n,p)$ and $\sigma = \sqrt{np(1-p)}$.
    For any $t > 0$, we have
    \begin{align*}
        \mathbb P\left(\frac{X - np}{\sigma} \geq t\right) \leq \exp\left(- \frac{t^2/2}{1+\frac{t}{3\sigma}}\right)\\
        \mathbb P\left(\frac{X - np}{\sigma} \leq -t\right) \leq \exp\left(- \frac{t^2/2}{1+\frac{t}{3\sigma}}\right)
    \end{align*}
\end{lemma}
\begin{proof}
    The result follows by a direct application of the Bernstein inequality for binomial distributions.
\end{proof}

\begin{lemma}\label{lem_Stirling}
    There exist two absolute constants $C_0, c_0$ such that, for any $n,k \in \mathbb N$ such that $k < n$ and $p \in (0,1)$, we have
    \begin{align*}
        {n \choose k} p^k (1-p)^{n-k} &\leq \frac{C_0}{c_0^2\sqrt{2\pi}} \sqrt{\frac{n}{(n-k)k}} \left(\frac{n(1-p)}{n-k}\right)^{n-k} \left(\frac{np}{k}\right)^k\\
        & = \frac{C_0}{c_0^2\sqrt{2\pi}} \sqrt{\frac{n}{(n-k)k}} \exp\left(-n(1-p) h_B\left(-\frac{k-np}{n(1-p)}\right) - np \, h_B\left(\frac{k-np}{np}\right)\right). 
    \end{align*}
\end{lemma}

\begin{proof}
    We recall that, by the Stirling formula, we have
\begin{align*}
m! \sim \sqrt{2\pi m} \left(\frac{m}{e}\right)^m, \qquad \text{ as }m \to \infty.
\end{align*}
Since the function $m\in \mathbb N \mapsto \frac{m!}{\sqrt{2\pi m} \left(\frac{m}{e}\right)^m}$ is positive and converges to $1$, there exist two absolute constants $c_0, C_0>0$ such that
\begin{align}
c_0\sqrt{2\pi m} \left(\frac{m}{e}\right)^m \leq m! \leq C_0 \sqrt{2\pi m} \left(\frac{m}{e}\right)^m, \qquad \forall m \in \mathbb N. \label{eq_Stirling}
\end{align}
The result follows.
\end{proof}

\begin{lemma}\label{lem_W_cvx_increasing}
    The function $w$ defined in~\eqref{eq_def_W_appendix} is convex and increasing over $[np_0, n)$.
\end{lemma}
\begin{proof}
    The derivative of $w$ is given by
    \begin{align*}
        w'(x) = 2 \log\left(\frac{x}{np_0} \frac{n(1-p_0)}{n-x}\right), \quad \forall x \in [np_0,n]
    \end{align*}
    which is positive if $x > np_0$. 
    The second derivative of $w$ is given by
    \begin{align*}
        \frac{2}{n-x} + \frac{2}{x}
    \end{align*}
    which is positive for any $x \in [np_0,n)$. 
\end{proof}

\subsection{Proofs of propositions in the main text}

We begin by restating and proving Proposition 1. 
\begin{proposition}
    Let $c > 0$. There exists a constant $C > 0$ such that if
    \[\frac{C}{k_1k_2} \log\left(e{n_1 \choose k_1}{n_2 \choose k_2}\right) \leq p_0 \leq \frac{1}{4},\]
    then it holds
    \(\bbP_0\left(\text{$\bG$ has an empty $k_1 \times k_2$ bipartite subgraph}\right) \leq c.\)
    
    Here, we say that a subgraph $\bG_1$ of $\bG$ is empty if all the vertices in $\bG_1$ has degree equal to 0.
\end{proposition}
\begin{proof}
\begin{align*}
    &\bbP_0\left(\text{$\bG$ has an empty $k_1 \times k_2$ bipartite subgraph}\right) \\ 
    &= \mathbb P_0\left(\exists (K_1,K_2) \in \mathcal{P}_{k_1}(n_1) \times \mathcal{P}_{k_2}(n_2), \forall (i,j) \in K_1 \times K_2: A_{ij} = 0\right) \\
    &\leq {n_1 \choose k_1} {n_2 \choose k_2}\bbP\left(\text{Bin}(k_1k_2, p_0) = 0\right) \\
    &=  {n_1 \choose k_1} {n_2 \choose k_2}(1 - p_0)^{k_1k_2} \\
    &\leq {n_1 \choose k_1} {n_2 \choose k_2} e^{-k_1k_2 p_0} \\
    &\leq {n_1 \choose k_1} {n_2 \choose k_2}\exp \left(-C\log\left(e{n_1 \choose k_1}{n_2 \choose k_2}\right)\right) \\
    &\leq \exp \left(-(C - 1)\log\left(e{n_1 \choose k_1}{n_2 \choose k_2}\right)\right) \quad \text{(for $C > 1$)} \\
    &\leq \exp\left(-(C - 1)\right) \\
    &\leq c,
\end{align*}
where the final inequality holds for $C$ taken sufficiently large.
\end{proof}
Now we state and prove Proposition 2. 
\begin{proposition}
    Let $c > 0$. There exists a constant $C > 0$ such that if
    \[\frac{C}{n_2k_1} \log\left(e{n_1 \choose k_1}\right) \leq p_0 \leq \frac{1}{4},\]
    then it holds
    \(\bbP_0\Big(\exists I \in \mathcal{P}_{k_1}(n_1): \sum_{i \in I} \sum_{j=1}^{n_2} A_{ij} = 0\Big) \leq c.\)
\end{proposition}
\begin{proof}
    \begin{align*}
    &\bbP_0\Big(\exists I \in \mathcal{P}_{k_1}(n_1): \sum_{i \in I} \sum_{j=1}^{n_2} A_{ij} = 0\Big) \\ 
    &\leq {n_1 \choose k_1} \bbP\left(\text{Bin}(k_1n_2, p_0) = 0\right) \\
    &=  {n_1 \choose k_1}(1 - p_0)^{k_1n_2} \\
    &\leq {n_1 \choose k_1} e^{-k_1n_2 p_0} \\
    &\leq {n_1 \choose k_1} \exp \left(-C\log\left(e{n_1 \choose k_1}\right)\right) \\
    &\leq \exp \left(-(C - 1)\log\left(e{n_1 \choose k_1}\right)\right) \quad \text{(for $C > 1$)} \\
    &\leq \exp\left(-(C - 1)\right) \\
    &\leq c,
\end{align*}
where the final inequality holds for $C$ taken sufficiently large.
\end{proof}
Finally, we state and prove Proposition 3.
\begin{proposition}
Suppose that $k_1^2 \geq \bar{c}n_1 k_2$ for a constant $\bar{c} > 0$ and $k_j \leq c_j n_j$ for $j \in \{1, 2\}$ where $c_1,c_2 > 0$ are sufficiently small constants. Additionally, suppose that there exists a constant $\alpha > 0$ such that $n_2 \geq k_2^{2 + \alpha}$ and that $\frac{n_1}{k_1} \geq e \log(\frac{n_2}{k_2})$. Then it holds
\begin{equation}
    \frac{(\delta^*)
^2}{p_0(1-p_0)} \asymp \frac{1}{k_2}\log\left(1 + \frac{n_1k_2}{k_1^2}\log(n_2)\right).
\end{equation}
In particular, this reveals a phase transition at $\frac{n_1k_2}{k_1^2}\log(n_2)\asymp 1$.
\end{proposition}
\begin{proof}
    Notice that the assumption $\frac{n_1}{k_1} \geq e \log(\frac{n_2}{k_2})$ can be made without loss of generality by Lemma \ref{lem_2.6}. Theorems 1 and 2 together imply
    \[\frac{(\delta^*)
^2}{p_0(1-p_0)} \asymp R.\]
Therefore, we need to show
\[R \asymp \frac{1}{k_2}\log\left(1 + \frac{n_1k_2}{k_1^2}\log(n_2)\right).\]
Note that the assumption $n_2 \geq k_2^{2 + \alpha}$ implies that $\phi_{21} = \infty$ if $c_2$ is large enough. 
Then by Lemma \ref{lem_simplify_dense_allrates}, it holds
\[R \asymp \psi_{21} \land \phi_{12}.\]
Furthermore, $n_2 \geq k_2^{2 + \alpha}$ implies that $\log(1 + \frac{n_2}{k_2^2}) \asymp \log(n_2)$. Using the inequality $\log(1 + x) \leq x$ which holds for any $x \geq 0$, we have
\begin{align*}
    \psi_{21} &= \frac{1}{k_2}\log\left(1 + \frac{n_1}{k_1^2}\log\left(e{n_2 \choose k_2}\right)\right)\\
    &\asymp \frac{1}{k_2}\log\left(1 + \frac{n_1k_2}{k_1^2}\log(n_2 / k_2)\right) \\&\asymp \frac{1}{k_2} \log\left(1 + \frac{n_1k_2}{k_1^2}\log(n_2)\right) \\
    &\leq \frac{1}{k_2} \frac{n_1k_2}{k_1^2}\log(n_2) \\
    &\asymp \frac{n_1}{k_1^2}\log\left(1 + \frac{n_2}{k_2^2}\right)\\
    & = \phi_{12}.
\end{align*}
This implies $R \asymp \psi_{21} \land \phi_{12} = \psi_{21} \asymp \log\left(1 + \frac{n_1k_2}{k_1^2}\log(n_2)\right)$. The proof is complete.
\end{proof}
\subsection{Miscellaneous analytic lemmas}

\begin{lemma}\label{lem_2.6}
    For any two real numbers $x,y>1$, at least one of the two inequalities $x \geq e \log(y)$ or $y \geq e\log(x)$ holds.
\end{lemma}
\begin{proof}[Proof of Lemma ~\ref{lem_2.6}]
    We first prove a preliminary result: For any $x > 1$, it holds that $x \geq e\log(e \log(x))$. 
    To see this, define the function $f: (1,\infty) \to \R;~ x \mapsto x - e\log(e \log(x))$. We have for any $x>1$
    \begin{align*}
        f'(x) = 1 - \frac{e}{x\log(x)}.
    \end{align*}
    The only value $x^*$ for which $f(x^*)=0$ is $x^* = e$, which implies that $f$ is decreasing over $(1,e)$ and increasing over $(e,\infty)$. Hence, $f$ is minimized at $e$ and its corresponding minimum value is
    \begin{align*}
        f(e) = e - e\log(e\log(e)) = 0.
    \end{align*}

    This fact being established, assume now for the sake of contradiction that there exist two real numbers $x,y$ such that $x<e \log(y)$ and $y< e\log(x)$. 
    Then we obtain $x< e\log(e\log(x))$ which is a contradiction. This concludes the proof.
\end{proof}

\begin{lemma}\label{lemma:probXk}

Let $X \sim \text{Bin}(n, p)$ with $p \leq \frac12$. Then for any $k \in \{1, ..., n\}$, we have 

\[\Pr(X = k) \leq \Big(\frac{2enp}{k}\Big)^k\exp\big(-np\big)\]
    
\end{lemma}

\begin{proof}
    Using the bound ${n \choose k } \leq \big(\frac{ne}{k}\big)^k$ (Appendix A in \cite{roch2024modern}), we have 

    \begin{align*}
        \Pr(X = k) &= {n \choose k}p^k\big(1 - p\big)^{n - k} \\
        &\leq \Big(\frac{npe}{k}\Big)^k\big(1 - p\big)^{n - k} \\
        &=  \Big(\frac{npe}{k}\Big)^k\big(1 - p\big)^{n}\big(1 - p\big)^{-k} \\ 
        &\leq \Big(\frac{npe}{k}\Big)^k\big(1 - p\big)^{n}\big(1/2\big)^{-k} \\
        &= \Big(\frac{2enp}{k}\Big)^k\Big(1 - p\Big)^{n} \\
        &\leq \Big(\frac{2enp}{k}\Big)^k\exp\big(-np\big)
    \end{align*}
where the final inequality uses $(1 - x)^b \leq \exp(-xb)$ for any $x, b \geq 0$.
\end{proof}

\begin{lemma}\label{lemma:log-plus-one}
    Let $f : S \rightarrow (0,\infty)$ be differentiable, where $S \subseteq (0,\infty)$ is an interval. For $x \in S$, define $g(x) = \frac1x\log\big(f(x)\big)$ and $g^{(1)}(x) = \frac1x \log\big(1 + f(x)\big)$. If $g$ is decreasing on $S$, then $g^{(1)}$ is decreasing on $S$ as well.
\end{lemma}

\begin{proof}
    Let $f'(x) = \frac{\D f}{\D x}(x)$ for any $x \in S$. By direct calculation, we have 

    \[\frac{\D g}{ \D x}(x) = \frac{f'(x)}{x f(x)} - \frac{\log \big(f(x)\big)}{x^2}.\]

    If $g$ is decreasing on $S$, then for each $x \in S$ it holds

    \[\frac{f'(x)}{x f(x)} < \frac{\log \big(f(x)\big)}{x^2}.\]

    Now calculating the derivative of $g^{(1)}$ at any $x \in S$, we have 

    \begin{align*}
        \frac{\D g^{(1)}}{\D x}(x) &= \frac{f'(x)}{x \big(1 + f(x)\big)} - \frac{\log \big(1 + f(x)\big)}{x^2} \\
        &= \frac{f'(x)}{x \big(1 + f(x)\big)}\frac{f(x)}{f(x)} - \frac{\log \big(1 + f(x)\big)}{x^2} \\
        &< \frac{\log \big(f(x)\big)}{x^2}\frac{f(x)}{\big(1 + f(x)\big)}- \frac{\log \big(1 + f(x)\big)}{x^2} \\
        &\leq 0
    \end{align*}

    where the final inequality uses $f(x) \leq 1 + f(x)$ and $\log\big(f(x)\big) \leq \log\big(1 + f(x)\big)$. This completes the proof.
\end{proof}

\begin{lemma}\label{lem_logs}
    Let $x,y > e$.
    \begin{enumerate}
        \item[(i)] If $x/\log(x) \leq y/2$, then we have $x \leq y \log(y)$.
        \item[(ii)] If $x/\log(x) \geq y$, then we have $x \geq y \log(y)$.
    \end{enumerate}
\end{lemma}
\begin{proof}
\begin{enumerate}
    \item[(i)] We prove the statement by contrapositive. Suppose that $x > y \log (y)$. Then since $t \mapsto \frac{t}{\log(t)}$ is increasing over $(e, \infty)$, we have
    \begin{align*}
        \frac{x}{\log(x)} \geq \frac{y \log(y)}{\log(y \log (y))} = \frac{y}{ 1 +\frac{\log \log(y)}{\log(y)}} > \frac{y}{2}.
    \end{align*}
    \item[(ii)] Again, we prove the statement by contrapositive. 
    Suppose that $x /\log(x) < y.$ Then since $t \mapsto \frac{t}{\log(t)}$ is increasing over $(e, \infty)$, we have
     \begin{align*}
        \frac{x}{\log(x)} < \frac{y \log(y)}{\log(y \log (y))} = \frac{y}{ 1 +\frac{\log \log(y)}{\log(y)}} < y.
    \end{align*}
    This completes the proof.
\end{enumerate}

\end{proof}

\begin{lemma}\label{lem_logk_constants}
    The following two properties hold.
    \begin{enumerate}
        \item[(i)] For any $x \geq 0$ and $c \in [0,1]$, we have $\log(1+cx) \geq c\log(1+x)$.
        \item[(ii)] For any $x \geq 0$ and $C\geq 1$, we have $C\log(1+x) \geq \log(1+Cx)$.
    \end{enumerate}
\end{lemma}

\begin{proof}[Proof of Lemma~\ref{lem_logk_constants}]
\begin{enumerate}
    \item[(i)] By concavity of the logarithm, we have $\log(1+cx) = \log\left((1-c)\cdot 1 + c(1+x)\right) \geq (1-c) \log(1) + c\log(1+x) = c \log(1+x)$ for any $c \in [0,1]$. 
    \item[(ii)] We have seen above that $\log(1+cy) \geq c \log(1+y)$ for any $y \geq 0$ and $c \in [0,1]$. 
Applying this with $y = Cx$ and $c = 1/C$ yields the result. 
\end{enumerate}

\end{proof}

\begin{lemma}\label{lem_w_quadratic}
    Let $n\in \mathbb N$ and $p_0 \in (0,1)$ and $w$ denote the function defined in~\eqref{eq_def_W_appendix}. For any $x \in [0,n]$, the following properties hold.
    \begin{enumerate}
        \item We have $w(x) \leq \frac{(x-np_0)^2}{ \sigma^2}$ where $\sigma^2 = np_0(1-p_0)$
        \item There exists a universal constant $c_0>0$ such that, for any $x \in [np_0, np_0 + c_0\sigma^2]$, we have $w(x) \geq \frac{(x-np_0)^2}{8 \sigma^2}$.
        \item For any $x \in [c_0 \sigma^2, n]$, we have $w(x) \asymp \sigma \frac{x-np_0}{\sigma} \log\left(1+ \frac{x-np_0}{\sigma^2}\right).$
    \end{enumerate}
\end{lemma}

\begin{proof}[Proof of Lemma~\ref{lem_w_quadratic}]
\phantom{ }

\begin{enumerate}
    \item 
Using the inequality $\log(1+x) \leq x$ that holds for any $x \in [-1,\infty)$ with the convention $\log(0) = -\infty$, we have, for any $x \in [0,n]$
    \begin{align*}
        w(x) &= n(1-p_0) \, h_B\!\left(-\frac{x-np_0}{n(1-p_0)}\right) + np_0 \, h_B\!\left(\frac{x-np_0}{np_0}\right)\\[5pt]
        & = n(1-p_0) \, \left\{\left(1 - \frac{x-np_0}{n(1-p_0)}\right) \log\left(1-\frac{x-np_0}{n(1-p_0)}\right) + \frac{x-np_0}{n(1-p_0)}\right\} \\[5pt]
        & \quad + np_0 \, \left\{\left(1 + \frac{x-np_0}{np_0}\right) \log\left(1+\frac{x-np_0}{np_0}\right) - \frac{x-np_0}{np_0}\right\}\\[5pt]
        & \leq n(1-p_0) \, \left\{\left(1 - \frac{x-np_0}{n(1-p_0)}\right) \left(-\frac{x-np_0}{n(1-p_0)}\right) + \frac{x-np_0}{n(1-p_0)}\right\} \\[5pt]
        & \quad + np_0 \, \left\{\left(1 + \frac{x-np_0}{np_0}\right) \left(\frac{x-np_0}{np_0}\right) - \frac{x-np_0}{np_0}\right\}\\[5pt]
        & = \frac{(x-np_0)^2}{n(1-p_0)} + \frac{2(x-np_0)^2}{np_0}\\[10pt]
        & = \frac{(x-np_0)^2}{\sigma^2}.
    \end{align*}

    \item For any $x \in [np_0, np_0 + c\sigma^2]$, we have $\frac{x-np_0}{np_0} \lor \frac{x-np_0}{n(1-p_0)}\in [0, c]$.  Choosing $c$ small enough that $\log(1+t) \geq t - \frac{3t^2}{4}$ for any $t \in [-c, c]$, we obtain that, 
    \begin{align*}
        w(x) &= n(1-p_0) \, h_B\!\left(-\frac{x-np_0}{n(1-p_0)}\right) + np_0 \, h_B\!\left(\frac{x-np_0}{np_0}\right)\\[5pt]
        & = n(1-p_0) \, \left\{\left(1 - \frac{x-np_0}{n(1-p_0)}\right) \log\left(1-\frac{x-np_0}{n(1-p_0)}\right) + \frac{x-np_0}{n(1-p_0)}\right\} \\[5pt]
        & \quad + np_0 \, \left\{\left(1 + \frac{x-np_0}{np_0}\right) \log\left(1+\frac{x-np_0}{np_0}\right) - \frac{x-np_0}{np_0}\right\}\\[5pt]
        & \geq n(1-p_0) \, \left\{\left(1 - \frac{x-np_0}{n(1-p_0)}\right) \left(-\frac{x-np_0}{n(1-p_0)} -\frac{3}{4}\frac{(x-np_0)^2}{(n(1-p_0))^2}\right) + \frac{x-np_0}{n(1-p)}\right\} \\[5pt]
        & \quad + np_0 \, \left\{\left(1 + \frac{x-np_0}{np_0}\right) \left(\frac{x-np_0}{np_0} - \frac{3}{4}\frac{(x-np_0)^2}{(n(1-p_0))^2}\right) - \frac{x-np_0}{np_0} \right\}\\[5pt]
        & \geq n(1-p_0) \frac{1}{4}\frac{(x-np_0)^2}{(n(1-p_0))^2} + np_0\left[\frac{1}{4} \frac{(x-np_0)^2}{(np_0)^2} - \frac{3}{4}\frac{(x-np_0)^3}{(np_0)^3}\right]\\[5pt]
        & \geq \frac{1}{8} \frac{(x-np_0)^2}{\sigma^2}
    \end{align*}
    provided $c$ is small enough.

    \item Assume now that $x \in [np_0 + c_0\sigma^2, n]$. 
    We have 
    \begin{align*}
        n(1-p_0) h_B\left(-\frac{x-np_0}{n(1-p_0)}\right) &= n(1-p_0) \, \left\{\left(1 - \frac{x-np_0}{n(1-p_0)}\right) \log\left(1-\frac{x-np_0}{n(1-p_0)}\right) + \frac{x-np_0}{n(1-p_0)}\right\}\\
        & \in \big[0, x-np_0\big].
    \end{align*}
    Moreover, using the relation $h_B(y) \asymp y\log(1+y)$ that holds for any $y > 0$, we have
    \begin{align*}
        np_0 \, h_B \! \left(\frac{x-np_0}{np_0}\right) \asymp (x-np_0) \log\left(\frac{x}{np_0}\right) \geq \log\left(1+\frac{3c_0}{4}\right)(x-np_0).
    \end{align*}
    Therefore, we obtain
    \begin{align*}
        w(x) &\asymp np_0 \, h_B \! \left(\frac{x-np_0}{np_0}\right) \asymp (x-np_0) \log\left(\frac{x}{np_0}\right)\\
        & \asymp \sigma \frac{x-np_0}{\sigma} \log\left(1+ \frac{x-np_0}{\sigma^2}\right).
    \end{align*}
    \end{enumerate}
\end{proof}

The following corollary follows by combining the properties from Lemma~\ref{lem_w_quadratic}.

\begin{corollary}\label{cor_w_growth_w}
    For any $x \in [np_0,n]$, it holds that $w(x) \asymp \sigma z \log(1+z/\sigma)$ where $z = \frac{x-np_0}{\sigma}$ and $\sigma^2 = np_0(1-p_0)$.
\end{corollary}

\end{document}